\documentclass[a4paper,10pt
]{article}%

\usepackage[dvipsnames]{xcolor}% adds colors
\usepackage{xr-hyper} % for \externaldocument
\usepackage[pagebackref, colorlinks, citecolor=PineGreen, linkcolor=PineGreen]{hyperref}
\hypersetup{
  final,
  pdftitle={Equivariant dendroidal sets and simplicial operads},
  pdfauthor={Bonventre, P. and Pereira, L. A.},
  linktoc=page,
  pdfpagemode=UseNone,
  pdfpagelayout=SinglePage,
}

\usepackage{amsmath, amsthm}% {amsfonts, amssymb}

\usepackage[latin1]{inputenc}%
\usepackage{MnSymbol}

\usepackage[normalem]{ulem} % underlining

\DeclareMathAlphabet\mathbb{U}{msb}{m}{n}
\usepackage{upgreek}
\usepackage{mathrsfs}

\usepackage[inline,shortlabels]{enumitem}%
\setenumerate{label=(\roman*)}

\usepackage[final]{microtype}
\usepackage{relsize}
\usepackage{geometry}

\usepackage{tikz}%
\usetikzlibrary{matrix,arrows,decorations.pathmorphing,
cd,patterns,calc}
\tikzset{%
  dummy/.style    = {circle,draw,inner sep=0pt,minimum size=2mm}%
}%

\usetikzlibrary[decorations.pathreplacing]

\makeatletter

\def\@testdef #1#2#3{%
  \def\reserved@a{#3}\expandafter \ifx \csname #1@#2\endcsname
  \reserved@a  \else
  \typeout{^^Jlabel #2 changed:^^J%
    \meaning\reserved@a^^J%
    \expandafter\meaning\csname #1@#2\endcsname^^J}%
  \@tempswatrue \fi}

\makeatother

\numberwithin{equation}{section} 
\numberwithin{figure}{section}

\usepackage{mathtools}
\mathtoolsset{showonlyrefs,showmanualtags} % Only number equations which are referenced with eqref

\newtheorem{theorem}[equation]{Theorem}%
\newtheorem*{theorem*}{Theorem}%
\newtheorem{lemma}[equation]{Lemma}%
\newtheorem{proposition}[equation]{Proposition}%
\newtheorem{corollary}[equation]{Corollary}%
\newtheorem*{conjecture*}{Conjecture}%
\newtheorem{innercustomgeneric}{\customgenericname}
\providecommand{\customgenericname}{}
\newcommand{\newcustomtheorem}[2]{%
  \newenvironment{#1}[1]
  {%
   \renewcommand\customgenericname{#2}%
   \renewcommand\theinnercustomgeneric{##1}%
   \innercustomgeneric
  }
  {\endinnercustomgeneric}
}

\newcustomtheorem{customthm}{Theorem}
\newcustomtheorem{customcor}{Corollary}
\theoremstyle{definition} % bold name, plain text
\newtheorem{definition}[equation]{Definition}%
\newtheorem*{definition*}{Definition}%
\newtheorem{example}[equation]{Example}%
\newtheorem{remark}[equation]{Remark}%
\newtheorem{notation}[equation]{Notation}%
\newcommand{\set}[1]{\left\{#1\right\}}%
\newcommand{\sets}[2]{\left\{ #1 \;|\; #2\right\}}%
\newcommand{\longto}{\longrightarrow}%
\newcommand{\into}{\hookrightarrow}%

\usepackage{harpoon}
\newcommand{\vect}[1]{\text{\overrightharp{\ensuremath{#1}}}}

\newcommand{\Sym}{\ensuremath{\mathsf{Sym}}}%
\newcommand{\Set}{\ensuremath{\mathsf{Set}}}

\newcommand{\sSet}{\ensuremath{\mathsf{sSet}}}%

\newcommand{\sOp}{\ensuremath{\mathsf{sOp}}}%
\newcommand{\dSet}{\mathsf{dSet}}
\newcommand{\sdSet}{\mathsf{sdSet}}
\newcommand{\PreOp}{\mathsf{PreOp}}

\DeclareMathOperator{\colim}{colim}%
\DeclareMathOperator{\Lan}{Lan}%
\DeclareMathOperator{\Ho}{Ho}
\DeclareMathOperator{\Aut}{Aut}%

\newcommand{\F}{\ensuremath{\mathcal F}}

\renewcommand{\O}{\ensuremath{\mathcal O}}
\renewcommand{\P}{\ensuremath{\mathcal P}}

\title{Equivariant dendroidal sets and simplicial operads}

\author{Peter Bonventre, Lu\'is A. Pereira}%

\begin{document}

\maketitle

\begin{abstract}
	We establish a Quillen equivalence between the homotopy theories of 
	equivariant Segal operads and equivariant simplicial operads with norm maps.
	Together with previous work, 
	we further conclude that the homotopy coherent nerve
	is a right-Quillen equivalence from the 
	model category of
	equivariant simplicial operads with norm maps
	to the
	model category structure
	for equivariant-$\infty$-operads in equivariant dendroidal sets.
\end{abstract}

\tableofcontents

\section{Introduction}

This paper is the last in a series of five
(after \cite{Per18,BP20,BP_FCOP,BP_ACOP})
and concludes a project to establish a homotopy theoretical equivalence
between equivariant colored simplicial operads with norm maps
and equivariant-$\infty$-operads in equivariant dendroidal sets,
thus generalizing the analogous 
Cisinski-Moerdijk project 
\cite{CM13a,CM13b,CM11}
in the non-equivariant setting.

The key novelty (and difficulty) faced in the equivariant setting is that the homotopy theory of operads
needs to account for an 
extra piece of structure,
the so called \emph{norm maps},
which we now briefly recall
(for further discussion, 
see the introductions to any of
\cite{Per18},\cite{BP21},\cite{BP_FCOP}).

For simplicity,
let us focus on the category
$\mathsf{sOp}_{\**} = \mathsf{Op}_{\**}(\mathsf{sSet})$
of single colored simplicial operads.
Letting $G$ be a fixed finite group,
a $G$-equivariant (single colored simplicial) operad
is a $G$-object
$\O \in \mathsf{sOp}_{\**}^G$.
Crucially, 
note that the $n$-th level
$\O(n)$ then admits commuting actions
by $\Sigma_n$ and $G$ or, equivalently,
a $G \times \Sigma_n$ action.
One upshot of Blumberg and Hill's work in 
\cite{BH15}
is that the preferred notion of equivalence
in $\mathsf{sOp}_{\**}^G$
is that of \emph{graph equivalence},
by which we mean maps
$\mathcal{O} \to \mathcal{P}$
in $\mathsf{sOp}_{\**}^G$
such that the fixed point maps for \emph{graph subgroups}
\begin{equation}\label{NORMMAP EQ}
	\mathcal{O}(n)^{\Gamma}
	\xrightarrow{\sim}
	\mathcal{P}(n)^{\Gamma}
\qquad
	\text{for }
	\Gamma \leq G \times \Sigma_n
	\text{ such that }
	\Gamma \cap \Sigma_n = \**
\end{equation}
are Kan equivalences in $\mathsf{sSet}$.
The term graph subgroup is motivated by a description of such $\Gamma$:
they are necessarily of the form
$\Gamma = \{(h,\phi(H)) | h \in H\}$
for some subgroup $H \leq G$
and homomorphism $\phi \colon H \to \Sigma_n$.
Such fixed points $\mathcal{O}(n)^{\Gamma}$
are called 
\emph{spaces of norm maps}
since, for $X$ an $\O$-algebra,
the algebra multiplication maps on the left below
(see, e.g., \cite[\eqref{AC-ALGNORM EQ}]{BP_ACOP})
\[
	\O(n) \times X^{\times n} \to X
\qquad
	\O(n)^{\Gamma} \times N_{\Gamma} X \to X
\]
induce $H$-equivariant maps as on the right above, 
where $N_{\Gamma} X$ is a so called \emph{norm object},
which denotes $X^{\times n}$ together with the $H$-action induced by the identification $H \simeq \Gamma$.

The cornerstone of this project was the discovery by the authors of 
a category $\Omega_G$ of $G$-trees 
whose objects encode compositions of norm maps
in a $G$-operad $\O$
(a detailed discussion can be found in \cite[\S 4.3]{Per18}),
and which extends the Moerdijk-Weiss category $\Omega$
of trees (whose objects encode composition in an operad).
This category $\Omega_G$
then allowed us to build model structures
on the categories
$\mathsf{dSet}^G = \mathsf{Set}^{\Omega^{op} \times G}$
of \emph{equivariant dendroidal sets}
\cite[Thm. 2.1]{Per18}
and $\mathsf{sOp}^G$
of \emph{equivariant colored simplicial operads}
\cite[Thm. \ref{AC-THMA}]{BP_ACOP},
where in both cases the notion of weak equivalence
is determined by (a colored variant of) the norm map data as in 
\eqref{NORMMAP EQ}.
Our main result in this paper is then the following,
generalizing \cite[Thm. 8.15]{CM13b}.
\begin{customthm}{I}\label{QE THM}
	There is a Quillen equivalence
	\begin{equation}
	\label{QE_EQ}
	W_! \colon \dSet^G \rightleftarrows \sOp^G \colon hcN
	\end{equation}
	between equivariant dendroidal sets and
	equivariant simplicial operads with norm maps.
\end{customthm}
Here the right adjoint $hcN$
is a variant of the nerve functor accounting for homotopy information,
called the \emph{homotopy coherent nerve},
while the left adjoint $W_!$
is a ``fattened operadification'' 
which is related to the Boardman-Vogt resolution of operads
(see \cite[\S 4]{CM13b} for details).

Writing $\eta$ for both the terminal category and its nerve,
one has natural identifications
of slice categories
$\mathsf{dSet}^G_{/\eta} \simeq 
\mathsf{sSet}^G$
and
$\mathsf{sOp}^G_{/\eta} \simeq 
\mathsf{sCat}^G$,
so that by slicing Theorem \ref{QE THM}
one recovers 
Bergner's result \cite{Ber17}
in the $\infty$-category context.

However, it is worth nothing that
the presence of norm maps make 
our operadic result far more subtle
than the categorical analogue.
Indeed, the model structures on 
$\mathsf{sSet}^G$, $\mathsf{sCat}^G$
used in \cite{Ber17}
are built formally from the non-equivariant model structures
on $\mathsf{sSet}$, $\mathsf{sCat}$
by using the abstract framework in \cite{Ste16}.
On the other hand, applying that framework to the operadic context results in model structures
on 
$\mathsf{dSet}^G$
and
$\mathsf{sOp}^G$
that only account for the \emph{trivial norm maps}
(i.e. those for which $\Gamma \leq G$ in \eqref{NORMMAP EQ}).

The (conclusion to the) proof of Theorem \ref{QE THM}
can be found at the end of \S \ref{PFMNTHM SEC}.
However, this proof requires some background, which we now recall.
Just as in \cite{CM13b},
we make use of two additional categories,
the category 
$\mathsf{sdSet}^G
= \mathsf{Set}^{\Omega^G\times \Delta^{op} \times G}$
of \emph{equivariant dendroidal simplicial sets}
and its 
subcategory
$\mathsf{PreOp}^G$
of \emph{equivariant preoperads}, which fit into a diagram
\begin{equation}\label{ADJSQ EQ}
	\begin{tikzcd}
	\mathsf{PreOp}^G \ar{d}[swap]{\gamma^{\**}}&
	\mathsf{sOp}^G \ar{l}[swap]{N} \ar{d}{hcN}
\\
	\mathsf{sdSet}^G &
	\mathsf{dSet}^G \ar{l}{c_{!}}
	\end{tikzcd}
\end{equation}
where $N$ is the \emph{nerve functor}
and $\gamma^{\**}$, $c_!$ are the natural inclusions.

The model structures on the categories featured in 
\eqref{ADJSQ EQ}
were built in previous work.
More specifically,
\cite[Thm. 2.1]{Per18} provides the model structure on
$\mathsf{dSet}^G$,
\cite[Thm. I]{BP_ACOP} provides the model structure on
$\mathsf{sOp}^G$,
\cite[Def. 4.22]{BP20} gives the model structure on 
$\mathsf{sdSet}^G$,
and \cite[Thm. 4.39]{BP20} provides the model structure on
$\mathsf{PreOp}^G$.
These model structures generalize those in the work of Cisinski-Moerdijk
in the non-equivariant operadic context,
which in turn generalize corresponding model structures in the categorical context.
The following table, first appearing in \cite{BP20}, summarizes the relevant model structures, along with the nomenclature for the fibrant objects.
\begin{table}[htbp]
	\label{TABLE}
	\centering
	\resizebox{\columnwidth}{!}{%
		\begin{tabular}{|c|c|c|}
			\hline
			``categories up to htpy'' & ``operads up to htpy'' & ``equivariant operads up to htpy''
			\\ \hline
			simplicial sets $\sSet$ & dendroidal sets $\dSet$ & equivariant dendroidal sets $\dSet^G$
			\\
			Joyal model structure & model str. from \cite{CM11} & model structure from \cite{Per18}
			\\
			$\infty$-categories & $\infty$-operads & $G$-$\infty$-operads
			\\ \hline
			bisimplicial sets $\mathsf{ssSet}$ & simp. dend. sets $\mathsf{sdSet}$ & equiv. simp. dend. sets $\mathsf{sdSet}^G$
			\\
			Rezk model structure & model str. from \cite{CM13a} & model structure from \cite{BP20}
			\\
			complete Segal spaces & complete dend. Segal spaces & complete equiv. dend. Segal spaces
			\\ \hline
			Segal precategories $\mathsf{SeCat}$ & Segal preoperads $\mathsf{PreOp}$ & equiv. Segal preoperads $\mathsf{PreOp}^G$
			\\
			Hirschowitz-Simpson & model str. from \cite{CM13a} & model structure from \cite{BP20}
			\\
			Segal categories & Segal operads & equiv. Segal operads
			\\ \hline
			simplicial categories $\mathsf{sCat}$ & simplicial operads $\sOp$ & equiv. simplicial operads $\sOp^G$
			\\
			Bergner model structure & model str. from \cite{CM13b} & model structure from \cite{BP_ACOP}
			\\ \hline
		\end{tabular}
	}
	\caption{A summary of models for $\infty$-categories, $\infty$-operads, and $G$-$\infty$-operads.}
\end{table}

Considering now the functors in 
\eqref{ADJSQ EQ},
we have previously established that 
$c_!$ and $\gamma^{\**}$
are both left-Quillen equivalences
\cite[Thms. 4.30 and 4.41]{BP20}.

The proof strategy for establishing 
Theorem \ref{QE THM}
can then be summarized as follows.

First, the $(W_!,hcN)$ adjunction
is shown to be Quillen (Proposition \ref{W!_LEFTQ_PROP}).

Second, the square 
\eqref{ADJSQ EQ}
is shown to commute at the level of homotopy categories 
(this is 
%what is actually
shown at the end of 
\S \ref{PFMNTHM SEC}, 
by establishing the zigzag of weak equivalences in 
\eqref{BIGZIG EQ}).

Third and last, it thus suffices to show that the top horizontal 
functor $N$ in \eqref{ADJSQ EQ}
induces an equivalence of homotopy categories.
As in \cite{CM13b},
this last step requires some care.
The functor $N$ preserves all weak equivalences
(cf. the proof of Theorem \ref{PREQUIEQUIV THM})
and is thus already a derived functor,
but it is not quite right Quillen 
due to $\mathsf{PreOp}^G$ not having enough fibrant objects 
or, dually, having too many cofibrant objects.
To address this, we show in \S \ref{TAMEDEFEX SEC}
that $\mathsf{PreOp}^G$
admits an alternative model structure, 
called the \emph{tame model structure} 
(Theorem \ref{TAMEMS_THM}),
with the same weak equivalences but
less cofibrant objects.
Using this alternative model structure,
it can then be shown 
(Theorem \ref{PREQUIEQUIV THM})
that $N$ becomes a right-Quillen equivalence,
concluding the argument.

\subsection{Outline}

First, in \S 2 we mostly recall 
some notions from previous work that are used throughout.
Namely, 
\S \ref{FORESTS_SEC} and \S \ref{GTREES SEC}
recall the necessary properties of the categories 
$\Omega$ of trees and $\Omega_G$ of $G$-trees,
while \S \ref{EDS_SEC}
recalls the model structure on the category 
$\mathsf{dSet}^G$
of equivariant dendroidal sets.

The overall goal of \S \ref{TAME_SEC}
is to establish the existence of the alternative
\emph{tame model structure} on $\mathsf{PreOp}^G$.
The content of \S \ref{JT_SEC} and \S \ref{SPREOP_SEC}
is again expository in nature, 
recalling the model structures on 
$\mathsf{sdSet}^G$, 
$\mathsf{PreOp}^G$.
In \S \ref{FIBTENS_SEC}
we introduce a somewhat novel construction,
called the \emph{fibered tensor product}
$\otimes_{\mathfrak{C}_{\bullet}}$,
which is used in 
\S \ref{TAMEDEFEX SEC}
to describe and build the tame model structure on 
$\mathsf{PreOp}^G$.
The use of $\otimes_{\mathfrak{C}_{\bullet}}$
is motivated by the observation that the
model structure on $\mathsf{sOp}^G$
can be described using an analogous tensor product,
thus simplifying the task of showing that  
$\tau \colon \mathsf{PreOp}^G_{tame}
\rightleftarrows
\mathsf{sOp}^G \colon N$
is a Quillen adjunction
(cf. Theorem \ref{PREQUIEQUIV THM}).

The goal of \S \ref{QE_SEC} is to prove 
Theorem \ref{QE THM}
(up to Lemma \ref{KEYPRVAR LEM},
whose proof is postponed to \S \ref{KEYRES SEC}).
First,
\S \ref{GSOP_SEC} recalls the model structure on
$\mathsf{sOp}^G$.
Then, in 
\S \ref{PREOPOPEQUIV SEC}
we establish Theorem \ref{PREQUIEQUIV THM},
showing that the top functor $N$ in 
\eqref{ADJSQ EQ}
induces an equivalence of homotopy categories.
Lastly, \S \ref{PFMNTHM SEC}
concludes the proof of Theorem \ref{QE THM}
by showing that
\eqref{ADJSQ EQ}
commutes in a homotopical sense.

Our last main section \S \ref{KEYRES SEC}
is dedicated to the rather technical proof of 
Lemma \ref{KEYPRVAR LEM},
which examines the homotopical properties
of certain pushouts in 
$\mathsf{Op}^G$
after applying the nerve functor
$N \colon \mathsf{Op}^G \to \mathsf{dSet}^G$,
and is at the core of the proof of 
Theorem \ref{PREQUIEQUIV THM}
and thus also of Theorem \ref{QE THM}.

Lastly, in Appendix \ref{HGEO AP}
we give an explicit description of the discretization of
a $G$-$\infty$-operad $X \in \mathsf{dSet}^G$,
adapting the similar non-equivariant description
in \cite[\S 6]{MW09}.
This then allows us to
show that,
for a fibrant operad
$\O \in \mathsf{sOp}^G$,
the natural discretizations of
$hcN(\O) \in \mathsf{dSet}^G$
and
$N(\O) \in \mathsf{sdSet}^G$
coincide (Proposition \ref{HOOPID_PROP}),
thus generalizing the non-equivariant analogue result
\cite[Prop. 4.8]{CM13b}.
Just as in the non-equivariant story,
Proposition \ref{HOOPID_PROP} is closely related
to the proof of 
Proposition \ref{W!_LEFTQ_PROP},
which shows that 
$W_! 
\colon 
\mathsf{dSet} 
\rightleftarrows 
\mathsf{Op}
\colon 
hcN$
is a Quillen adjunction,
though Proposition \ref{W!_LEFTQ_PROP}
does not require the full strength of 
Proposition \ref{HOOPID_PROP}.
A detailed discussion can be found in 
Remark \ref{TWOHOMOP REM}.

\begin{remark}
        \label{MSLIST_REM}
        This paper utilizes and compares several model structures on related categories.
        We list them below, along with internal references for their definitions.
        \begin{itemize}
        \item model structure on $\mathsf{dSet}^G$, Theorem \ref{DSETGMOD THM}.
        \item joint model structure on $\mathsf{sdSet}^G$, Theorem \ref{JB_THM}.
        \item normal model structure on $\mathsf{PreOp}^G$, Theorems \ref{PREOPMS THM} and \ref{FIBPREOP THM}.
        \item tame model structure on $\mathsf{PreOp}^G$, Theorem \ref{TAMEMS_THM}.
        \item model structure on $\mathsf{sOp}^G$, Theorem \ref{SOPG_THM}.
        \end{itemize}
\end{remark}

\section{Equivariant trees and dendroidal sets}
\label{EQTRDS SEC}

In this mostly expository section, 
we recall the categories of trees,
as well as the associated presheaf categories,
which will be needed throughout the paper.
A more detailed discussion can be found in 
\cite[\S 5,\S 6]{Per18}, \cite[\S 2]{BP20}.

\subsection{Trees and forests}
\label{FORESTS_SEC}

We start by recalling the Moerdijk-Weiss category $\Omega$ of trees
\cite{MW07}.
First, each object of $\Omega$ can be encoded by 
a (rooted) tree diagram $T$ as below.
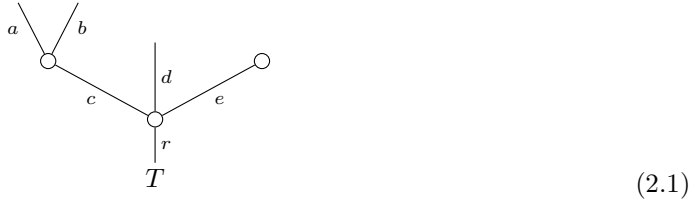
\begin{equation}\label{eq:TREE}
	\begin{tikzpicture}[auto,grow=up, level distance = 2.2em,
	every node/.style={font=\scriptsize,inner sep = 2pt}]%
	\tikzstyle{level 2}=[sibling distance=4em]%
	\tikzstyle{level 3}=[sibling distance=2.25em]%
            \node [font=\normalsize] {$T$}
            child{node [dummy] {}
              child{node [dummy] {}
                edge from parent node [swap] {$e$}
              }
              child[level distance = 2.9
              em]{edge from parent node [swap] {$d$}}
              child{node [dummy] {}
                child{edge from parent node [near end, swap] {$b$}}
                child{edge from parent node [near end] {$\phantom{b}a$}}
                edge from parent node {$c$}
              }
              edge from parent node [swap] {$r$}
            };        
      \end{tikzpicture}
\end{equation}
Edges with no vertices $\circ$ above them are called \textit{leaves}, the unique bottom edge is called the \textit{root},
and edges that are neither are called \textit{inner edges}.
In the example above, $a$, $b$ and $d$ are leaves, $r$ is the root, and $c$ and $e$ are inner edges.
The sets of edges, inner edges, and vertices of a tree $T$ are denoted 
$\boldsymbol{E}(T)$, 
$\boldsymbol{E}^{\mathsf{i}}(T)$, 
and $\boldsymbol{V}(T)$, respectively.

Describing the maps in $\Omega$ requires some care.
To do so, we recall the algebraic notion of
a \emph{broad poset},
originally due to Weiss \cite{Wei12}
and further developed in \cite{Per18}.
For each edge $t$ in a tree topped by a vertex $\circ$, we write
$t^{\uparrow}$
for the tuple of edges immediately above $t$.
In \eqref{eq:TREE} one has  
$r^{\uparrow} = cde$, 
$c^\uparrow = ab$, 
and $e^\uparrow = \epsilon$,
where $\epsilon$ denotes the empty tuple.
We then encode each vertex symbolically as
$t^{\uparrow} \leq t$,
which we call a 
\emph{generating broad relation}.
This notation is motivated by a form of transitivity.
For example,
in \eqref{eq:TREE}
the relations
$cde \leq r$ and $ab \leq c$
generate, under \emph{broad transitivity},
the relation $abde \leq r$,
and one may similarly obtain relations
$cd \leq r$ and $abd \leq r$.
These relations, together with identity relations $t \leq t$,
then form the \emph{broad poset associated with $T$}
(alternatively, this broad poset data
is essentially equivalent to the data of
the colored operad $\Omega(T)$ associated to $T$, 
cf. \cite[\S 3]{MW07},
\cite[Rem. 4.4]{Per18}
or \eqref{TAUNER EQ}).

A map of trees $\varphi \colon S \to T$
in $\Omega$ is then an underlying map
of edge sets 
$\varphi \colon \boldsymbol{E}(S) \to \boldsymbol{E}(T)$
which preserves broad relations.

If an edge $t$ is pictorially above (or equal to) an edge $s$, we write $t \leq_d s$.
Equivalently, $t \leq_d s$ if there exists a broad relation $s_1\dots s_n \leq s$ such that $t = s_i$ for some $i$.
\\

Moreover,
our discussion will be simplified by assuming 
that $\Omega$ has exactly one representative of 
each \emph{planarized tree},
by which we mean a tree together with a planar representation as in \eqref{eq:TREE}
(alternatively, 
planarizations can be formalized as
suitable extensions of $\leq_d$ to a total order
\cite[\S 3.1]{BP21}).
Importantly, this implies that each map  
$\varphi \colon S \to T$ in $\Omega$
has a strictly\footnote{That is, this factorization is not simply unique up to unique isomorphism.} unique factorization
$S \xrightarrow{\simeq} S' \to T$
as an isomorphism followed by a \emph{planar map}
\cite[Prop. 3.24]{BP21}.
Informally, $S'$ is obtained by giving $S$ the planarization ``pulled back'' from $T$.
Note that, 
in particular, the subcategory of planar maps is skeletal,
i.e. the only planar isomorphisms are the identities.

\begin{notation}
	We write $\eta$ for the \textit{stick} tree, the unique tree with a single edge and no vertices.
\end{notation}

\begin{example}\label{TREEMAP_EX}
	The edge labels in each tree $S_i$ below determine maps
	$\boldsymbol{E}(S_i) \to \boldsymbol{E}(T)$,
	where $T$ is as in \eqref{eq:TREE}.
	For $i \leq 4$ this encodes maps
	$S_i \to T$ in $\Omega$,
	but not for $i=5$.
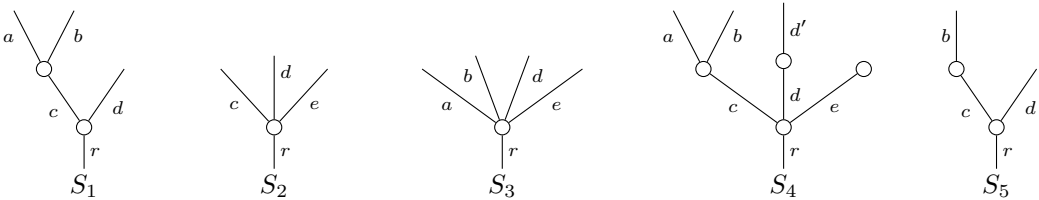
\begin{equation}
\begin{tikzpicture}[auto, grow=up, level distance = 2.2em,
	every node/.style={font=\scriptsize,inner sep = 2pt}]
\tikzstyle{level 2}=[sibling distance=3em]%
\tikzstyle{level 3}=[sibling distance=2.25em]%
	\node at (0,0) [font=\normalsize] {$S_1$} % 
                  child{node [dummy] {}
                    child{edge from parent node [swap] {$d\phantom{c}$}}
                    child{node [dummy] {}
                      child{edge from parent node [near end, swap] {$b$}}
                      child{edge from parent node [near end] {$\phantom{b}a$}}
                      edge from parent node {$\phantom{d}c$}
                    }
                    edge from parent node [swap] {$r$}
                  };
\tikzstyle{level 2}=[sibling distance=2em]%
	\node at (2.5,0) [font=\normalsize] {$S_2$} %
                  child{node [dummy] {}
                    child{edge from parent node [swap] {$e$}}
                    child [level distance = 2.7em] {edge from parent node [swap,near end] {$d$}}
                    child{edge from parent node {$c$}}
                    edge from parent node [swap] {$r$}
                  };
	\node at (5.5,0) [font=\normalsize] {$S_3$} 
                  child{node [dummy] {}
                    child{
                      edge from parent node [swap] {$e\phantom{a}$}
                    }
                    child [level distance = 2.7em]{edge from parent node [swap, very near end] {$d$}}
                    child [level distance = 2.7em]{edge from parent node [very near end] {$b$}}
                    child{edge from parent node {$\phantom{e}a$}}
                    edge from parent node [swap] {$r$}
                  };
\tikzstyle{level 2}=[sibling distance=3em]%
	\node at (9.2,0) [font=\normalsize] {$S_4$} 
                  child{node [dummy] {}
                    child{node [dummy] {}
                      edge from parent node [swap] {$e$}
                    }
                    child[level distance = 2.5em]{node [dummy] {}
                      child[level distance = 2.2em]{edge from parent node [swap] {$d'$}}
                      edge from parent node [swap] {$d$}
                    }
                    child{node [dummy] {}
                      child{edge from parent node [swap,near end] {$b$}}
                      child{edge from parent node [near end] {$\phantom{b}a$}}
                      edge from parent node {$c$}
                    }
                    edge from parent node [swap] {$r$}
                  };                    
	\node at (12,0) [font=\normalsize] {$S_5$}
                  child{node [dummy] {}
                    child{edge from parent node [swap] {$d\phantom{c}$}}
                    child{node [dummy] {}
                      child{edge from parent node {$b$}}
                      edge from parent node {$\phantom{d}c$}
                    }
                    edge from parent node [swap] {$r$}
                  };
\end{tikzpicture}
\end{equation}
\end{example}

A map of trees $\varphi \colon S \to T$ is called:
\begin{itemize}
\item a \textit{tall map} if
      $\varphi(\underline{l}_S) = \underline{l}_T$ and $\varphi(r_S) = r_T$,
      with $\underline{l}_{(-)}$ and $r_{(-)}$ denoting the tuple of leaf edges and the root edge;
\item a \textit{face map} if it is injective on edges;
      an \textit{inner face} if it is also tall; and
      an \textit{outer face} if, for any factorization
      $\varphi \simeq \varphi_1\varphi_2$
      with $\varphi_1,\varphi_2$ face maps
      and $\varphi_2$ inner, 
      $\varphi_2$ is an isomorphism;
%       it does not
%      admit a factorization as a non-isomorphism inner face followed by a face map;
\item a \textit{degeneracy} if it is surjective on edges and preserves leaves
      (and is thus tall).
\end{itemize}

Pictorially, inner face maps 
$S \to T$ remove some edges in $T$
(and merge the vertices adjacent to those edges),
outer face maps remove some vertices of $T$,
and degeneracies collapse some of the unary vertices of $S$.

\begin{example}
	In Example \ref{TREEMAP_EX},
	$S_1 \to T$ is an inner face, 
	$S_2 \to T$ is an outer face,
	$S_3 \to T$ is a face that is neither inner nor outer,
	and $S_4 \to T$ is a degeneracy.
\end{example}

\begin{notation}\label{MAPLABELS_NOT}
	In the remainder of \S \ref{EQTRDS SEC}
      we will label a 
      map in $\Omega$
      by the letters d/i/o/t/f/p
      to indicate that the map is
      a degeneracy/inner face/outer face/tall/face/planar.
\end{notation}

\begin{proposition}[{\cite[Prop. 2.2]{BP20}}]
        \label{NP_TREEFACT_PROP}
	A map of trees $\varphi \colon S \to T$ has a factorization, unique up to unique isomorphisms,
        \begin{equation}\label{NP_TREEFACT_EQ}
              S \xrightarrow{d} 
              S' \xrightarrow{i} 
              S'' \xrightarrow{o}
              T
        \end{equation}
        as a degeneracy followed by an inner face followed by an outer face.
\end{proposition}

\begin{remark}\label{TODF REM}
	A map $\varphi \colon S \to T$
	is tall (resp. a face)
	iff in the decomposition \eqref{NP_TREEFACT_EQ}
	the component labeled $o$ (resp. $d$)
	is an isomorphism.
	As such, by combining the 
	first two (resp. last two)
	maps in \eqref{NP_TREEFACT_EQ}
	one recovers the 
	``tall-outer face'' 
	(resp. ``degeneracy-face'')
	factorization of the map $\varphi$
	\cite[Prop. 3.36]{BP21}, \cite[Prop. 5.37]{Per18}.
\end{remark}

\begin{remark}\label{CNVXM REM}
	Following the previous remark, 
	it is natural to consider the class of maps
	$\varphi \colon S \to T$
	such that the inner face factor
	in \eqref{NP_TREEFACT_EQ} is an isomorphism.
	We call these maps \textit{convex},
	since they are readily seen to be 
	characterized by the following property:
	if $e <_d e' <_d e''$ in $T$ 
	and $e,e''$ are in the image of $\varphi$
	then so is $e'$.
	Notably, it follows from this characterization that convex maps are also closed under composition.
	
	Equivalently, $\varphi$ is convex precisely if
	the ``tall-outer face'' and ``degeneracy-face''
	factorizations coincide.
	In particular, outer faces are characterized
	as the convex faces.
\end{remark}

When accounting for planar structures,
one has the following refinement of 
Proposition \ref{NP_TREEFACT_PROP}.

\begin{proposition}[{cf. \cite[Prop. 2.2]{BP20}}]
      \label{TREEFACT_PROP}
      A map of trees $\varphi \colon S \to T$ has a strictly unique factorization
\begin{equation}\label{TREEFACT_EQ}
	S \xrightarrow{\simeq} 
	S_p \xrightarrow{pd} 
	\varphi S \xrightarrow{pi} 
	\overline{\varphi S} \xrightarrow{po} T
\end{equation}
	as an isomorphism followed by a planar degeneracy, a planar inner face, and a planar outer face.
\end{proposition}

\begin{remark}\label{TREEFACT_REM}
      The notation $\varphi S$ is motivated by the fact that this tree has edge set
      $\boldsymbol{E}(\varphi S) = \varphi (\boldsymbol{E}(S))$,
      while the 
      notation $\overline{\varphi S}$ is an instance of the 
      \emph{outer closure of an inner face}
      notation in \cite[Not. 2.14]{BP20}.
\end{remark}

\begin{remark}
        \label{TREEMAPCOMP_REM}
	Generalizing Remarks \ref{TODF REM} and \ref{CNVXM REM},
	one has that, for any subset 
	$\mathcal{S} \subseteq \{\simeq,pd,pi,po\}$
	of the arrow labels 
	in \eqref{TREEFACT_EQ},
	the type of maps whose
	factors labeled by $\mathcal{S}$ are identities 
	is closed under composition.

	For example, the maps such that 
	the factors labeled $\simeq$ and $pi$
	are identities are the planar convex maps,
	while those maps such that
	the factors labeled $pd$ and $po$ 
	are identities are the 
	(possibly not planar) inner face maps.
	Both of these kinds of maps are closed under composition.	
\end{remark}

A \textit{corolla} is a tree with a single vertex.
We note that the subcategory of $\Omega$ spanned by corollas and isomorphisms is naturally identified with to the category $\Sigma$ of standard finite ordered sets
$\{1,2,\cdots,n\}$ and (non-ordered) isomorphisms.

%\begin{notation}
%      \label{LR_NOT}
%      For each $T \in \Omega$, there exists a unique corolla $\mathsf{lr}(T) \in \Sigma$ equipped with a planar tall map $\mathsf{lr}(T) \to T$,
%      which we call the \textit{leaf-root} of $T$.
%\end{notation}

\begin{notation}
        \label{TV_NOT}
        For $T \in \Omega$ and $v \in \boldsymbol{V}(T)$, 
        we write $T_v \hookrightarrow T$
        for the subcorolla whose vertex is $v$.
\end{notation}

Next, we recall the categories of (colored) forests used in 
\cite[Def. \ref{OC-COLFOR DEF}]{BP_FCOP}.

\begin{definition}\label{FOR DEF}
      The category $\Phi$ of \textit{forests} is the coproduct completion of the category $\Omega$ of trees:
      objects are formal coproducts 
      $F = \amalg_{i \in I} F_i$ with $F_i \in \Omega$,
      and an arrow 
      $\varphi \colon \amalg_{i \in I} F_i \to
      \amalg_{j \in J} F'_j$ is given by
      a map of indexing sets $\varphi \colon I \to J$ and
      maps $\varphi_i \colon F_i \to F'_{\varphi(i)}$ in $\Omega$ for each $i \in I$.

The sets of \emph{edges}, \emph{inner edges}, \emph{vertices}
of a forest $F = \amalg_i F_i$
are defined in the natural way as
\[
	\boldsymbol{E}(F)
	\simeq
	\amalg_i \boldsymbol{E}(F_i),
\qquad
	\boldsymbol{E}^{\mathsf{i}}(F)
	\simeq
	\amalg_i \boldsymbol{E}^{\mathsf{i}}(F_i),
\qquad
	\boldsymbol{V}(F)
	\simeq
	\amalg_i \boldsymbol{V}(F_i).
\]
\end{definition}

\begin{remark}
        \label{PLANARFOR_REM}
        As with trees $T \in \Omega$, 
        we assume that each forest $F = \amalg_{i \in I} F_i \in \Phi$
        is planarized \cite[Def. 3.2 and Rem. 3.15]{BP21}, 
        which is equivalent to choosing a total order of the indexing set $I$ and a planarization of each $F_i$.
        Moreover, we similarly assume that $\Phi$ contains exactly one representative of each planarization,
        so that the only planar isomorphisms are again the identities.
        
        However, we caution that, in order to pullback a planarization
        along $\varphi \colon F \to \tilde{F}$,
        one needs to assume  
        $\varphi$ sends roots of $F$
        to $\leq_d$-incomparable edges of $\tilde{F}$, 
        i.e. that $\varphi$ is an \emph{independent map}
        \cite[Def. 5.28]{Per18}.
        As such, in the context of forests our definition of 
        planar map requires that the map is independent
        \cite[Prop. 3.19]{BP21}.
        In particular, the factorization
        $F \xrightarrow{\simeq} F' \to \tilde{F}$
        of a map $\varphi$ as an isomorphism followed by a planar map
        exists only if $\varphi$ is independent
        \cite[Prop. 3.24]{BP21}.
\end{remark}

\begin{definition}\label{CFOREST_DEF}
      Let $\mathfrak C$ be a set of colors.
      The category $\Phi_{\mathfrak C}$ of \textit{$\mathfrak C$-forests} has:
      \begin{itemize}
      \item objects pairs $\vect F = (F, \mathfrak c)$ with
            $F \in \Phi$ a forest and
            $\mathfrak c \colon \boldsymbol{E}(F) \to \mathfrak C$ a coloring of its edges;
      \item arrows $\vect F = (F, \mathfrak c) \to (F', \mathfrak c') = \vect{F'}$ maps
            $\varphi \colon F \to F'$ in $\Phi$ such that $\mathfrak c = \mathfrak c' \varphi$.
      \end{itemize}      
	Lastly, we write
	$\Omega_{\mathfrak{C}} \subset \Phi_{\mathfrak{C}}$,
	which we call the category of 
	\emph{$\mathfrak{C}$-trees},
	for the full subcategory spanned by the objects 
	whose underlying forest is a tree,
	and 
	$\Sigma_{\mathfrak{C}} \subset \Omega_{\mathfrak{C}}$,
	which we call the category of 
	\emph{$\mathfrak{C}$-corollas},
	for the further subcategory of objects whose underlying tree is a corolla
	and whose maps are isomorphisms.
	Note that a change of color map
	$f \colon \mathfrak{C} \to \mathfrak{D}$
	induces a map
	$f \colon \Phi_\mathfrak{C} \to \Phi_\mathfrak{D}$ via $\vect{C} = (\mathfrak{c}_i)_{0\leq i \leq n}
	\mapsto
	(f\mathfrak{c}_i)_{0\leq i \leq n}= f\vect{C}$.
\end{definition}

\subsection{Equivariant trees}
\label{GTREES SEC}

We next recall the category $\Omega_G$ of equivariant trees, which encodes the combinatorics of compositions of norm maps.
A thorough discussion can be found in \cite[\S 5]{Per18} or \cite[\S 2]{BP20}.

We begin with an example.
Let $G =  
\langle \rho, \sigma | \rho^4=1, \sigma^2 =1, 
\sigma \rho = \rho^3 \sigma \rangle =
\set{1,\rho,\rho^2,\rho^3, \sigma, \rho\sigma, \rho^2\sigma, \rho^3\sigma}$
be the dihedral group with 8 elements,
and $L \leq K \leq H \leq G$ denote the subgroups
$H = \langle \rho^2,\sigma \rangle$, $K = \langle \rho^2 \rangle$, $L = \{1\}$.
There is then a $G$-tree $T \in \Omega_G$ with 
\textit{expanded representation} given by the two trees on the left below,
and \textit{orbital representation} given by the 
single tree on the right.
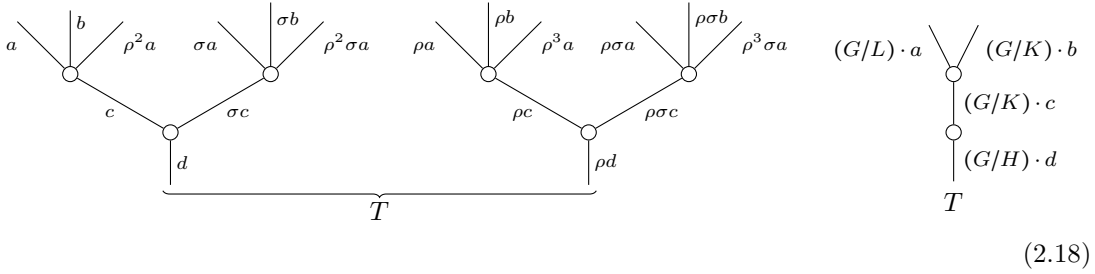
\begin{equation}\label{GTREE_EQ}
\begin{tikzpicture}[auto,grow=up, level distance = 2.2em,
	every node/.style={font=\scriptsize,inner sep = 2pt}]%
\tikzstyle{level 2}=[sibling distance=7.5em]%
\tikzstyle{level 3}=[sibling distance=2.2em]%
	\node at (0,0){}%	
	child{node [dummy] {}%
		child{node [dummy] {}%
			child{node {}%
			edge from parent node [swap,very near end] {$\rho^2\sigma a$}}%
			child[level distance = 2.7em]{
			edge from parent node [swap,near end] {$\sigma b$}}%
			child{node {}%
			edge from parent node [very near end] 
			{$\phantom{\rho^2}\sigma a$}}%
		edge from parent node [swap] {$\sigma c$}}%
		child{node [dummy] {}%
			child{node {}%
			edge from parent node [swap,very near end] {$\rho^2 a$}}%
			child[level distance = 2.4em]{%
			edge from parent node [swap,near end] {$b\phantom{j}$}}%
			child{node {}%
			edge from parent node [very near end] {$\phantom{\rho^2}a$}}%
		edge from parent node  {$c$}}%
	edge from parent node [swap] {$d$}};%
	\node at (5.5,0) {}%
	child{node [dummy] {}%
		child{node [dummy] {}%
			child{node {}%
			edge from parent node [swap,very near end] {$\rho^3 \sigma a$}}%
			child[level distance = 2.7em]{%
			edge from parent node [swap,near end] {$\rho \sigma b$}}%
			child{node {}%
			edge from parent node [very near end] 
			{$\phantom{\rho^3} \rho \sigma a$}}%
		edge from parent node [swap] {$\rho \sigma c$}}%
		child{node [dummy] {}%
			child{node {}%
			edge from parent node [swap,very near end] {$\rho^3 a$}}%
			child[level distance = 2.7em]{%
			edge from parent node [swap,near end] {$\rho b$}}%
			child{node {}%
			edge from parent node [very near end] 
			{$\phantom{\rho^3}\rho a$}}%
		edge from parent node  {$\rho c$}}%
	edge from parent node [swap] {$\rho d$}};%
            \begin{scope}[every node/.style={font=\footnotesize}]%
                  \node at (10.3,0) {}
                  child{node [dummy] {}%
                    child{node [dummy] {}%
                      child{node {}%
                        edge from parent node [swap,very near end] {$(G/K) \cdot b$}}%
                      child{node {}%
                        edge from parent node [very near end] {$(G/L) \cdot a$}}%
                      edge from parent node [right] {$(G/K) \cdot c$}}%
                    edge from parent node [right] {$(G/H) \cdot d$}};%
            \end{scope}%
            \draw[decorate,decoration={brace,amplitude=2.5pt}] (5.6,0) -- (-0.1,0) node[midway,inner sep=4pt,font=\normalsize]{$T$}; %
            \node at (10.3,-0.15) [font=\normalsize] {$T$}; 
     \end{tikzpicture}%
\end{equation}
Edge labels in the expanded representation 
encode a $G$-action, 
so that the edges labeled $a,b,c,d$
have isotropies $L,K,K,H$, respectively.
On the other hand, the orbital representation displays the edge orbits in the expanded representation,
labeled by the associated transitive $G$-set.

Formally, $\Omega_G$ is defined as follows.
Write $\Phi^G$ for the category of 
$G$-forests, i.e. $G$-objects in the category 
$\Phi$ of forests
in Definition \ref{FOR DEF}.
We then define
$\Omega_G \subset \Phi^{G}$
as the full subcategory consisting of the
(non-empty) $G$-forests
such that the $G$-action is transitive on tree components.

For $T \in \Omega_G$, we write
$\boldsymbol{E}_G(T) = \boldsymbol{E}(T)/G$, $\boldsymbol{E}^{\mathsf{i}}_G(T) = \boldsymbol{E}^{\mathsf{i}}(T)/G$,
$\boldsymbol{V}_G(T) = \boldsymbol{V}(T)/G$
for the sets of \textit{edge orbits}, \textit{inner edge orbits}, and \textit{$G$-vertices}, respectively.

\begin{remark}
	We caution that there is a 
	\emph{proper} inclusion,
	$\Omega^G \subsetneq \Omega_G$,
	where $\Omega^G$ denotes $G$-objects in $\Omega$.
\end{remark}

\begin{remark}
	We find it convenient to consider
	multiple descriptions of a $G$-tree $T\in \Omega_G$.

	First, we have the forest description
	$T = \amalg_{i \in I} T_i$,
	where $I$ is a transitive $G$-set.
	Second, given a tree component $T_{\**}$
	with stabilizer $H \leq G$,
	one has a decomposition
	$T \simeq G \cdot_H T_{\**}$.
	Third, writing 
	$\Gamma \leq G \times \Aut(T_{\**})$
	for the graph subgroup (cf. \eqref{NORMMAP EQ})
	associated to the homomorphism 
	$\phi \colon H \to \Aut(T_{\**})$ that encodes the $H$-action on $T_{\**}$,
	one has a quotient description
	$T \simeq \left(G \cdot T_{\**}\right)/\Gamma$.
\end{remark}

\begin{remark}
	The two representations in \eqref{GTREE_EQ} 
	play complementary roles in 
	the discussion of $G$-trees.
	On the one hand,
	maps in $\Omega_G$ are best understood 
	via expanded representations,
	which encode the relevant broad poset structures.
	On the other hand, orbital representations
	pictorially encode composition data 
	for norm maps of operads
	(see e.g. \cite[Ex. 4.9]{Per18}, \cite[Rem. 3.39]{BP20}).
\end{remark}

Following the discussion at the end of \S \ref{FORESTS_SEC},
each $G$-tree
$T = \amalg_{i \in I} T_i$
is a planarized forest,
so that each tree component $T_i$
is planarized and the set $I$ of components
is totally ordered.

As was the case in $\Omega$,
maps $\varphi \colon S \to T$ in $\Omega_G$
are likewise built from a few basic types of maps.

Recall that a map
$\varphi \colon 
\amalg_I S_i
\to
\amalg_J S_j
$
in $\Omega_G$
is described by a map of sets
$\varphi \colon I \to J$
and maps of trees
$S_i \to T_{\varphi(i)}$ for $i \in I$.
A map $\varphi$
is called a 
\emph{quotient map}
if all maps 
$S_i \to T_{\varphi(i)}, i \in I$
are isomorphisms in $\Omega$
and called a 
\emph{sorted map} if
$\varphi \colon I \to J$
is an ordered isomorphism of sets.
Further, a sorted map
is called a
\emph{sorted degeneracy/tall map/face/inner face/outer face}
if each of the component maps is a
degeneracy/tall map/face/inner face/outer face 
in $\Omega$.

\begin{definition}
	We write $\Omega_G^0 \subseteq \Omega_G$ for the wide subcategory of $G$-trees and quotient maps.
	Further, we write $\Sigma_G \subseteq \Omega_G^0$ for the full subcategory spanned by \textit{$G$-corollas} and quotient maps,
	where $C \in \Omega_G$ is a $G$-corolla if $|\boldsymbol{V}_G(C)| = 1$
	or, equivalently,
	if its trees components are corollas.
\end{definition}

\begin{example}
      Let $G = \mathbb{Z}/ 2\mathbb{Z} = \{\pm 1\}$
      be the cyclic group with two elements.
      In the diagram below,
      the assignment 
	$\alpha \mapsto a$, 
	$\bar{\alpha} \mapsto -a$, 
	$\beta \mapsto b$, 
	$\gamma \mapsto c$
	determines a quotient map
	$S \to T$.
\[
\begin{tikzpicture}[grow=up,auto,level distance=2em,
	every node/.style={font=\scriptsize,inner sep = 2pt}]%
\tikzstyle{level 2}=[sibling distance=2.5em]%
\begin{scope}[every node/.style={font=\footnotesize},xshift=6.25cm]%
	\node at (0,0) {}
	child{node [dummy] {}
                          child{edge from parent node [swap,near end] {$G + \bar{\alpha}$}}
                          child[level distance=2.5em]{node {$G + \beta$}}
                          child{edge from parent node [near end] {$G + \alpha$}}
	edge from parent node [swap] {$G + \gamma$} };
	\node at (2.25,0.8) {$\xrightarrow{\qquad}$};
                        \node at (4,0) {}
                        child{node [dummy] {}
                          child{edge from parent node [swap,near end] {$G/G + b$}}
                          child{edge from parent node [near end]{$G+a$}}
                          edge from parent node [swap] {$G/G+c$}
                        };
	\node at (0,-.15) [font=\normalsize] {$S$};
	\node at (4,-.15) [font=\normalsize] {$T$};
\end{scope}%
\begin{scope}[level distance=2em]
\tikzstyle{level 2}=[sibling distance=1.8em]%
	\node at (-2,0) {}
	child{node [dummy] {}
		child{
		edge from parent node [swap, near end] {$\bar{\alpha}\phantom{|}$}}
		child[level distance=2.5em]{
		edge from parent node [swap, very near end] {$\beta$}}
		child{
		edge from parent node [near end] {$\phantom{|}\alpha$}}
	edge from parent node [swap] {$\gamma$} };
	\node at (0.2,0) {}
	child{node [dummy] {}
		child{
		edge from parent node [swap, near end] {$-\bar{\alpha}\phantom{|}$}}
		child[level distance=2.5em]{
		edge from parent node [swap, very near end] {$-\beta$}}
		child{
		edge from parent node [near end] 
		{$\phantom{|}-\alpha$}}
	edge from parent node [swap] {$-\gamma$}};
	\node at (1.75,0.8) {$\xrightarrow{\qquad}$};
	\node at (3,0) {}
	child{node [dummy] {}
		child{
		edge from parent node [swap, near end] {$b$}}
		child[level distance=2.5em]{
		edge from parent node [swap, very near end] {$-a$}}
		child{
		edge from parent node [near end] {$\phantom{b}a$}}
	edge from parent node [swap] {$c$}};
\draw[decorate,decoration={brace,amplitude=2.5pt}] 
(0.3,0) -- (-2.1,0) node[midway,inner sep=4pt,font=\normalsize]{$S$}; %
\node at (3,-0.15) [font=\normalsize] {$T$};
\end{scope}
\end{tikzpicture}
\]
\end{example}

\begin{notation}\label{TVG_NOT}
        Following Notation \ref{TV_NOT}, 
        for $T \in \Omega_G$ and a $G$-vertex $v \in \boldsymbol{V}_G(T)$
        given as a $G$-orbit $[v_i]$ with 
        $v_i\in \boldsymbol{V}(T)$,
        we write
		$T_v = \amalg_{v = [v_i]} T_{v_i}$.
		Informally, $T_v$ is the $G$-corolla formed by the 
		associated ``$G$-orbit of corollas'' $T_{v_i}$.
\end{notation}

\begin{corollary}[{cf. Prop. \ref{NP_TREEFACT_PROP}, \cite[Rem. 5.49]{Per18}}]
	A map $\varphi \colon S \to T$ in $\Omega_G$ has
	a factorization
	\[
	S \xrightarrow{sd} 
	S' \xrightarrow{si} 
	S'' \xrightarrow{so} 
	T^{\**} \xrightarrow{q} T,
	\]
	unique up to unique 
	sorted isomorphisms, as
	a sorted degeneracy followed by
	a sorted inner face,
	a sorted outer face,
	and a quotient map.
\end{corollary}

\begin{remark}
        \label{PULLBACK_REM}
	All sorted maps
	$\varphi \colon S \to T$ in $\Omega_G$
	are independent maps,
	so one may ask if such a map is a planar map,
	which is then equivalent
	to each component map
	$S_i \to T_{\varphi(i)}$
	being planar.

	On the other hand, a quotient map
	$S \to T$ in $\Omega_G$
	is independent iff it is an isomorphism,
	so a quotient is planar iff it is an identity.
	More generally, 
	asking for a 
	a quotient map to be componentwise planar,
	i.e. for the maps 
	$S_i \to T_{\varphi(i)}$
	to be planar,
        is the same as requiring the maps
	$S_i \to T_{\varphi(i)}$
	being identities (since we already know they are isomorphisms).
	A quotient map with this additional property
	of being componentwise planar is called 
	a \emph{pullback map} (cf. \cite[Ex. 3.27]{BP21}).
\end{remark}

\begin{corollary}[{cf. Prop. \ref{TREEFACT_PROP}, \cite[Rem. 5.49]{Per18}}]
        \label{OMGFACT COR}
	A map $\varphi \colon S \to T$ in $\Omega_G$ has
	a strictly unique factorization
\[
	S \xrightarrow{\simeq,s}
	S \xrightarrow{spd} 
	S' \xrightarrow{spi} 
	S'' \xrightarrow{spo} 
	\varphi^{\**}T \to T,
\]
	as a sorted isomorphism
	followed by a sorted planar degeneracy,
	a sorted planar inner face,
	a sorted planar outer face,
	and a pullback map.
\end{corollary}

%\begin{notation}\label{LRG_NOT}
%      As in Notation \ref{LR_NOT}, 
%      for any $G$-tree $T = \amalg_{i}T_i$ 
%      there is a unique $G$-corolla 
%      $\mathsf{lr}(T)$ together with a planar tall map $\mathsf{lr}(T) \to T$,
%      called the \textit{leaf-root of $T$}.
%      Explicitly, $\mathsf{lr}(T) = \amalg_{i} \mathsf{lr}(T_i)$.
%\end{notation}

Lastly, in contrast to the notion of sorted face map above,
the description of the model structure
on equivariant dendroidal sets $\mathsf{dSet}^G$
will require an additional notion of face.
\begin{definition}
	Let $T = \amalg_i T_i$ be a $G$-tree.	
	An \textit{(outer) face} of $T$ is an (outer) face map 
	$U \to T_i$ from some $U \in \Omega$ to a component $T_i$ of $T$.
	A face of $T$ is called \textit{planar} if the map $U \to T_i$ is a planar map.

	We write $\mathsf{Face}(T)$ for the $G$-poset of planar faces of $T$,
	and let $\mathsf{Face}_{\mathsf{sc}}(T) \subseteq \mathsf{Face}(T)$ denote the subposet spanned by the planar outer faces of $T$ with no inner edges
	(these are the faces determined by either a single edge or a single vertex of $T$).
\end{definition}

% We can think of these faces as coming from manipulations of the \textit{expanded} representation of a $G$-tree.

\subsection{Equivariant dendroidal sets}
\label{EDS_SEC}

Recall that the category of \emph{dendroidal sets} \cite{MW07}
is the presheaf category
$\mathsf{dSet} = \mathsf{Set}^{\Omega^{op}}$.
For a (finite) group $G$,
the category of
\emph{$G$-dendroidal sets}
is then the $G$-object category
$\mathsf{dSet}^G = \mathsf{Set}^{G \times \Omega^{op}}$.

One key subtlety when working with 
$\mathsf{dSet}^G$
is that for each equivariant dendroidal set 
$X \in \mathsf{dSet}^G$
its levels $X(U)$ are indexed by
non-equivariant trees $U \in \Omega$,
while the key classes of maps in $\mathsf{dSet}^G$ are defined in terms
of $G$-trees $T \in \Omega_G$.

To describe these maps, 
we first extend the Yoneda embedding
$\Omega[-]\colon \Omega \to \mathsf{dSet}$
notation
for the representable functors
$\Omega[U](V) = \Omega(V,U)$
to obtain extended Yoneda embeddings
(here the right embedding is simply obtained from the left embedding by taking $G$-objects)
\begin{equation}\label{FORESREP EQ}
\begin{tikzcd}[row sep=0]
	\Phi \ar{r}{\Omega[-]}
&
	\mathsf{dSet}
&&%%
	\Phi^G \ar{r}{\Omega[-]}
&
	\mathsf{dSet}^G
\\
	\amalg_i F_i
	\ar[mapsto]{r}
&
	\amalg_i \Omega[F_i]
&&%%
	\amalg_i F_i
	\ar[mapsto]{r}
&
	\amalg_i \Omega[F_i]
\end{tikzcd}
\end{equation}

Note that,
since $G$-trees $\Omega_G$
are defined as a subcategory of $G$-forests $\Phi^G$,
the right side of 
\eqref{FORESREP EQ}
defines representables 
$\Omega[T] \in \mathsf{dSet}^G$
for $T \in \Omega_G$.
These representable presheaves 
$\Omega[T]$
then allow us to generalize the key 
presheaves in dendroidal sets $\mathsf{dSet}$
(cf. \cite[\S 2]{CM13a})
to equivariant dendroidal sets
$\mathsf{dSet}^G$
(cf. \cite[\S 6]{Per18} or \cite[\S 2.3]{BP20}),
as follows.

\begin{definition}\label{DSETPRESHEAF_DEF}
      Let $T = \amalg_i T_i$ be $G$-tree.
      The \emph{boundary}
      $\partial \Omega[T] \subseteq \Omega[T]$ is defined by
      \begin{equation}
            \partial \Omega[T]
            = \coprod_i \partial \Omega[T_i]
            = \mathop{\colim}\limits_{U \in \mathsf{Face}(T), U \neq T_i} \Omega[U]
            = \bigcup_{U \in \mathsf{Face}(T), U \neq T_i} \Omega[U].
      \end{equation}
      Next, for $\emptyset \neq E \subseteq \boldsymbol{E}^{\mathsf{i}}(T)$ a non-empty $G$-subset of inner edges and writing $E_i = E \cap \boldsymbol{E}(T_i)$,
      the \textit{$G$-inner horn} 
      $\Lambda^E[T] \subseteq \Omega[T]$ is defined by
      \begin{equation}\label{GINNERHORN_EQ}
            \Lambda^{E}[T]
            = \coprod_i \Lambda^{E_i}[T_i]
            = \mathop{\colim}\limits_{U \in \mathsf{Face}(T), (T_i - E_i) \not\into U} \Omega[U]
            = \bigcup_{U \in \mathsf{Face}(T), (T_i - E_i) \not\into U} \Omega[U].
      \end{equation}
      Lastly, the \textit{Segal core} $Sc[T] \subseteq \Omega[T]$ is defined by
      \begin{equation}\label{eq:SC}
              Sc[T] 
              = \coprod_i Sc[T_i]
              = \mathop{\colim}\limits_{U \in \mathsf{Face}_{\mathsf{sc}}(T)} \Omega[U]
              = \bigcup_{U \in \mathsf{Face}_{\mathsf{sc}}(T)} \Omega[U].
      \end{equation}
\end{definition}

\begin{remark}
For $T \in \Omega_G$, a decomposition $T \simeq G \cdot_H T_{\**}$ with $T_{\**} \in \Omega^H$ yields identifications
\[
	\Omega[T] \simeq G \cdot_H \Omega[T_{\**}],
\qquad
	\partial\Omega[T] \simeq G \cdot_H \partial \Omega[T_{\**}],
\qquad
	\Lambda^{E}[T] \simeq G \cdot_H \Lambda^{E_{\**}}[T_{\**}],
\qquad
	Sc[T] \simeq G \cdot_H Sc[T_{\**}],
\]
where $E_{\**} = E \cap \boldsymbol{E}(T_{\**})$.
\end{remark}

Adapting the non-equivariant story, the maps in Definition \ref{DSETPRESHEAF_DEF} are then the basis for a model structure on $\dSet^G$.
In the following, recall that a class of maps in a category is called \textit{saturated} if it is closed under pushouts, retracts, and transfinite compositions.

\begin{definition}\label{GINNERANN DEF}
	The class of \textit{$G$-normal monomorphisms}
	in $\mathsf{dSet}^G$
	is the saturation of the boundary inclusions 
	$\partial \Omega[T] \into \Omega[T]$ for $T \in \Omega_G$.
	
	The class of \textit{$G$-inner anodyne extensions}
	in $\mathsf{dSet}^G$
	is the saturation of the $G$-inner horn inclusions
	$\Lambda^E [T] \into \Omega[T]$ 
	for $T \in \Omega_G$ and 
	$\emptyset \neq E \subseteq
	\boldsymbol{E}^{\mathsf{i}}(T)$
	a non-empty $G$-subset.
\end{definition}

\begin{definition}\label{GINNERFIB DEF}
	A map $X \to Y$ in $\dSet^G$
	is called a \emph{$G$-inner fibration}
	if it has the right lifting property
	with respect to all $G$-inner horn inclusions
	$\Lambda^E[T] \to \Omega[T]$.
	
	Moreover, 
	if $X \to \**$
	is a $G$-inner fibration
	then 
	$X \in \dSet^G$ is called a \textit{$G$-$\infty$-operad}.
\end{definition}

Informally, one may view
$G$-$\infty$-operads as
``operads with weak composition laws for norm maps''.

%\begin{example}
%       We recall the setup of \eqref{GTREE_EQ}, with $G = D_8$, $H = \langle \sigma^2,\rho \rangle$, $K = \langle \sigma^2 \rangle$, $L = \**$, and $T$ the given tree diagram.
%       Let $X \in \dSet^G$. Then we have natural solid arrows
%       \[
%             \begin{tikzcd}
%                   X(H/K) \times X(K/L \amalg K/K) \arrow[r, start anchor = east, end anchor = west, dashed, bend left, shift left]
%                   &
%                   \Hom_{\dSet^G}(\Omega[T], X) \arrow[l, shift left] \arrow[r]
%                   &
%                   X(H/L \amalg H/K)
%             \end{tikzcd}
%      \]
%       given by the presheaf structure,
%       If $X$ is a $G$-$\infty$-operad, there exists a dashed section (possibly many) as written,
%       and the composite map provides a choice of a composition law.
%       providing a choice of composition.     
%\end{example}

To recall the model structure on $\mathsf{dSet}^G$,
we need two more ingredients.
First, we write
\begin{equation}\label{IOTASHADJ EQ}
	\iota_! \colon 
	\mathsf{sSet}
	\rightleftarrows
	\mathsf{dSet}
	\colon
	\iota^{\**}
\qquad
	\iota_! \colon 
	\mathsf{Cat}
	\rightleftarrows
	\mathsf{Op}
	\colon
	\iota^{\**}
\end{equation}
for the adjunctions
where the left adjoints $\iota_!$
are the natural inclusions
(given by ``extension by $\emptyset$'').
Note that the leftmost adjunction is induced by the natural inclusion
$\iota \colon \Delta \to \Omega$
as the linear trees.
Second, we write
\begin{equation}\label{TAUADJ EQ}
	\tau \colon 
	\mathsf{sSet}
	\rightleftarrows
	\mathsf{Cat}
	\colon
	N
\qquad
	\tau \colon 
	\mathsf{dSet}
	\rightleftarrows
	\mathsf{Op}
	\colon
	N
\end{equation}
for the adjunctions where the right adjoints are the 
\emph{nerve functors}
given by
$(N \mathcal{C})(n) = 
\mathsf{Cat}\left([n],\mathcal{C}\right)$
for $\mathcal{C}\in \mathsf{Cat}$
and $[n] = (0 \to 1 \to \cdots \to n)$,
and
$(N \mathcal{O})(T) = 
\mathsf{Op}\left(\Omega(T),\mathcal{O}\right)$
for $\mathcal{O}\in \mathsf{Op}$
and $\Omega(T)$ the colored operad (of sets) generated by $T$
(cf. \cite[\S 3]{MW07}, \cite[Rem. 4.4, Ex. 4.6]{Per18} or \eqref{TAUNER EQ}).
%\cite[\S 3]{MW07}, \cite[Rem. 4.4, Ex. 4.6]{Per18}, \eqref{OMEGADEF_EQ}).

Recall \cite[Prop. 5.3 and Thm. 6.1]{MW09}
that the nerve functors $N$ are fully faithful, 
with their essential image characterized
as those presheaves 
that satisfy a Segal condition.

The following theorem synthesizes Theorem 2.1, Proposition 8.8, and Theorem 8.22 of \cite{Per18}.
% \cite[Thm. 2.1, Prop. 8.8, Thm. 8.22]{Per18}]
\begin{theorem}[\cite{Per18}]
        \label{DSETGMOD THM}
        There exists a left proper model structure on $\dSet^G$ such that:
        \begin{itemize}
	\item the cofibrations are the $G$-normal monomorphisms;
	\item the fibrant objects are the $G$-$\infty$-operads;
	\item the fibrations between
                $G$-$\infty$-operads 
                are the $G$-inner fibrations $X \to Y$
                such the induced maps on 
                fixed-point homotopy categories 
                $\tau \iota^{\**}(X^H \to Y^H)$ are isofibrations of categories for all $H \leq G$;
        \item the weak equivalences are the smallest class of maps closed under 2-out-of-3 which
                contains the $G$-inner anodyne extensions and the trivial fibrations
                (i.e. those maps with the right lifting property against the $G$-normal monomorphisms).
        \end{itemize}
\end{theorem}

In addition to the category
$\mathsf{dSet}^G = \mathsf{Set}^{G \times \Omega^{op}}$
of equivariant dendroidal sets,
there is also a category
$\mathsf{dSet}_G = \mathsf{Set}^{\Omega_G^{op}}$,
which we call the category of 
\emph{genuine dendroidal sets}.
As it turns out,
several natural constructions in the non-equivariant setting 
generalize to produce objects in 
$\mathsf{dSet}_G$
rather than in 
$\mathsf{dSet}^G$
(e.g. $\mathsf{dSet}_G$ is essential to establishing the characterization
of the fibrant objects in $\mathsf{dSet}^G$,
cf. \cite[\S 8.2]{Per18}),
so we next recall the connection between the two categories.

Let $\mathsf{O}_G$
denote the \emph{orbit category} of the group $G$,
i.e. the category of transitive $G$-sets $G/H$ for $H\leq G$ and $G$-set maps.
Regarding the group $G$ as a single object category,
one then has a fully faithful inclusion
$\upsilon \colon G^{op} \to \mathsf{O}_G$
sending the object of $G$ to the free $G$-orbit $G/e$. 
In addition, there is a fully faithful inclusion
$\Omega \times \mathsf{O}_G \to \Omega_G$
given by
$(T,G/H) \mapsto G/H \cdot T$.
Altogether, one obtains a commutative diagram of fully faithful inclusions as follows.
\begin{equation}\label{UPSIOTADIAG EQ}
\begin{tikzcd}[row sep=9,column sep=11]
	&
	\Delta \times \mathsf{O}_G 
	\ar{rd}{\iota}
	\ar[bend left=10]{rrd}{\iota_G}
	&
\\
	\Delta \times G^{op} 
	\ar{ru}{\upsilon}
	\ar{rd}[swap]{\iota}
	&&
	\Omega \times \mathsf{O}_G
	\ar{r}
	&
	\Omega_G
\\
	&
	\Omega \times G^{op}
	\ar{ru}[swap]{\upsilon}
	\ar[bend right=10]{rru}[swap]{\upsilon_G}
	&
\end{tikzcd}
\end{equation}

The connection between 
$\mathsf{dSet}^G$ and $\mathsf{dSet}_G$
is then given by the rightmost adjunction
\begin{equation}\label{UPSILONADJ EQ}
	\upsilon^{\**} \colon 
	\mathsf{sSet}^G
	\rightleftarrows
	\mathsf{sSet}^{\mathsf{O}_G^{op}}
	\colon
	\upsilon_{\**}
\qquad
	\upsilon_G^{\**} \colon 
	\mathsf{dSet}^G
	\rightleftarrows
	\mathsf{dSet}_G
	\colon
	\upsilon_{G,\**}
\end{equation}
where we note that the right adjoints 
$\upsilon_{\**}$, $\upsilon_{G,\**}$
are fully faithful inclusions.
Explicitly, one has 
$\upsilon_{\**} X(G/H) = X^H$ and
$\upsilon_{G,\**} X(T) = \dSet^G(\Omega[T], X) \simeq X(T_{\**})^H$
for $T \simeq G \cdot_H T_{\**}$ with $T_{\**} \in \Omega^H$.

The fully faithful functors appearing in 
\eqref{IOTASHADJ EQ},
\eqref{TAUADJ EQ},
\eqref{UPSILONADJ EQ}
then fit into commutative diagrams
as below, where $\mathsf{Op}_G$
is the category of \emph{genuine equivariant operads} (discussed below).
\begin{equation}\label{ALLFULL EQ}
\begin{tikzcd}
\mathsf{Cat}^G
\ar{r}{\iota_!}
\ar{d}[swap]{N}
&
\mathsf{Op}^G
\ar{r}{\upsilon_{G,\**}}
\ar{d}[swap]{N}
&
\mathsf{Op}_G
\ar{d}{N_G}
&&
\mathsf{Cat}^G
\ar{r}{\upsilon_{\**}}
\ar{d}[swap]{N}
&
\mathsf{Cat}^{\mathsf{O}_G^{op}}
\ar{r}{\iota_{G,!}}
\ar{d}[swap]{N}
&
\mathsf{Op}_G
\ar{d}{N_G}
\\
\mathsf{sSet}^G
\ar{r}[swap]{\iota_!}
&
\mathsf{dSet}^G
\ar{r}[swap]{\upsilon_{G,\**}}
&
\mathsf{dSet}_G
&&
\mathsf{sSet}^G
\ar{r}[swap]{\upsilon_{\**}}
&
\mathsf{sSet}^{\mathsf{O}_G^{op}}
\ar{r}[swap]{\iota_{G,!}}
&
\mathsf{dSet}_G
\end{tikzcd}
\end{equation}

Here, genuine equivariant operads are an extension of the notion of operad
which in the single colored context was 
first defined in \cite{BP21}
via algebraic means. 
However, to sidestep the technical work needed to extend 
the definition in \cite{BP21} 
to the colored context,
here we follow the approach in
\cite[Def. 3.35]{BP20}
and regard $\mathsf{Op}_G$
simply as the full subcategory of 
$\mathsf{dSet}_G$
of those objects satisfying the strict Segal condition below.
As such, 
the fact that $N_G$ in \eqref{ALLFULL EQ}
is fully faithful is tautological
(so that the top 
$\upsilon_{G,\**}$ functor
is just a restriction of 
the lower $\upsilon_{G,\**}$ functor).

\begin{definition}\label{OPG_DEF}
	A presheaf $Z \in \dSet_G$ is called a \textit{genuine equivariant operad} if
	it satisfies the strict right lifting condition against the Segal core inclusions
	$\upsilon_{G,\**}\left(Sc[T] \to \Omega[T] \right)$ for all $T \in \Omega_G$.
\end{definition}

By taking left adjoints of 
the vertical nerve functors in \eqref{ALLFULL EQ} 
we obtain the following diagram
where those squares that feature natural transformations
\emph{do not} commute. 
Here the existence of the dashed left adjoint 
$\tau_G$ to $N_G$
requires justification, 
with a full discussion of $\tau_G$
being the objective of Appendix \ref{HGEO AP}
(see also the discussion in 
\cite[Rem. \ref{W-GTAUFUNEX REM}]{BP_WCONS}).
\begin{equation}\label{TAUFUNCTS EQ}
\begin{tikzcd}
	\mathsf{sSet}^G
	\ar{r}{\iota_!}
	\ar{d}[swap]{\tau}
	&
	\mathsf{dSet}^G
	\ar{r}{\upsilon_{G,\**}}
	\ar{d}[swap]{\tau}
	&
	\mathsf{dSet}_G
	\ar[dashed]{d}{\tau_G}
	\ar[Rightarrow,dashed]{dl}
&&
	\mathsf{sSet}^G
	\ar{r}{\upsilon_{\**}}
	\ar{d}[swap]{\tau}
	&
	\mathsf{sSet}^{\mathsf{O}_G^{op}}
	\ar{r}{\iota^{\**}_G}
	\ar{d}[swap]{\tau}
	\ar[Rightarrow]{dl}
	&
	\mathsf{dSet}_G
	\ar[dashed]{d}{\tau_G}
\\
	\mathsf{Cat}^G
	\ar{r}[swap]{\iota_!}
	&
	\mathsf{Op}^G
	\ar{r}[swap]{\upsilon_{G,\**}}
	&
	\mathsf{Op}_G
&&
	\mathsf{Cat}^G
	\ar{r}[swap]{\upsilon_{\**}}
	&
	\mathsf{Cat}^{\mathsf{O}_G^{op}}
	\ar{r}[swap]{\iota_{G}^{\**}}
	&
	\mathsf{Op}_G
\end{tikzcd}
\end{equation}

\section{The tame model structure on preoperads}
\label{TAME_SEC}

Our goal in this section is to build the alternative 
\emph{tame model structure}
$\mathsf{PreOp}^G_{tame}$
on preoperads that is needed for the nerve functor
$N \colon \mathsf{sOp}^G \to \mathsf{PreOp}^G$
in \eqref{ADJSQ EQ}
to be right Quillen. 

First, \S \ref{JT_SEC},\S \ref{SPREOP_SEC} 
recall the model structures on
simplicial dendroidal sets $\mathsf{sdSet}^G$
and preoperads $\mathsf{PreOp}^G$ built in \cite{BP20}.
Then, \S \ref{FIBTENS_SEC}
builds an auxiliary construction, 
the \emph{fibered simplicial tensoring} $\otimes_{\mathfrak{C}_{\bullet}}$,
which has a key role in our description of
the tame model structure in \S \ref{TAMEDEFEX SEC}.

\subsection{Equivariant dendroidal Segal spaces}
\label{JT_SEC}

We recall the several model structures on the category of
\textit{equivariant simplicial dendroidal sets}
$\mathsf{sdSet}^G = \Set^{\Delta^{op} \times \Omega^{op} \times G}$
introduced in \cite{BP20}.
We first recall some notation.

\begin{notation}
        \label{SDSETEVAL_NOT}
      For $X \in \mathsf{sdSet}^G$, we write $X_n(U)$ for the evaluation at $n \in \Delta$ and $U \in \Omega$,
      and refer to $n$ and $U$ as the \textit{simplicial} and \textit{dendroidal} directions.
      More generally, we write
      \begin{equation}
            \label{SDSET_EQ}
            X_{(-)} \colon \sSet \to \dSet^G,
            \qquad
            X(-) \colon \dSet^G \to \sSet
      \end{equation}
      for the colimit-preserving functors
      such that $X_{\Delta[n]} = X_n$ and 
      $X\left(G \cdot\Omega[U]\right) = X(U)$ for $n \geq 0$, $U \in \Omega$.
      Explicitly, $X_K(U) = \sSet(K, X(U))$ and $X(A)_n = \dSet^G(A, X_n)$.
      
      Additionally, we have a natural fully-faithful inclusion
      \[
              c_{!} \colon \dSet^G \longto \mathsf{sdSet}^G
      \]
      as presheaves that are constant along the simplicial direction
      (that is, $X(T) \in \sSet^G$ is discrete for all $T \in \Omega$),
      and we will often identify presheaves with their images under $c_!$.
      
      Lastly, 
      for $A \in \dSet^G$ and $K \in \sSet$ we write $A \times K$ for the presheaf $(A \times K)_n(U) = A(U) \times K_n$;
      more generally,
      for maps $A \to B$ in $\dSet^G$ and $K \to L$ in $\sSet$,
      we write $(A \to B) \square (K \to L)$ for the \textit{pushout product} map with target $B \times L$ (see, for example, \cite[11.1.7]{Ri14}).
\end{notation}

The various model structures on $\mathsf{sdSet}^G$ arise
from the theory of (generalized) Reedy categories.

First, since $\Delta^{op}$ is a Reedy category,
the identification
$\mathsf{sdSet}^G \simeq 
\left(\mathsf{dSet}^G\right)^{\Delta^{op}}$
together with the model structure on 
$\mathsf{dSet}^G$ from \cite{Per18}
(cf. Theorem \ref{DSETGMOD THM})
yields the 
\textit{simplicial Reedy model structure} on $\mathsf{sdSet}^G$.
Note that the weak equivalences in this model structure
are the \textit{dendroidal equivalences},
i.e. maps
$f \colon X \to Y$
such that $X_n \to Y_n$ is a weak equivalence in $\dSet^G$ for all $n \geq 0$.

Second, as discussed in 
\cite[Ex. A.7]{BP20},
$\Omega^{op} \times G$ is \emph{a generalized Reedy category}
and the family of graph subgroups 
$\{\Gamma \leq \mathsf{Aut}(T)^{op} \times G
\ | \ \Gamma \cap \mathsf{Aut}(T)^{op} = \**\}_{T\in \Omega}$
(see \eqref{NORMMAP EQ}) is
\emph{Reedy admissible}
in the sense of \cite[Ex. A.2]{BP20}.
Hence, by \cite[Thm. A.8]{BP20},
the identification 
$\mathsf{sdSet}^G \simeq 
\mathsf{sSet}^{\Omega^{op} \times G}$
together with the Kan model structure on 
$\mathsf{sSet}$
yields the 
\textit{(equivariant) dendroidal Reedy model structure} on $\mathsf{sdSet}^G$.
The weak equivalences in this model structure are the
\emph{simplicial equivalences},
i.e. maps $f \colon X \to Y$
such that 
$X(\Omega[T]) \to Y(\Omega[T])$
are Kan equivalences in $\mathsf{sSet}$
for all $T \in \Omega_G$.

Third, as the simplicial and dendroidal Reedy model structures
in $\mathsf{sdSet}^G$ above have the same cofibrations,
the joint left Bousfield localization framework in 
\cite[\S 4.1]{BP20}
yields the following.

\begin{theorem}\label{JB_THM}
	The simplicial and dendroidal Reedy model structures on 
	$\mathsf{sdSet}^G$
	have a smallest\footnote{Here ``smallest'' means that the class of weak equivalences is as small as possible.}
	common left Bousfield localization,
	which we call the \emph{joint model structure}.
	Moreover:
\begin{enumerate}[label = (\roman*)]
	\item \label{PROPER_LBL}
	the joint model structure is left proper;
	\item \label{SDEQUIV_LBL}
	both the dendroidal and simplicial equivalences in $\mathsf{sdSet}^G$ are also joint equivalences;
	\item \label{JTFIB_LBL}
	$X$ is joint fibrant iff $X$ is both simplicial and dendroidal Reedy fibrant;
	\item \label{SFIB_JEQ_LBL} if $X,Y$ are joint fibrant
	then a map $X \to Y$ is a joint equivalence iff it is a simplicial
	equivalence iff it is a dendroidal equivalence;
	\item \label{DFIB_JEQ_LBL} if $X,Y$ are dendroidal fibrant
	then a map $X \to Y$ is a joint equivalence iff 
	it is a dendroidal equivalence iff $X_0 \to Y_0$ is an equivalence in $\mathsf{dSet}^G$;
	\item \label{JTCFIB_MAP_LBL} 
	if $X \to Y$ is a joint (co)fibration,
	the level maps 
	$X_n \to Y_n, n \geq 0$
	are (co)fibrations in $\mathsf{dSet}^G$
	and the maps
	$X\left(\Omega[T]\right) \to Y\left(\Omega[T]\right), T \in \Omega_G$
	are (co)fibrations in $\mathsf{sSet}$.

	In particular, if 
	$X$ is joint fibrant
	then $X_n \in \dSet^G$ and $X(\Omega[T]) \in \sSet$ are fibrant.
\end{enumerate}
\end{theorem}

\begin{proof}
	The existence of a smallest common left Bousfield localization is 
	an application of \cite[Prop. 4.1]{BP20},
	with the hypothesis that the dendroidal/simplicial Reedy model structures admit localizations being guaranteed by 
	Hirschhorn's existence result 
	\cite[Thm. 4.1.1]{Hir03}
        (in particular, both Reedy model structures are left proper and cellular by \cite[Theorems 15.3.4, 15.7.6]{Hir03} respectively,
        the proofs of which do not depend on the Reedy category being strict).
	\ref{PROPER_LBL} then follows from \cite[Thm. 4.1.1(3)]{Hir03}.
	\ref{SDEQUIV_LBL} holds by definition.
	\ref{JTFIB_LBL} and \ref{SFIB_JEQ_LBL} are \cite[Prop. 4.1(i)(ii)]{BP20}.
	\ref{DFIB_JEQ_LBL} is \cite[Cor. 4.29(iii)]{BP20}.
	Lastly, \ref{JTCFIB_MAP_LBL} follows from \cite[Lemmas A.27(i), A.29(i)]{BP20}.
\end{proof}

\begin{remark}\label{SQUAREEQUI REM}
	For $A \to B$ a normal monomorphism in $\mathsf{dSet}^G$
	and $K\to L$ a monomorphism in $\mathsf{sSet}$,
	the map $(A\to B) \square (K\to L)$
	is a cofibration in any of the model structures on $\mathsf{sdSet}^G$.
	Moreover, if $A\to B$ (resp. $K\to L$)
	is a weak equivalence in $\mathsf{dSet}^G$ (resp. $\mathsf{sSet}$),
	then $(A\to B) \square (K\to L)$ is a 
	dendroidal (resp. simplicial)
	equivalence in $\mathsf{sdSet}^G$.
\end{remark}

Fourth (and last), one has the \textit{(equivariant) dendroidal Segal space} model structure on $\mathsf{sdSet}^G$,
which is the left Bousfield localization of the dendroidal Reedy model structure with respect to the Segal core inclusions
\[
	Sc[T] \longto \Omega[T],
\qquad
	T \in \Omega_G.
\]

\begin{remark}
	The joint model structure on
	$\mathsf{sdSet}^G$
	can also be described as a further localization of the 
	dendroidal Segal space model structure
	imposing a "completion condition"
	\cite[Rem. 4.27]{BP20},
	and is thus also called the 
	\emph{complete (equivariant) dendroidal Segal space}
	model structure
        with \textit{complete equivalences}.
\end{remark}

\begin{lemma}\label{FCOLIM_WE_LEM}
	The weak equivalences in the dendroidal Reedy, dendroidal Segal space, and joint Reedy model structures 
	on $\mathsf{sdSet}^G$ are closed under filtered colimits.
\end{lemma}

\begin{remark}\label{WEFILRES REM}
	More explicitly, weak equivalences are closed
	under filtered colimits if a map of filtered colimits
	$\colim_i C_i \to \colim_i D_i$
	is a weak equivalence whenever all $C_i \to D_i$ are.
	Notably, this is 
	%readily shown to be 
	equivalent to 
	the claim that filtered colimits
	$\colim_i C_i$ are homotopy colimits.
\end{remark}

\begin{proof}
	Weak equivalences in the dendroidal Reedy model structure are simplicial equivalences, 
	so in that case the result is inherited from the analogous claim for $\mathsf{sSet}$.
	The result for the latter two model structures follows
	since they are left Bousfield localizations of the dendroidal Reedy model structure
	(as the alternative condition in
	Remark \ref{WEFILRES REM}
	is clearly preserved under localization).
\end{proof}

In what follows we will often make reference
to the $0$-(co)skeleton
of some $X \in \mathsf{sdSet}^G$
in the dendroidal Reedy structure.
To avoid confusion with the
$0$-(co)skeleton for the simplicial Reedy structure,
and noting that $\eta$ is the only tree in $\Omega$
of degree $0$, we introduce the following notation.

\begin{notation}
Let $X \in \mathsf{sdSet}^G$.
We write
$\mathsf{sk}_{\eta}X,\mathsf{csk}_{\eta}X \in \mathsf{sdSet}^G$
for the (co)skeleta described by
\[
	\left(\mathsf{sk}_{\eta} X\right)(U) =
	\coprod_{\boldsymbol{E}(U)} X(\eta),
\qquad \qquad
	\left(\mathsf{csk}_{\eta} X\right)(U) = 
	\prod_{\boldsymbol{E}(U)} X(\eta).
\]
\end{notation}

\subsection{Segal preoperads}
\label{SPREOP_SEC}

We next recall the normal model structure on preoperads
$\mathsf{PreOp}^G$
studied in \cite[\S 4, \S 5]{BP20}.

\begin{definition}\label{PREOP DEF} 
	The category of \textit{(equivariant) preoperads} $\mathsf{PreOp}^G$ is the full subcategory of $\mathsf{sdSet}^G$ spanned by those $X$ such that
	$X(\eta)$ is a discrete simplicial set.
\end{definition}

\begin{definition}
	Given $X \in \mathsf{PreOp}^G$ we call $X(\eta)$ the 
	\emph{color set of $X$}
	and denote the associated \emph{color set functor} as follows.
	\begin{equation}\label{COLORSET EQ}
	\begin{tikzcd}[row sep=0]
	\mathsf{PreOp}^G \ar{r}{\mathfrak{C}_{\bullet}} &
	\mathsf{Set}^G
	\\
	X \ar[mapsto]{r} &
	\mathfrak{C}_X = X(\eta)
	\end{tikzcd}
	\end{equation}

Further, for each fixed $\mathfrak{C} \in \mathsf{Set}^G$ we write 
$\mathsf{PreOp}^G_{\mathfrak{C}} \subset \mathsf{PreOp}^G$
for the fiber subcategory of 
\eqref{COLORSET EQ}
over $\mathfrak{C}$,
consisting of those $X$
such that $X(\eta) = \mathfrak{C}$
and maps that are the identity on colors.

	Lastly, for $f \colon \mathfrak{C} \to \mathfrak{D}$
	a map of $G$-sets of colors
	we define adjoint functors
\begin{equation}\label{PREOPCOLCH EQ}
	f_{!} \colon
	\mathsf{PreOp}^{G}_{\mathfrak{C}}
	\rightleftarrows
	\mathsf{PreOp}^{G}_{\mathfrak{D}}
	\colon f^{\**}
\end{equation}
	via the pushout and pullback squares below
	(note that 
	$\mathsf{sk}_{\eta} f_! A = \coprod_{\mathfrak{C}} \Omega[\eta]$ depends only on 
	$\mathfrak{C}$ while 
	$\mathsf{csk}_{\eta} f^{\**} X
	= \prod_{\mathfrak{D}} \Omega[\eta]$ depends only on
	$\mathfrak{D}$)
	\[
	\begin{tikzcd}
	\mathsf{sk}_{\eta} A \ar{r} \ar{d} \arrow[dr, phantom, "\ulcorner", very near start]  &
	\mathsf{sk}_{\eta} f_! A \ar{d}
	&&
	f^{\**} X \ar{r} \ar{d} &
	X \ar{d}
	\\
	A \ar{r} & 
	f_! A
	&&
	\mathsf{csk}_{\eta} f^{\**} X \ar{r} & 
	\mathsf{csk}_{\eta} X
	\arrow[ul, phantom, "\lrcorner", very near start]
	\end{tikzcd}
	\]
\end{definition}

\begin{remark}\label{GROTHFIBOP REM}
	The color functor \eqref{COLORSET EQ}
	is both a Grothendieck opfibration and fibration
	(cf. \cite[\S \ref{OC-GROTFIB SEC}]{BP_FCOP}),
	with cocartesian (resp. cartesian) arrows the natural maps
	$A \to f_! A$ (resp. $f^{\**} X \to X$).
	In particular, maps
	$f_! A \to B$ over a
	map of colors $\mathfrak{D} \to \mathfrak{E}$
	are in bijection with maps
	$A \to B$ over the
	over the composite map of colors $\mathfrak{C} \to \mathfrak{D} \to \mathfrak{E}$,
	and dually for $f^{\**} X$.
\end{remark}

The inclusion 
$\gamma^{\**} \colon \mathsf{PreOp}^G \to \mathsf{sdSet}^G$
admits both a left adjoint $\gamma_!$
and a right adjoint $\gamma_{\**}$
\[
      \begin{tikzcd}[column sep =5em]
            \mathsf{PreOp}^G \ar{r}[swap]{\gamma^{\**}} 
            &
            \mathsf{sdSet}^G
            \ar[bend right]{l}[swap,midway]{\gamma_{!}}
            \ar[bend left]{l}{\gamma_{\**}}
      \end{tikzcd}
\]
described by the following pushout and pullback squares.
\begin{equation}\label{GAMMASTAR_EQ}
\begin{tikzcd}
	\mathsf{sk}_{\eta} X \ar{r} \ar{d} \arrow[dr, phantom, "\ulcorner", very near start]  
&
	\pi_0 \mathsf{sk}_{\eta} X \ar{d}
&& 
	\gamma_{\**}X \ar{r} \ar{d} 
&
	X \ar{d}
\\
	X \ar{r} 
&
	\gamma_! X 
&&
	\mathsf{csk}_{\eta} X_0
	\ar{r} 
&
	\mathsf{csk}_{\eta} X 
	\arrow[lu, phantom, "\lrcorner", very near start]
\end{tikzcd}
\end{equation}
More explicitly: 
$\gamma_{!}X (U) = X(U)$ for non-linear trees $U \in \Omega \setminus \Delta$,
while $\gamma_{!}X ([n])$ for $[n] \in \Delta$ linear is given by the pushout on the left below; 
$\gamma_{\**}X(U)$ is given by the pullback on the right below.
\begin{equation}\label{GAMMASTAR2_EQ}
      \begin{tikzcd}
            X(\eta) \ar{r} \ar{d} \arrow[dr, phantom, "\ulcorner", very near start]  &
            \pi_0 X(\eta) \ar{d}
            && 
            \gamma_{\**}X(U) \ar{r} \ar{d} & X(U) \ar{d}
            \\
            X([n]) \ar{r} & \gamma_! X([n]) 
            &&
            \prod_{\boldsymbol{E}(U)} X_0(\eta) \ar{r} &
            \prod_{\boldsymbol{E}(U)} X(\eta)
            \arrow[lu, phantom, "\lrcorner", very near start]
      \end{tikzcd}
\end{equation}

%Moreover, following \cite[Remark 4.33]{BP20}, % GAMMASH REM
%we have that any monomorphism $A \to B$ in $\mathsf{sdSet}^G$
%such that $A(\eta) \to B(\eta)$ is an isomorphism
%induces a pushout square
%\begin{equation}\label{GAMMASH EQ}
%\begin{tikzcd}
%	A \ar{r} \ar{d} \arrow[dr, phantom, "\ulcorner", very near start]
%&
%	\gamma^{\**}\gamma_! A \ar{d}
%\\
%	B \ar{r} 
%&
%	\gamma^{\**}\gamma_! B 
%\end{tikzcd}
%\end{equation}

By largely formal arguments, the joint model structure on 
$\mathsf{sdSet}^G$ in Theorem \ref{JB_THM}
induces a model structure on $\mathsf{PreOp}^G$,
as follows.
We say a map $X \to Y$ in $\PreOp^G$ is a 
\textit{joint equivalence} (resp. \textit{normal monomorphism}) if
$\gamma^{\**} X \to \gamma^{\**} Y$ is a joint equivalence 
(resp. normal monomorphism) in $\sdSet^G$.
The following is then \cite[Thms. 4.39 and 4.42]{BP20},
with the additional ``moreover'' claims 
inherited from the 
analogous conditions in $\mathsf{sdSet}^G$ in
Theorem \ref{JB_THM}\ref{PROPER_LBL}
and Lemma \ref{FCOLIM_WE_LEM}.

\begin{theorem}\label{PREOPMS THM}
	There is a model structure on $\PreOp^G$, called the 
	\emph{normal model structure}, 
	such that weak equivalences (resp. cofibrations) are the joint equivalences (normal monomorphisms).
	
	Moreover, this model structure is left proper
	and weak equivalences are closed under filtered colimits.
	
	Lastly, the adjunction $\gamma^{\**} \colon \PreOp^G \rightleftarrows \sdSet^G \colon \gamma_{\**}$ is a Quillen equivalence
	between the normal model structure and the joint model structure on $\sdSet^G$.
\end{theorem}

For our purposes, we also need to recall
a convenient ``Dwyer-Kan description'' of the joint equivalences
between fibrant objects in $\PreOp^G$.
For this purpose, we first introduce the following new notation,
which extends notation in \cite[Def. 5.7]{BP20}
and
will simplify our discussion of the nerve functor
(see, e.g. \eqref{STRSEGCON EQ}).

\begin{notation}\label{XUDECALT NOT}
	Let $X \in \mathsf{PreOp}^G$, 
	$A \in \mathsf{dSet}^G$,
	and $\mathfrak{c} \colon A(\eta) \to X(\eta)$ 
	be a $G$-equivariant map.
	
	We define $X_{\mathfrak{c}}(A)\in \mathsf{sSet}$
	as the pullback	below
	(here the two squares are identical,
	providing only different descriptions of the bottom-right corner).
\begin{equation}\label{XUDECALT EQ}
\begin{tikzcd}
	X_{\mathfrak{c}}(A) \ar{r} \ar{d}
&
	X(A) \ar{d}
&%%
	X_{\mathfrak{c}}(A) \ar{r} \ar{d}
&
	X(A) \ar{d}
\\
	\** \ar{r}[swap]{\mathfrak{c}} 
&
	X \left( \mathsf{sk}_{\eta} A \right)
	\arrow[lu, phantom, "\lrcorner", very near start]
&%%
	\** \ar{r}[swap]{\mathfrak{c}} 
&
	\left(\prod_{A(\eta)} X(\eta)\right)^G
	\arrow[lu, phantom, "\lrcorner", very near start]
\end{tikzcd}
\end{equation}
Further, when $A = G \cdot \Omega[U]$
for $U \in \Omega$,
we abbreviate
$X_{\mathfrak{c}}(G \cdot \Omega[U])$
as 
$X_{\mathfrak{c}}(U)$.

Note that one thus has a coproduct decomposition
\begin{equation}
\label{COLDEC_EQ}
X(A) \simeq 
\coprod_{\mathfrak{c} \colon A(\eta) \to X(\eta)}
X_{\mathfrak{c}}(A).
\end{equation}
\end{notation}

\begin{remark}\label{PRIEX REM}
	Our primary instances of Notation \ref{XUDECALT NOT}
	occur when $A=\Omega[T]$ for $T \in \Omega_G$,
	in which case 
	$\Omega[T](\eta) = \boldsymbol{E}(T)$,
	so that
	$\mathfrak{c} \colon 
	\Omega[T](\eta) \to X(\eta)$
	can be regarded as a coloring of the edges 
	$\boldsymbol{E}(T)$
	of $T$ by the colors $X(\eta)$ of the preoperad $X$.
\end{remark}

\begin{remark}\label{MAPSPTRANS REM}
	Specifying Remark \ref{PRIEX REM} to the case 
	of $T=C$ a $G$-corolla, the coloring
	$\mathfrak{c} \colon 
	\Omega[C](\eta) \to X(\eta)$
	is tantamount to a $G$-equivariant map
	$\mathfrak{c} \colon \partial \Omega[C] \to X$,
	i.e. to a \emph{$C$-profile} in the sense of 
	\cite[Def. 5.6]{BP20}.
	Further, since there is an identification
	$X(\partial \Omega[C])
	\simeq
	\prod_{[e_i] \in \boldsymbol{E}_G(C)} X(\eta)^{H_i}$
	with $H_i \leq G$ the isotropy of the edge $e_i$,
	the data of the coloring $\mathfrak{c}$
	is equivalent to a choice of $x_i \in X(\eta)^{H_i}$
	for each $[e_i] \in \boldsymbol{E}_G(C)$.
	As such, one has an identification
\begin{equation}\label{MAPSPTRANS EQ}
	X_{\mathfrak{c}}(\Omega[C]) = X(x_1,\cdots,x_n;x_0)
\end{equation}
	where the \emph{mapping space}
	$X(x_1,\cdots,x_n;x_0)$
	is as defined in \cite[Def. 5.7]{BP20}.
	Further, 
	the decomposition in \eqref{COLDEC_EQ} %Remark \ref{COLDEC REM}
	then extends the decomposition in 
	\cite[Rem. 5.14]{BP20}.
\end{remark}

\begin{remark}
	Fix $\mathfrak{C} \in \mathsf{Set}^G$
	and consider the fiber subcategory 
	$\mathsf{PreOp}^G_{\mathfrak{C}} \subset \mathsf{PreOp}^G$.
	
	For $U \in \Omega$, the decomposition
	$X(U) = \coprod_{\mathfrak{c} \colon \boldsymbol{E}(U) \to \mathfrak{C}} X_{\mathfrak{c}}(U)$
	in \eqref{COLDEC_EQ} % Remark \ref{COLDEC REM}
	(note that we are using the abbreviated notation at the end of
	Notation \ref{XUDECALT NOT})
	then induces an equivalence of categories
\begin{equation}\label{PREOPCOLFIXEQ EQ}
\begin{tikzcd}[row sep = 0]
	\mathsf{PreOp}^G_{\mathfrak{C}}
	\ar{r}{\simeq}
&
	\mathsf{Fun}_{\**}(G \ltimes \Omega^{op}_{\mathfrak{C}},\mathsf{sSet})
\\
	(U \mapsto X(U))
	\ar[mapsto]{r}
&
	((U,\mathfrak{c}) \mapsto X_{\mathfrak{c}}(U))
\end{tikzcd}
\end{equation}
	where $\Omega_{\mathfrak{C}}$
	denotes the category of $\mathfrak{C}$-colored trees of 
	Definition \ref{CFOREST_DEF},
	$G \ltimes \Omega^{op}_{\mathfrak{C}}$,
	is an instance of
	\cite[Ex. \ref{OC-GLTIMES EQ}]{BP_FCOP}
	extending the category  
	$\Omega^{op}_{\mathfrak{C}}$
	by adding $G$-action arrows
	$\vect{U} = (U, \mathfrak{c})
	\to (U, g \mathfrak{c}) = g\vect{U}$
	for $g \in G$,
	and 
	$\mathsf{Fun}_{\**}(G \ltimes \Omega^{op}_{\mathfrak{C}},\mathsf{sSet})
	\subset
	\mathsf{Fun}(G \ltimes \Omega^{op}_{\mathfrak{C}},\mathsf{sSet})$
	is the subcategory of pointed functors,
	i.e. functors $Y$ such that
	$Y(\eta_{\mathfrak{c}}) = \**$,
	where $\mathfrak{c} \in \mathfrak{C}$ is a color
	and $\eta_{\mathfrak{c}}$ denotes the stick tree colored by $\mathfrak{c}$.
\end{remark}

\begin{remark}
	Given the equivalence \eqref{PREOPCOLFIXEQ EQ}
	and the alternative notation
	$\vect{U} = (U,\mathfrak{c})$
	for $\mathfrak{C}$-trees in $\Omega_{\mathfrak{C}}$, 
	it seems natural to abbreviate 
	$X_{\mathfrak{c}}(U)$ as $X(\vect{U})$.
	However, we will make significant use of the notation
	$X_{\mathfrak{c}}(\Omega[T])$
	in Remarks \ref{PRIEX REM}, \ref{MAPSPTRANS REM},
	and this latter notation is not readily recovered from 
	\eqref{PREOPCOLFIXEQ EQ}.
	As such, when dealing with preoperads we work only with the 
	$X_{\mathfrak{c}}(A),X_{\mathfrak{c}}(U)$
	notations,
	reserving the $\O(\vect{C})$ style notation for the context of operads 
	$\O \in \mathsf{sOp}^G$ (see \S \ref{GSOP_SEC}).
\end{remark}

\begin{remark}\label{SCTCOLPR REM}
	Let $X \in \mathsf{PreOp}^G$.
	For any $G$-tree $T \in \Omega_G$
	and coloring 
	$\mathfrak{c} \colon \boldsymbol{E}(T) \to X(\eta)$
	one has
	\[
                X_{\mathfrak{c}}(Sc[T]) 
                \simeq
                \prod_{v \in \boldsymbol{V}_G(T)}
                X_{\mathfrak{c}_v}(\Omega[T_v]) 
	\]
        where
        $\mathfrak{c}_v$
        denotes the restricted coloring given by the composite
        $\boldsymbol{E}(T_v) \to \boldsymbol{E}(T) 
        \xrightarrow{\mathfrak{c}} X(\eta)$,
        and $T_v$ is as in Notation \ref{TVG_NOT}.
\end{remark}

We can now recall the notion of 
Segal operad, 
cf. \cite[Def. 5.5]{CM13b}, \cite[Def. 4.40]{BP20}.

\begin{definition}\label{SEGCOLCHAR DEF}
	A preoperad $X \in \mathsf{PreOp}^G$ is called a \emph{(equivariant) Segal operad} if
	$X(\Omega[T]) \to 	X(Sc[T])$
	is a Kan equivalence for each $T \in \Omega_G$.
	A Segal operad $X$ is further called a
	\emph{Reedy fibrant Segal operad} 
	if $\gamma^{\**}X$ is dendroidal Reedy fibrant in $\sdSet^G$.
\end{definition}

\begin{remark}\label{SEGCOLCHAR_REM}
        By \eqref{COLDEC_EQ} % Remarks \ref{COLDEC REM}
	and Remark \ref{SCTCOLPR REM},
	$X \in \mathsf{PreOp}^G$ is a Segal operad iff
        the natural maps
        \begin{equation}% \label{SEGCOLCHAR EQ}
                X_{\mathfrak{c}}(\Omega[T])
                \longto % \xrightarrow{\sim}
		\prod_{v \in \boldsymbol{V}_G(T)}
                X_{\mathfrak{c}_v}(\Omega[T_v])
        \end{equation}
	are Kan equivalences for all 
	$T \in \Omega_G$
	and $G$-equivariant colorings
	$\mathfrak{c} \colon 
	\boldsymbol{E}(T) \to \mathfrak{C}_X = X(\eta)$.

        Additionally,
        by \cite[Rem. 4.41]{BP20}
        a Segal operad $X$ is a Reedy fibrant Segal operad iff
        $\gamma^{\**}X$ is fibrant in the dendroidal Segal space model structure on $\sdSet^G$.
\end{remark}

\begin{remark}\label{REPSEGOPS REM}
Since,
for any preoperad 
$X \in \mathsf{PreOp}^G$
the discrete simplicial sets 
$X(\eta)$ are Kan complexes,
one can form a dendroidal 
fibrant replacement 
$X \to \widetilde{X}$ in 
$\mathsf{sdSet}^G$
such that 
$X(\eta) = \widetilde{X}(\eta)$,
so that $\widetilde{X}$
is again a preoperad.

Moreover, since the maps
$X(\Omega[T]) \to \widetilde{X}(\Omega[T])$ for $T \in \Omega_G$
are Kan equivalences,
\eqref{COLDEC_EQ} % Remark \ref{COLDEC REM}
implies that so are the maps
$X_{\mathfrak{c}}(\Omega[T]) \to \widetilde{X}_{\mathfrak{c}}(\Omega[T])$
for any coloring
$\mathfrak{c} \colon \boldsymbol{E}(T) \to X(\eta)$.

Note that Remark \ref{SEGCOLCHAR_REM} % \eqref{SEGCOLCHAR EQ} 
then implies that $X$ is a Segal operad iff 
$\widetilde{X}$ is a Reedy fibrant Segal operad.
\end{remark}

We will show that joint equivalences 
between Segal operads
admit a Dwyer-Kan type description 
in terms of fully faithfulness and essential surjectivity 
conditions (cf. Theorem \ref{SOPG_THM}).
To describe essential surjectivity,
we need to recall a discrete 
algebraic structure associated to a Segal preoperad.
In the following, we make use of the category
$\mathsf{dSet}_G= \mathsf{Set}^{\Omega_G^{op}}$
of genuine dendroidal sets discussed in \S \ref{EDS_SEC},
as well as its natural simplicial generalization
$\mathsf{sdSet}_G= \mathsf{sSet}^{\Omega_G^{op}}$.

\begin{definition}\label{HOMGENOP DEF}
	Given a Segal operad $X \in \mathsf{PreOp}^G$,
	we define its \emph{homotopy genuine operad} $ho(X) \in \dSet_G$ by
\begin{equation}\label{HOGENOP EQ}
	ho(X) = \pi_0\left(\upsilon_{\**}\gamma^{\**}X\right)
\end{equation}
	with 
	$\upsilon_{\**} \colon \mathsf{sdSet}^G \to \mathsf{sdSet}_G$
	(cf. \eqref{UPSILONADJ EQ})
	and 
	$\pi_0 \colon \mathsf{sdSet}_G \to \mathsf{dSet}_G$
	defined in the natural way. 
\end{definition}

\begin{remark}\label{HOISNERVENON REM}
	When $G=\**$, \eqref{HOGENOP EQ}
	reduces to $ho(X) = \pi_0(\gamma^{\**}X)$
	and the Segal condition for 
	$X \in \mathsf{PreOp}$
	induces a \emph{strict} Segal condition
	on $ho(X) \in \mathsf{dSet}$,
	i.e. $ho(X)$ is the nerve of an operad,
	cf. \eqref{TAUADJ EQ}.
	Moreover, if
	$X \in \mathsf{PreCat} \subset \mathsf{PreOp}$ is a \textit{precategory},
        i.e. $X \in \mathsf{ssSet} \subset \mathsf{sdSet}$ with $X(\eta)$ a discrete simplicial set,
	then $ho(X) \in \mathsf{sSet} \subset \mathsf{dSet}$
	is the nerve of a category.
\end{remark}

\begin{remark}
	The ``genuine operad'' moniker for 
	$ho(X)\in \dSet_G$ refers to the fact that this presheaf satisfies a certain strict Segal condition, as shown in \cite[Prop. 5.9]{BP20}
	(technically the cited result only covers the case of $X$ Reedy fibrant, but it is immediate that for 
	$X,\widetilde{X}$ as in
	Remark \ref{REPSEGOPS REM} 
	there is a natural identification 
	$ho(X) \simeq ho(\widetilde{X})$).

	However, for our present purposes we will not need the full strength of this statement, 
	but only a more familiar consequence.
	Recalling 
	(cf. \eqref{UPSIOTADIAG EQ})
	the inclusion
	$\iota_G \colon \Delta \times \mathsf{O}_G \to \Omega_G$
	given by 
	$([n],G/H) \mapsto G/H \cdot [n]$,
	one has that
	$\iota_G^{\**}ho(X)
	\in \mathsf{sSet}^{\mathsf{O}_G^{op}}$
	is given by
	$\left(\iota_G^{\**}ho(X)\right)(G/H)
	= \iota^{\**} ho(X^H) = 
	ho(\iota^{\**} X^H)$,
	where $\iota^{\**},ho$ are as in
	\eqref{IOTASHADJ EQ},
	Remark \ref{HOISNERVENON REM}.
	As such, the Segal condition for $ho(X)$ implies that 
	$\iota_G^{\**}ho(X)$
	is a coefficient system of 
	nerves of categories \cite[Rem. 5.11]{BP20}.
\end{remark}

\begin{definition}\label{DKEQUIV DEF}
      A map $f \colon X \to Y$ of Segal operads in $\mathsf{PreOp}^G$ is called:
\begin{enumerate}[label = (\roman*)]
	\item \textit{fully-faithful} if,
	for all $G$-corollas $C \in \Sigma_G$ 
	and $G$-equivariant colorings
	$\mathfrak{c} \colon \boldsymbol{E}(C) \to
	\mathfrak{C}_X = X(\eta)$,
	the induced map
\[
	X_{\mathfrak{c}}(\Omega[C]) \to 
	Y_{f\mathfrak{c}}(\Omega[C])
\]
	is a Kan equivalence in $\mathsf{sSet}$;
	\item \textit{essentially surjective} if the map 
      $\iota_G^{\**}ho(X) \to \iota_G^{\**}ho(Y)$ of $G$-coefficient systems of categories
      is levelwise essentially surjective;
	\item a \textit{Dwyer-Kan equivalence} if it is both fully-faithful and essentially surjective.
      \end{enumerate}
\end{definition}

\begin{notation}\label{HEQUIV NOT}
	Arrows in the category (with nerve)
	$\left(\iota_G^{\**}ho(X)\right)(G/H)
	= ho(\iota^{\**}X^H)$
	are encoded by maps
	$\Omega[1] \to X^H$.
	If $\Omega[1] \to X^H$
	encodes an isomorphism,
	we call $\Omega[1] \to X^H$
	a \emph{$H$-equivalence}.
\end{notation}

The following summarizes 
\cite[Thms. 5.51 and 5.48]{BP20},
with the additional fact that the ``further'' claim
holds for all Segal operads, rather than just the Reedy fibrant ones, 
following from Remark \ref{REPSEGOPS REM}.

\begin{theorem}\label{FIBPREOP THM}
	The fibrant objects in $\mathsf{PreOp}^G_{normal}$
	are precisely the Reedy fibrant Segal operads,
	i.e. the preoperads $X \in \mathsf{PreOp}^G$ such that
	$\gamma^{\**} X \in \mathsf{sdSet}^G$ is a Segal space.
	
	Further, a map between Segal operads is a joint equivalence iff it is a Dwyer-Kan equivalence.
\end{theorem}

\subsection{Fibered simplicial (co)tensor on preoperads}
\label{FIBTENS_SEC}

In this section we introduce an auxiliary 
simplicial tensoring on
$\mathsf{PreOp}^G$
that will play a key role in our definition of the tame model structure 
$\mathsf{PreOp}^G_{tame}$ in \S \ref{TAMEDEFEX SEC},
as well as streamline the comparison 
between
$\mathsf{PreOp}^G_{tame}$
and 
$\mathsf{sOp}^G$ in \S \ref{PREOPOPEQUIV SEC}.

We first define the adjoint simplicial cotensoring,
which admits a very simple description in terms
of the $X_{\mathfrak{c}}(U)$ construction introduced in
Notation \ref{XUDECALT NOT}
(and the identification 
\eqref{PREOPCOLFIXEQ EQ}).

\begin{definition}
	Given $X \in \mathsf{PreOp}^{G}_{\mathfrak{C}}$ 
	and $K \in \mathsf{sSet}$
	we define their
	\emph{fiber cotensor} 
	$\left\{K, X \right\}_{\mathsf{\mathfrak{C}_{\bullet}}} \in \mathsf{PreOp}^G_{\mathfrak{C}}$
	by
	\begin{equation}\label{EASYCOTEN EQ}
	\left(
	\left\{ K, X \right\}_{\mathsf{\mathfrak{C}_{\bullet}}}
	\right)_{\mathfrak{c}} (U)
	=
	X_{\mathfrak{c}}(U)^K.
	\end{equation}	
	for $U \in \Omega$ a tree and 
	$\mathfrak{c} \colon \boldsymbol{E}(T) \to \mathfrak{C} =  X(\eta)$
	a coloring.
	
	Alternatively, 
	$\left\{K, X \right\}_{\mathsf{\mathfrak{C}_{\bullet}}}$
	is given by the pullback in $\mathsf{sdSet}^G$
	(where the bottom map is induced by the map $K \to \**$,
	and the left square simply evaluates the right square
	at $U \in \Omega$)
\begin{equation}
\begin{tikzcd}
	\left\{ K, X \right\}_{\mathsf{\mathfrak{C}_{\bullet}}} (U) \ar{r} \ar{d}
&
	X(U)^K \ar{d}
&%%
	\left\{ K, X \right\}_{\mathsf{\mathfrak{C}_{\bullet}}} \ar{r} \ar{d}
&
	X^K \ar{d}
\\
	\prod_{\boldsymbol{E}(U)} X(\eta) \ar{r}
&
	\left(\prod_{\boldsymbol{E}(U)} X(\eta)\right)^K
	\arrow[lu, phantom, "\lrcorner", very near start]
&%%
	\mathsf{csk}_{\eta} X  \ar{r}
&
	\left(\mathsf{csk}_{\eta} X\right)^K
	\arrow[lu, phantom, "\lrcorner", very near start]
\end{tikzcd}
\end{equation}
\end{definition}

\begin{definition}
        \label{FIBERTENSOR_DEF}
	Given $X \in \mathsf{PreOp}^{G}_{\mathfrak{C}}$ 
	and $K \in \mathsf{sSet}$
	we define their
	\emph{fiber tensor}
	$X \otimes_{\mathfrak{C}_{\bullet}} K
	\in \mathsf{PreOp}^G_{\mathfrak{C}}$
	by the pushout in $\mathsf{sdSet}^G$
\begin{equation}\label{PREOPTENS EQ}
\begin{tikzcd}
	\left(\mathsf{sk}_{\eta}X \right) \times K \ar{r} \ar{d} \arrow[dr, phantom, "\ulcorner", very near start]  
&
	\mathsf{sk}_{\eta}X \ar{d}
\\
	X \times K \ar{r} 
& 
	X \otimes_{\mathfrak{C}_{\bullet}} K.
\end{tikzcd}
\end{equation}
More explicitly, 
one has
$(X \otimes_{\mathfrak{C}_{\bullet}} K)(U) = X(U) \times K$
whenever $U$ is a non-linear tree 
(equivalently, $\Omega(U,\eta)=\emptyset$)
and that $(X \otimes_{\mathfrak{C}_{\bullet}} K)([n])$
is given by the following pushout when $U=[n]$ is linear.
\[
\begin{tikzcd}
	X(\eta) \times K \ar{r} \ar{d} \arrow[dr, phantom, "\ulcorner", very near start]  
&
	X(\eta) \ar{d}
\\
	X([n]) \times K \ar{r} 
& 
	(X \otimes_{\mathfrak{C}_{\bullet}} K)([n]) 
\end{tikzcd}
\]
We have an analogous \textit{fibered pushout product} map we'll denote by $(X \to Y) \square_{\mathfrak C_\bullet} (K \to L)$.
\end{definition}

\begin{remark}\label{BYHAND1 REM}
	In \cite[(7.1.5)]{CM13b}
	the objects
	$\Omega[T] \otimes_{\mathfrak{C}_{\bullet}} K$ were denoted 
	$\Omega[K,T]$
	and built by hand.
\end{remark}

\begin{remark}\label{OTIMCON REM}
	If $K \in \mathsf{sSet}$ is connected,
	comparing the left square in \eqref{GAMMASTAR_EQ}
	with the square \eqref{PREOPTENS EQ} yields an identification
	$\gamma_! \left(X \times K\right) \simeq 
	X \otimes_{\mathfrak{C}_{\bullet}} K$.
\end{remark}

\begin{remark}\label{TENSCOADJ REM}
	For each fixed $K \in \mathsf{sSet}$
	the fiber tensor and cotensor 
	determine an adjunction as on the left below.
\[
\begin{tikzcd}[column sep =50]
	\mathsf{PreOp}^G \ar[shift left=1]{r}
	{(-) \otimes_{\mathfrak{C}_{\bullet}} K}
&
	\mathsf{PreOp}^G \ar[shift left=1]{l}
	{\{K,-\}_{\mathfrak{C}_{\bullet}}}
&
	\mathsf{PreOp}^G_{\mathfrak{C}} \ar[shift left=1]{r}
	{(-) \otimes_{\mathfrak{C}_{\bullet}} K}
&
	\mathsf{PreOp}^G_{\mathfrak{C}} \ar[shift left=1]{l}
	{\{K,-\}_{\mathfrak{C}_{\bullet}}}
\end{tikzcd}
\]
Moreover, this adjunction is a fibered adjunction over the color set functor
$\mathfrak{C}_{\bullet} \colon
\mathsf{PreOp}^G \to \mathsf{Set}^G$,
cf. \cite[Def. \ref{OC-FIBADJ DEF}]{BP_FCOP}.
In particular, for each fixed $G$-set of colors
$\mathfrak{C}$
one has a restricted adjunction as on the right above.
Moreover, by 
\cite[Prop. \ref{OC-FIBADJCAR PROP}]{BP_FCOP} and its dual,
these functors are compatible with (co)cocartesian arrows.
I.e., cf. Remark \ref{GROTHFIBOP REM},
one has
$(f_!A) \otimes_{\mathfrak{C}_{\bullet}} K
\simeq
f_!(A \otimes_{\mathfrak{C}_{\bullet}} K)$
and 
$\{K,f^{\**}X\}_{\mathfrak{C}_{\bullet}}
\simeq
f^{\**}(\{K,X\}_{\mathfrak{C}_{\bullet}})$.
\end{remark}

\begin{remark}\label{NOTTWOVARADJ REM}
	Remark \ref{TENSCOADJ REM} implies that the fiber cotensor
	$X \otimes_{\mathfrak{C}_{\bullet}} K$
	preserves colimits on the $X$ variable.
	However, some care is needed when dealing with the 
	$K$ variable.
	For each fixed color $G$-set $\mathfrak{C}$, one has that the functor
\[
\begin{tikzcd}[row sep = 0, column sep = 40pt]
	\mathsf{PreOp}^G_{\mathfrak{C}} \times \mathsf{sSet} \ar{r}{(-)\otimes_{\mathfrak{C}_{\bullet}}(-)} 
&
	\mathsf{PreOp}^G_{\mathfrak{C}}
\end{tikzcd}
\]
is part of a two-variable adjunction,
which in particular means that
$X \otimes_{\mathfrak{C}_{\bullet}} (-) \colon 
\mathsf{sSet} \to \mathsf{PreOp}^G_{\mathfrak{C}}$ 
(where $\mathfrak{C}=X(\eta)$)
preserves colimits.
On the other hand, this means that
$X \otimes_{\mathfrak{C}_{\bullet}} (-) \colon
\mathsf{sSet} \to \mathsf{PreOp}^G$
only preserves those colimits which coincide in
$\mathsf{PreOp}^G$ and $\mathsf{PreOp}^G_{\mathfrak{C}}$,
namely the connected colimits.
On the other hand, for coproducts one instead has that the canonical map
\[
\amalg_i X \otimes_{\mathfrak{C}_{\bullet}} K_i
	\to
X \otimes_{\mathfrak{C}_{\bullet}} (\amalg_i K_i)
\]
is a cocartesian arrow over the fold map
$\nabla \colon 
\amalg_i \mathfrak{C} \to \mathfrak{C}$, 
i.e.
$
\nabla_!
\left(\amalg_i X \otimes_{\mathfrak{C}_{\bullet}} K_i\right)
\simeq
X \otimes_{\mathfrak{C}_{\bullet}} (\amalg_i K_i)
$.
\end{remark}

\begin{remark}\label{COLORTENSGAM REM}
	Let $K \in \mathsf{sSet}$, and
        $X \to Y$ any map in $\mathsf{PreOp}^G$
	which is the identity on colors and 

	Then the top horizontal maps in \eqref{PREOPTENS EQ}
	for $X,Y$ coincide, 
	and likewise for the 
	left square in \eqref{GAMMASTAR_EQ} for
	$X \times K, Y \times K$.
	It thus follows that the squares below are pushout squares in $\mathsf{sdSet}^G$.
\[
\begin{tikzcd}
	X \times K \ar{r} \ar{d} 
	\arrow[dr, phantom, "\ulcorner", very near start] 
&
	\gamma_! \left( X \times K \right) \ar{r} \ar{d} 
	\arrow[dr, phantom, "\ulcorner", very near start] 
&
	X \otimes_{\mathfrak{C}_{\bullet}} K \ar{d}
\\
	Y \times K \ar{r} 
&
	\gamma_! \left( Y \times K \right) \ar{r} 
&
	Y \otimes_{\mathfrak{C}_{\bullet}} K
\end{tikzcd}
\]
\end{remark}

\begin{lemma}\label{OTIMSETPUSH LEM}
	Let $f\colon X \to Y$ be a map in $\mathsf{PreOp}^G$
	and $k \colon K \to L$ be a map in $\mathsf{sSet}$.

	Then $\gamma^{\**}\left(f \square_{\mathfrak{C}_{\bullet}} k\right)$
	is a pushout in $\mathsf{sdSet}^G$
	of the map
\[
	\left(f_! X \to Y\right)
	\square
	\left( K \to L \right).
\]
\end{lemma}

\begin{proof}
Since the left square below is a pushout square
(this follows by noting that the 
horizontal arrows are cocartesian, as per 
Remark \ref{TENSCOADJ REM};
see the universal property in Remark \ref{GROTHFIBOP REM})
\[
\begin{tikzcd}
	X \otimes_{\mathfrak{C}_{\bullet}} K \ar{r} \ar{d} 
	\arrow[dr, phantom, "\ulcorner", very near start]
&
	f_! X \otimes_{\mathfrak{C}_{\bullet}} K \ar{r} \ar{d} 
&
	Y \otimes_{\mathfrak{C}_{\bullet}} K \ar{d}
\\
	X \otimes_{\mathfrak{C}_{\bullet}} L \ar{r} 
&
	f_! X \otimes_{\mathfrak{C}_{\bullet}} L \ar{r} 
&
	Y \otimes_{\mathfrak{C}_{\bullet}} L
\end{tikzcd}
\]
one has
$(X \to Y) \square_{\mathfrak{C}_{\bullet}} k 
\simeq 
(f_!X \to Y) \square_{\mathfrak{C}_{\bullet}} k$.
Since $f_!X \to Y$ is the identity on colors,
Remark \ref{COLORTENSGAM REM} then says that 
the two middle squares below are pushout squares.
\[
\begin{tikzcd}
	f_! X \times K 
	\ar{r} \ar{d} 
&
	f_! X \times L
	\ar{r} \ar{d} 
	\arrow[dr, phantom, "\ulcorner", very near start] 
&
	f_! X \otimes_{\mathfrak{C}_{\bullet}} L
	\ar{d}
&
	f_! X \times K \ar{r} \ar{d} 
	\arrow[dr, phantom, "\ulcorner", very near start] 
&
	f_! X \otimes_{\mathfrak{C}_{\bullet}} K 
	\ar{r} \ar{d} 
&
	f_! X \otimes_{\mathfrak{C}_{\bullet}} L
	\ar{d}
\\
	Y \times K 
	\ar{r} 
&
	Y \times L
	\ar{r}
&
	Y \otimes_{\mathfrak{C}_{\bullet}} L
&
	Y \times K 
	\ar{r} 
&
	Y \otimes_{\mathfrak{C}_{\bullet}} K
	\ar{r}
&
	Y \otimes_{\mathfrak{C}_{\bullet}} L
\end{tikzcd}
\]
A standard argument
(e.g. \cite[Obs. 5.1]{RV14})
then shows that the pushout map for the rightmost square above is a pushout of the pushout map for the leftmost square, 
finishing the proof.	
\end{proof}

\subsection{Definition and existence of the tame model structure}
\label{TAMEDEFEX SEC}

The following adapts \cite[\S 7.7]{CM13b},
repackaged using the fibered cotensoring
$\otimes_{\mathfrak{C}_{\bullet}}$.

\begin{definition}\label{TAMEGENCOF DEF}
	The \emph{tame cofibrations} in $\mathsf{PreOp}^G$
	are the saturation of the following maps
	\begin{itemize}
		\item[(TC1)] $G/H \cdot \left(\emptyset \to\Omega[\eta]\right)$ for $H\leq G$;
		\item[(TC2)] 
		$\Omega[C] \otimes_{\mathfrak{C}_{\bullet}} \left(\partial \Delta[n] \to \Delta[n]\right)$ for 
		$G$-corollas $C \in \Sigma_G$, $n \geq 0$;
		\item[(TC3)] 
		$\left( Sc[T] \to \Omega[T] \right) 
		\square_{\mathfrak{C}_{\bullet}} 
		\left(\partial \Delta[n] \to \Delta[n]\right)$ for 
		$G$-trees $T \in \Omega_G$, $n \geq 0$.
	\end{itemize}
\end{definition}

\begin{definition}\label{PSEUINT DEF}
	A precategory $I \in \mathsf{PreCat}\simeq \mathsf{PreOp} \downarrow \Omega[\eta]$ 
	is called a \emph{pseudo-interval}
	if $I(\eta) = \{0,1\}$,
	the map 
	$\Omega[\eta] \amalg \Omega[\eta]
	= \mathsf{sk}_{\eta} I \to I$
	is a tame cofibration,
	and the map $I \to \Omega[\eta]$ is a weak equivalence. % TODO WHAT TYPE OF WEAK EQUIVALENCE IS THIS?
\end{definition}

\begin{definition}\label{TAMEGENANO DEF}
	The \emph{tame anodyne cofibrations} in $\mathsf{PreOp}^G$ 
	are the saturation of the following maps
	\begin{itemize}
		\item[(TA1)] $G/H \cdot 
		\left(\Omega[\eta] \to I \right)$ for $H \leq G$
		and $I \in \mathsf{PreCat}$
		a countable pseudo-interval;
		\item[(TA2)] $\Omega[C] \otimes_{\mathfrak{C}_{\bullet}} \left(\Lambda^i[n] \to \Delta[n]\right)$ for $C \in \Sigma_G$, $0 \leq i \leq n$, $n \geq 1$;
		\item[(TA3)] 
		$\left( Sc[T] \to \Omega[T] \right) 
		\square_{\mathfrak{C}_{\bullet}}
		\left(\partial \Delta[n] \to \Delta[n]\right)$ for $T \in \Omega_G$, $n \geq 0$.
	\end{itemize}
\end{definition}

We can now state the main result of this section.

\begin{theorem}[{cf. \cite[Thm. 7.19]{CM13b}}]
      \label{TAMEMS_THM}
	There is a left proper model structure on 
	$\mathsf{PreOp}^G$,
	called the \emph{tame model structure},
	such that:
	\begin{itemize}
		\item weak equivalences are the 
		joint equivalences 
		(i.e. detected by inclusion into $\mathsf{sdSet}^G$);
		\item the generating cofibrations are the maps (TC1),(TC2),(TC3);
		\item $X \in \mathsf{PreOp}^G$ is fibrant iff
		$X \to \**$ has the right lifting property against 
		(TA1),(TA2),(TA3);
		\item a map $X \to Y$ between fibrant objects is a fibration iff
		it has the right lifting property against 
		(TA1),(TA2),(TA3).
	\end{itemize}
	Moreover, the identity adjunction
	$
	\mathsf{PreOp}^G_{tame} 
	\rightleftarrows
	\mathsf{PreOp}^G_{normal} 
	$
	is a Quillen equivalence.
\end{theorem}

Before proving Theorem \ref{TAMEMS_THM},
we collect a few lemmas.

\begin{lemma}\label{TAMECOFCOF_LEM}
	Tame cofibrations (resp. tame anodyne cofibrations) are  cofibrations (resp. trivial cofibrations) in the normal model structure on $\mathsf{PreOp}^G$.
\end{lemma}

\begin{proof}
	It suffices to check the given claims for the generating maps
	in Definitions \ref{TAMEGENCOF DEF}, \ref{TAMEGENANO DEF}.
	
	The (TC1) case is immediate.
	For (TC2),(TC3),(TA2),(TA3)
	we apply 
	Lemma \ref{OTIMSETPUSH LEM}
	(note that for (TC2),(TA2)
	the map $f\colon X \to Y$ is $\emptyset \to \Omega[C]$,
	so that $f_!X \to Y$ is the inclusion
	$\partial \Omega[C] \to \Omega[C]$)
	and in all such cases it is 
	clear by Remark \ref{SQUAREEQUI REM} that the 
	corresponding map $(f_!X \to Y) \square k$
	is a (trivial) cofibration in $\mathsf{sdSet}^G$.
	(TA1) follows by definition of pseudo-interval
	and the 
	(TC1),(TC2),(TC3) cases.
\end{proof}

\begin{lemma}\label{TAMETRIVFIB LEM}
	Any map $X \to Y$ which has the right lifting property against 
	(TC1),(TC2),(TC3)
	is a joint equivalence in 
	$\mathsf{PreOp}^G$.
\end{lemma}

\begin{proof}
Writing $f \colon \mathfrak{C} \to \mathfrak{D}$ for the underlying map of colors,
consider the factorization $X \to f^{\**}Y \to Y$.
Noting that lifting problems against (TC1) depend only on objects and both of (TC2) and (TC3) consist of maps which are identities on objects,
we see that $X \to Y$ has the right lifting property against (TC1) iff 
$f^{\**} Y \to Y$ does
and the right lifting property against 
(TC2),(TC3) iff $X \to f^{\**}Y$ does.
We argue separately that 
$f^{\**} Y \to Y$ and $X \to f^{\**}Y$
are joint equivalences.

Consider first the map $f^{\**} Y \to Y$. Note now that $f^{\**} Y \to Y$ has the right lifting proper against all maps 
$\left(\partial \Omega[T] \to \Omega[T] \right) \times \Delta[n]$.
Indeed, if $T \simeq G/H \cdot \eta$ is a stick $G$-tree,
this is precisely the lifting condition against (TC1), and otherwise it follows automatically since $\left(\partial \Omega[T] \to \Omega[T] \right) \times \Delta[n]$ is the identity on objects.
Therefore, the levels 
$\left(f^{\**} Y \right)_n \to Y_n$ are trivial fibrations in 
$\mathsf{dSet}^G$, showing that 
$f^{\**} Y \to Y$ is a dendroidal equivalence, 
and thus a joint equivalence, 
cf. Theorem \ref{JB_THM}\ref{SDEQUIV_LBL}.

Consider now the map $X \to f^{\**} Y$.
The lifting property against (TC2) 
together with the decompositions in
\eqref{COLDEC_EQ} % Remark \ref{COLDEC REM}
then say that the maps
$X_{\mathfrak{c}}(\Omega[C]) \to 
(f^{\**} Y)_{\mathfrak{c}} (\Omega[C])$
are trivial Kan fibrations for all $G$-corollas $C \in \Sigma_G$
and colorings $\mathfrak{c} \colon \boldsymbol{E}(C) \to \mathfrak{C}$.
Now let $T \in \Omega_G$ be a $G$-tree and 
$\mathfrak{c} \colon \boldsymbol{E}(T) \to \mathfrak{C}$
be a coloring,
and consider the following diagram.
\[
\begin{tikzcd}
X_{\mathfrak{c}}(\Omega[T]) \ar{r} \ar[->>]{d}{\sim}
&
X_{\mathfrak{c}}(Sc[T]) \ar{r}{\simeq} \ar[->>]{d}{\sim}
&
\prod_{v \in \boldsymbol{V}_G(T)} 
X_{\mathfrak{c}_v}(\Omega[T_v])
\ar[->>]{d}{\sim}
\\
(f^{\**} Y)_{\mathfrak{c}}(\Omega[T]) \ar{r}
&
(f^{\**} Y)_{\mathfrak{c}}(Sc[T]) \ar{r}{\simeq} 
&
\prod_{v \in \boldsymbol{V}_G(T)} 
(f^{\**} Y)_{\mathfrak{c}_v}
(\Omega[T_v]) 
\end{tikzcd}
\]
By the above, the right vertical map is a trivial Kan fibration,
and thus so is the isomorphic map
$X_{\mathfrak c}(Sc[T]) \to (f^{\**} Y)_{\mathfrak{c}} (Sc[T])$,
and likewise for the total map
$X(Sc[T]) \to f^{\**} Y (Sc[T])$.
As the lifting property against (TC3)
yields that
$
X(\Omega[T]) \overset{\sim}{\twoheadrightarrow}
f^{\**}Y(\Omega[T]) \times_{f^{\**}Y(Sc[T])} X(Sc[T]) 
$
is a trivial Kan fibration for all $G$-trees,
so is $X(\Omega[T]) 
\overset{\sim}{\twoheadrightarrow} f^{\**} Y (\Omega[T])$,
showing that $X \to f^{\**} Y$ is a simplicial equivalence, and thus a joint equivalence,
cf. Theorem \ref{JB_THM}\ref{SDEQUIV_LBL}.
\end{proof}

\begin{remark}\label{TC2TC3REP REM}
	Tame cofibrant replacement in $\mathsf{PreOp}^G$
	can be performed without changing objects.
	Indeed, given any 
	$A \in \mathsf{PreOp}^G$,
	one has that
	$\mathsf{sk}_{\eta} A = 
	\coprod_{A(\eta)} \Omega[\eta]$
	is tame cofibrant by (TC1).
	Thus, the small object argument for (TC2),(TC3) 
	applied to the map 
	$\mathsf{sk}_{\eta} A \to A$
	gives a factorization
	$\mathsf{sk}_{\eta} A \to \widetilde{A} \to A$
	where:
	$\mathsf{sk}_{\eta} A \to \widetilde{A}$
	is in the saturation of (TC2),(TC3), so that
	$\widetilde{A}$ is tame cofibrant;
	$\widetilde{A} \to A$
	has the right lifting property against 
	(TC2),(TC3) by construction and against (TC1)
	since it is the identity on objects, 
	and is thus a joint equivalence by 
	Lemma \ref{TAMETRIVFIB LEM}.
\end{remark}

\begin{lemma}\label{SLIMOD LEM}
	Let $X \in \mathsf{PreOp}^G$
	be a Segal operad,
	and $\Omega[1] \to X^H$
	be a $H$-equivalence
	(Notation \ref{HEQUIV NOT})
        for some $H \leq G$.
	Then there exists a countable
	pseudo-interval $I$
	and factorization
	$\Omega[1] \to I \to X^H$.
\end{lemma}

\begin{proof}
	Recall (Remark \ref{REPSEGOPS REM})
	that one can find a simplicial equivalence
	$X \xrightarrow{\sim} \widetilde{X}$
	with $\widetilde{X}$ fibrant in the normal model structure on 
	$\mathsf{PreOp}^G$.
	Since, cf. Theorem \ref{FIBPREOP THM},
	Reedy fibrant Segal preoperads are in particular
	Segal spaces, 
	it is well known that,
	for $J$ the nerve of the contractible groupoid
	with two objects,
	one has a dashed arrow as on the left below
	(this is originally due to 
	Rezk \cite[Thm. 6.2]{Rez01},
	applied to $\iota^{\**} \widetilde{X}^H$; 
	alternatively, see \cite[Prop. 5.26(iv)]{BP20}).
\[
\begin{tikzcd}
	\Omega[1] \ar{r} \ar{d} 
&
	X^H \ar{d}{\sim}
&&
	\Omega[1] \ar[dashed]{r}
&
	J' \ar[dashed,twoheadrightarrow]{r} \ar[dashed]{d}[swap]{\sim} 
&
	X^H \ar{d}{\sim}
\\
	J \ar[dashed]{r}{\sim}
&
	\widetilde{X}^H
&&&
	\widetilde{J} \ar[twoheadrightarrow]{r}
&
	\widetilde{X}^H
	\arrow[lu, phantom, "\lrcorner", very near start]
\end{tikzcd}
\]	
Writing 
$J \xrightarrow{\sim} 
\widetilde{J} \twoheadrightarrow 
\widetilde{X}^H$
for the ``trivial cofibration followed by fibration''
factorization in the normal model structure
on $\mathsf{PreOp}$,
we now form the right diagram above, 
where the square is a pullback.
Here we note that
$X^H \to \widetilde{X}^H$
is a simplicial equivalence,
i.e. the maps
$X^H(T) \to \widetilde{X}^H(T)$
are Kan equivalences for each $T \in \Omega$,
while the maps 
$\widetilde{J}(T) \to \widetilde{X}^H(T)$
are Kan fibrations.
Hence, 
since $\mathsf{sSet}$ is (right) proper, 
$J' \to \widetilde{J}$ is again a simplicial equivalence.

By construction, the canonical map 
$J' \to \Omega[\eta]$
is a simplicial equivalence, 
but to obtain the required 
countability and the tame cofibrancy condition
for $J'$ to be a pseudo-interval 
(Definition \ref{PSEUINT DEF}),
we will need to replace $J'$.
Firstly, a countable replacement can be obtained by adapting
either the argument
between Lemmas 4.2 and 4.3 of
\cite{Ber07}
or the more refined argument in the proof of 
\cite[Lemma 5.1.7]{HSS}.
Briefly, since the spaces $J'(\vect{T})$
are all contractible,
one may build nested
countable subpresheaves
$I'^{,n} \subseteq J'$ as in
\[
\Omega[1] = 
I'^{,0} \subseteq
I'^{,1} \subseteq
I'^{,2} \subseteq
\cdots \subseteq
J'
\]
such that all maps 
$I'^{,n}(\vect{T}) \to I'^{,n+1}(\vect{T})$
are nullhomotopic
(informally, and given a countable $I'^{,n}$, 
one needs only countably many simplices of $J'$
to kill of the homotopy groups of $I'^{,n}$; 
hence by adding those simplices and closing under the presheaf operations one obtains $I'^{,n+1}$).
Thus, setting $I' = \bigcup_{n} I'^{,n}$
we still have that $I' \to \Omega[\eta]$
is a simplicial equivalence, 
but $I'$ is now countable.

Lastly, the small object argument for 
(TC2),(TC3) applied to 
$\Omega[1] \to I'$
gives a factorization 
$\Omega[1] \to I \to I'$
where $I \to I'$ is a joint equivalence
(see the argument in Remark \ref{TC2TC3REP REM}),
and the countable preoperad $I$
now has the tame cofibrancy property required 
to be a pseudo-interval. % (Definition \ref{PSEUINT DEF}).
\end{proof}

\begin{proof}[Proof of Theorem \ref{TAMEMS_THM}]
	Note that, assuming the existence of the tame model structure,
	the ``moreover'' and ``left proper'' claims
	follow from Lemma \ref{TAMECOFCOF_LEM},
	saying that tame cofibrations are normal cofibrations,
	and the fact that
	$\mathsf{PreOp}^{G}_{tame}$, 
	$\mathsf{PreOp}^{G}_{normal}$
	both have joint equivalences as weak equivalences. % TODO GO FROM HERE

	To show the existence claims, we will verify conditions C1,C2,C3,C4,C5 in 
	\cite[Prop. 2.3]{Sta14},
	which is a variation of 
	J. Smith's theorem \cite[Thm. 1.7]{Bek00}
	that includes a further criterion for detecting 
	fibrations between fibrant objects.
	In the remainder of the proof,
	we write $\mathcal{I}$ (resp. $\mathcal{J}$)
	for the union of the sets 
	(TC1),(TC2),(TC3)
	(resp. (TA1),(TA2),(TA3)),
	and $\mathcal{W}$ for the class of weak equivalences,
	so that notions such as 
	$\mathcal{I}$-cof, $\mathcal{J}$-fibrant
	have the same meaning as in \cite{Sta14}.
	
	$\mathsf{PreOp^G}$ is certainly locally presentable, 
	as it is a presheaf category.
	That the weak equivalences in $\mathsf{PreOp^G}$ are accessible follows since they are the preimage by $\gamma^{\**}$ of the weak equivalences in 
	$\mathsf{sdSet}^G$ 
	(see \cite[Cor. A.2.6.5]{Lur09} and \cite[Cor. A.2.6.6]{Lur09}).
	
	Conditions C1 and C3 therein are equivalent to the $2$-out-of-$6$
	condition for weak equivalences, and are thus inherited 
	from $\mathsf{sdSet}^G$.
	Moreover, C2 has already been verified in Lemma \ref{TAMETRIVFIB LEM}.

We next check C4. 
Note first that the maps in 
$\mathcal{I} \text{-cof} \cap \mathcal{W}$
are closed under pushout and transfinite composition, as they are trivial cofibrations in the normal model structure in 
$\mathsf{PreOp}^G$,
so that C4 needs only be checked for the maps in (TA1),(TA2),(TA3) themselves, rather than their saturation.
The case of maps in (TA1) is tautological, by definition
of pseudo-interval.
The fact that the maps in 
(TA2),(TA3) are in the saturation of (TC2),(TC3) is clear, and the fact that these maps are weak equivalences follows from 
Lemma \ref{TAMECOFCOF_LEM}.

Lastly, we check C5.
The lifting condition against (TA3) says that 
$\mathcal{J}$-fibrant objects are such that the maps $X(\Omega[T]) \to X(Sc[T])$
are trivial Kan fibrations, and thus that such $X$ are Segal operads.
Therefore, by the ``further'' statement in
Theorem \ref{FIBPREOP THM} 
it suffices to check that $\mathcal{J}$-fibrations
between Segal operads which are also DK equivalences have the right lifting property against the maps in (TC1),(TC2),(TC3).
Given $X \to Y$ a $\mathcal{J}$-fibration with 
$\mathcal{J}$-fibrant $Y$,
the lifting property against (TC3) is tautological since 
(TC3) equals (TA3).
Next, the lifting property against (TA2) says that the maps
$X(\Omega[T]) \to f^{\**} Y(\Omega[T])$
are Kan fibrations, and the DK condition says that these are Kan equivalences,
so that such maps have the right lifting property against (TC2).
Lastly, given any lifting problem against a map in (TC1),
essential surjectivity (Definition \ref{DKEQUIV DEF}(ii))
 and Lemma \ref{SLIMOD LEM}
produce a lifting problem against a map in (TA1) which has a solution, providing a solution to the original problem.
This finishes the proof.	
\end{proof}

For later use, we record the following.

\begin{lemma}\label{OMEGATTAME_LEM}
	For all $T \in \Omega_G$, the objects $Sc[T],\Omega[T]$ are tame cofibrant in $\mathsf{PreOp}^G$.
\end{lemma}

\begin{proof}
	The case of $Sc[T]$ follows from the pushout below, 
	where $\coprod_{\boldsymbol{E}(T)}\Omega[\eta]$
	is tame cofibrant by (TC1) and the left vertical map is 
	a tame cofibration by (TC2) with $n=0$.
	\[
	\begin{tikzcd}
	\displaystyle{
		\coprod_{v \in \boldsymbol{V}_G(T)} \partial\Omega[T_{v}]
	}
	\arrow[d] \arrow[r]
	&
	\displaystyle{
		\coprod_{\boldsymbol{E}(T)} \Omega[\eta]
	}
	\arrow[d]
	\\
	\displaystyle{
		\coprod_{v \in \boldsymbol{V}_G(T)} \Omega[T_{v}]
	}
	\arrow[r]
	&
	Sc[T]
	\end{tikzcd}
	\]
	The case of $\Omega[T]$
	follows since (TC3) with $n=0$
	says that $Sc[T] \to \Omega[T]$ is a tame cofibration.
\end{proof}

\section{The Quillen equivalences}\label{QE_SEC}

This section establishes our main result, Theorem \ref{QE THM},
up to the key Lemma \ref{KEYPRVAR LEM}, 
to be shown in \S \ref{KEYRES SEC}.

We do so in two steps.
After recalling the Dwyer-Kan model structure on operads $\mathsf{sOp}^G$ in \S \ref{GSOP_SEC},\S \ref{NERTENSCOOP SEC}
we show in Theorem \ref{PREQUIEQUIV THM} in \S \ref{PREOPOPEQUIV SEC},
that the nerve functor
$N \colon \mathsf{sOp}^G
 \to 
\PreOp^G_{tame}$
is a right Quillen equivalence,
where $\PreOp^G$
has the tame model structure from \S \ref{TAMEDEFEX SEC}.
Using \eqref{ADJSQ EQ},
this yields two zigzags of Quillen adjunctions between $\sOp^G$ and the joint model structure on $\sdSet^G$,
and the proof concludes in \S \ref{PFMNTHM SEC}
by showing that the two associated derived composites agree up to joint equivalence,
cf. \eqref{BIGZIG EQ}.

\subsection{Equivariant simplicial operads}
\label{GSOP_SEC}

We write $\mathsf{sOp}^G = \mathsf{Op}^{G}(\mathsf{sSet})$
for the 
category of \textit{$G$-equivariant colored simplicial operads}.
This category admits several descriptions:
as the algebras for a \emph{composition product $\circ$};
as the algebras for a \emph{free operad monad} $\mathbb{F}$;
as the subcategory of preoperads satisfying a \emph{strict Segal condition}.

Our primary goal in this section and the following is to recall (and repackage) the model structure on $\mathsf{sOp}^G$ built in 
\cite[Thm. \ref{AC-THMA}]{BP_ACOP}.
The work in loc. cit. is based on the free operad monad perspective on 
$\mathsf{sOp}^G$, which is technically involved,
but this paper needs only a brief overview of that perspective.
Instead, and in preparation for the proof of the Quillen equivalence
$\mathsf{PreOp}^G \rightleftarrows \mathsf{sOp}^G$
in \S \ref{PREOPOPEQUIV SEC},
we will find it useful in this section to also use the 
strict Segal condition perspective,
which also plays a key role in Appendix \ref{HGEO AP}
and the side paper \cite{BP_WCONS}.

We start by recalling the category 
$\Sigma_{\mathfrak{C}}$
of $\mathfrak{C}$-corollas (cf. Definition \ref{CFOREST_DEF}).
An object of $\Sigma_{\mathfrak{C}}$
is given by a 
$\mathfrak{C}$-colored corolla as on the left below.
Then letting $n$ be the number of leaves of $\vect{C}$,
which we call the \emph{arity of $\vect{C}$},
and $\sigma \in \Sigma_n$,
the picture below depicts a 
generic map in $\Sigma_{\mathfrak{C}}$.
\begin{equation}\label{CSYM EQ2}
\begin{tikzpicture}
[grow=up,auto,level distance=2.3em,every node/.style = {font=\footnotesize},dummy/.style={circle,draw,inner sep=0pt,minimum size=1.75mm}]

\node at (0,0) [font=\normalsize]{$\vect{C}$}
child{node [dummy] {}
	child{
		edge from parent node [swap,near end] {$\mathfrak c_n$} node [name=Kn] {}}
	child{
		edge from parent node [near end] {$\mathfrak c_1$}
		node [name=Kone,swap] {}}
	edge from parent node [swap] {$\mathfrak c_0$}
};
\draw [dotted,thick] (Kone) -- (Kn) ;
\node at (5,0) [font=\normalsize] {$\vect{C} \sigma^{-1}
	$}
child{node [dummy] {}
	child{
		edge from parent node [swap,near end] {$\mathfrak c_{\sigma^{-1}(n)}$} node [name=Kn] {}}
	child{
		edge from parent node [near end] {$\mathfrak c_{\sigma^{-1}(1)}$}
		node [name=Kone,swap] {}}
	edge from parent node [swap] {$\mathfrak c_0$}
};
\draw [dotted,thick] (Kone) -- (Kn) ;

\draw[->] (1.5,0.8) -- node{$\sigma$} (3,0.8);
\end{tikzpicture}
\end{equation}
Alternatively,
$\mathfrak{C}$-corollas can be represented simply as
strings in $\mathfrak{C}$,
which we call $\mathfrak{C}$-profiles \footnote{These were called $\mathfrak C$-signatures in \cite{BP_ACOP,BP_FCOP}.}
(cf. Remark \ref{MAPSPTRANS REM}).
In profile notation, 
\eqref{CSYM EQ2} then becomes
\begin{equation}\label{CSYM EQ1}
\vect{C} =
(\mathfrak c_1, \dots, \mathfrak c_n; \mathfrak c_0) \xrightarrow{\sigma} (\mathfrak c_{\sigma^{-1}(1)}, \dots, \mathfrak c_{\sigma^{-1}(n)}; \mathfrak c_0)
= \vect{C} \sigma^{-1}.
\end{equation}
For this reason, we also refer to 
$\Sigma_{\mathfrak{C}}$ as the 
$\mathfrak{C}$-symmetric category,

Lastly, we note that the notation
$\vect{C} \sigma^{-1}$ used above
comes from the natural right action of $\Sigma_n$
on $\mathfrak{C}$-profiles of arity $n$ via
$
(\mathfrak c_1, \dots, \mathfrak c_n; \mathfrak c_0) \sigma
=
(\mathfrak c_{\sigma(1)}, \dots, \mathfrak c_{\sigma(n)}; \mathfrak c_0)$.

\begin{definition}\label{SSYM DEF}
	The category $\mathsf{sSym}$ of \textit{simplicial symmetric sequences} has:
\begin{itemize}
	\item objects given by pairs
	$(\mathfrak C, X)$ with $\mathfrak{C} \in \mathsf{Set}$
	a set of colors and
	$\Sigma_{\mathfrak C}^{op} \xrightarrow{X}\sSet$
	a functor; 
	\item maps
	$(\mathfrak{C},X) \to (\mathfrak{D},Y)$
	given by maps of colors
	$f \colon \mathfrak{C} \to \mathfrak{D}$
	and $\Phi$ as below.
\begin{equation}
\begin{tikzcd}[row sep = tiny, column sep = 45pt]
	\Sigma_{\mathfrak{C}}^{op} \arrow[dr, "X"{name=U}] 
	\arrow{dd}[swap]{f}
\\
	& \mathsf{sSet}
\\
	|[alias=V]| \Sigma_{\mathfrak{D}}^{op} \arrow[ur, "Y"']
\arrow[Leftarrow, from=V, to=U,shorten >=0.25cm,shorten <=0.25cm, swap,"\Phi"]
\end{tikzcd}
\end{equation}
\end{itemize} 
\end{definition}

More explicitly,
by specifying \eqref{CSYM EQ2} to a map
$\vect{C} \sigma \to \vect{C}$ in $\Sigma_{\mathfrak{C}}$
(note that $\vect{C} = (\vect{C}\sigma)\sigma^{-1}$),
one has that a symmetric sequence
$X \colon \Sigma_{\mathfrak{C}}^{op} \to \mathsf{sSet}$
has structure maps
\begin{equation}\label{SSYMSTMAP EQ}
	X(\vect{C}) = 
	X(\mathfrak c_1, \dots, \mathfrak c_n; \mathfrak c_0) \xrightarrow{\simeq} 
	X(\mathfrak c_{\sigma(1)}, \dots, \mathfrak c_{\sigma(n)}; \mathfrak c_0) =	
	X(\vect{C}\sigma)
\end{equation}
for $\vect{C}$ of arity $n$ and $\sigma \in \Sigma_n$.

We next describe operads.
We write 
$\mathsf{lr} \colon 
\Omega_{\mathfrak{C}} \to \Sigma_{\mathfrak{C}}$,
which we call the \emph{leaf-root} functor,
for the assignment
$\vect{T} = (T,\mathfrak{c})
\mapsto
(\mathfrak{c}(l_1),\cdots,\mathfrak{c}(l_n);
\mathfrak{c}(r))$
where the $l_i$ are the leaves of $T$ and $r$ the root.
The free operad monad $\mathbb{F}$
is then a monad on $\mathsf{sSym}$ which,
when evaluated on $X \colon \Sigma_{\mathfrak{C}}^{op} \to \mathsf{sSet}$,
is given by the left Kan extension below
(the description of the monad structure is found in 
\cite[Defs. \ref{OC-FREEOP DEF},\ref{OC-NCOLOR DEF},\ref{OC-COLORMON_DEF}]{BP_FCOP}).
\begin{equation}\label{FREEOP_EQ}
\begin{tikzcd}[column sep = 150pt]
	\Omega^{0,op}_{\mathfrak{C}}
	\arrow[d, "\mathsf{lr}^{op}"']
	\arrow[r, 
	"\vect{T} \mapsto \prod_{v \in \boldsymbol{V}(T)} X(\vect{T}_v)"{name=U}]
&
	\mathsf{sSet}
\\
	|[alias=V]|
	\Sigma^{op}_{\mathfrak{C}}
	\arrow[ur, "\Lan = \mathbb F X"']
	\arrow[Leftarrow, from=V, to=U,shorten >=0.25cm,shorten <=0.25cm]
\end{tikzcd}
\end{equation}
The category 
$\mathsf{sOp}$ of \emph{colored simplicial operads}
can then be described as the category of 
fiber $\mathbb{F}$-algebras on $\mathsf{sSym}$
(where fiber algebras are those algebras for which the multiplication
$\mathbb{F}\mathcal{O} \to \mathcal{O}$
preserves colors, cf. \cite[Def. \ref{OC-FIBMON DEF}]{BP_FCOP}).
For $G$ a finite group, we then write
$\mathsf{sOp}^G$ (resp. $\mathsf{sSym}^G$)
for the category of 
$G$-objects on 
$\mathsf{sOp}$ (resp. $\mathsf{sSym}$)
which we call the category of
\emph{$G$-equivariant colored simplicial operads
(symmetric sequences)}.
Note that, by an abstract argument \cite[Prop. \ref{OC-DIAGRAMFM_PROP}]{BP_FCOP},
$\mathbb{F}$
induces a monad on $\mathsf{sSym}^G$
whose category of fiber algebras is $\mathsf{sOp}^G$.

Mirroring \eqref{COLORSET EQ},
we then have \emph{color set functors}
\begin{equation}\label{OPERCOLFUN EQ}
	\mathsf{sSym} \xrightarrow{\mathfrak{C}_{\bullet}} \mathsf{Set}
\qquad
	\mathsf{sOp} \xrightarrow{\mathfrak{C}_{\bullet}} \mathsf{Set}
\qquad
	\mathsf{sSym}^G \xrightarrow{\mathfrak{C}_{\bullet}} \mathsf{Set}^G
\qquad
	\mathsf{sOp}^G \xrightarrow{\mathfrak{C}_{\bullet}} \mathsf{Set}^G.
\end{equation}

\begin{remark}\label{CTTCOLSET REM}
	Replacing $\mathsf{sSet}$ with 
	$\mathsf{Set}$
	in Definition \ref{SSYM DEF}
	and \eqref{FREEOP_EQ}
	one recovers the analogue non-simplicial categories
	$\mathsf{Sym}$, 
	$\mathsf{Op}$, 
	$\mathsf{Sym}^G$, 
	$\mathsf{Op}^G$. 
	As is well known, there is then a fully faithful inclusion
	$\mathsf{sOp} \subset \mathsf{Op}^{\Delta^{op}}$
	as those simplicial objects with a constant set of colors.
\end{remark}

The color set functors above are all 
Grothendieck fibrations and, moreover, the monad 
$\mathbb{F}$ is suitably compatible with these fibrations.
In \cite[\S \ref{OC-ECO_SEC}]{BP_FCOP} the 
Grothendieck fibration perspective 
is used to describe the fibers 
$\mathsf{sSym}^G_{\mathfrak{C}}$, 
$\mathsf{sOp}^G_{\mathfrak{C}}$
of those objects with a fixed $G$-color set $\mathfrak{C}$
and maps which are the identity on colors. 
However, here we will be able to take a more direct approach.

If $\mathfrak{C}$ is a $G$-set of colors, 
one has a left $G$-action on the set of $\mathfrak{C}$-profiles.
Thus, if $X \in \mathsf{sSym}^G$
has color set $\mathfrak{C}_X =\mathfrak{C}$,
one has, generalizing
\eqref{SSYMSTMAP EQ},
that $X$ has structure maps
\begin{equation}\label{SSYMGSTMAP EQ}
X(\vect{C}) = 
X(\mathfrak c_1, \dots, \mathfrak c_n; \mathfrak c_0) \xrightarrow{\simeq} 
X(g\mathfrak c_{\sigma(1)}, \dots, g \mathfrak c_{\sigma(n)}; g \mathfrak c_0) =	
X(g \vect{C}\sigma)
\end{equation}
for $\vect{C}$ a $\mathfrak{C}$-profile of arity $n$
and
$(g,\sigma) \in G \times \Sigma^{op}_n$.

Note that, implicit in the $g \vect{C} \sigma$ notation in
\eqref{SSYMGSTMAP EQ}
is the fact that $G \times \Sigma_n^{op}$
has a left action on 
$\mathfrak{C}$-profiles of arity $n$.
As such, given a subgroup
$\Gamma \leq G \times \Sigma_n^{op}$
and $\vect{C}$ of arity $n$, we say that 
\emph{$\Gamma$ stabilizes $\vect{C}$} 
if 
$g \vect{C} \sigma = \vect{C}$
for all $(g,\sigma) \in \Gamma$.
In particular, 
\eqref{SSYMGSTMAP EQ}
then implies that, 
if $\Gamma$ stabilizes $\vect{C}$
and for
$X\in \mathsf{sSym}^G$
(resp.
$\O \in \mathsf{sOp}^G$)
with $G$-color set $\mathfrak{C}$,
the level
$X(\vect{C})$
(resp. $\O(\vect{C})$)
has an action by $\Gamma$.

\begin{notation}\label{SIGFREE NOT}
	For $x \in X(\vect{C})$
	with $\vect{C}$ of arity $n$ and 
	$(g,\sigma) \in G \times \Sigma_n^{op}$
	we write
	$gx\sigma \in X(g \vect{C} \sigma)$
	for the image of $x$ under 
	\eqref{SSYMGSTMAP EQ}.
	Note that this defines an action of
	$G \times \Sigma_n^{op}$
	on $\coprod_{\vect{C} \text{ of arity }n} X(\vect{C})$.
	
	Moreover, if $x \sigma = x$ only when $\sigma = id$ 
	we say that $x$ is \emph{$\Sigma$-free}.
\end{notation}

\begin{remark}\label{SYMGCPRESH REM}
	Let $\mathcal{G}_{\mathfrak{C}}$
	denote the groupoid with objects the 
	$\mathfrak{C}$-profiles
	and arrows 
	$\vect{C} \to g \vect{C} \sigma$
	for $\vect{C}$ of arity $n$
	and $(g,\sigma) \in G \times \Sigma^{op}_n$.
	In other words, 
	$\mathcal{G}_{\mathfrak{C}}$
	is the coproduct over $n\geq 0$ of the action groupoids 
	for the actions of
	$G \times \Sigma^{op}_n$
	on $n$-ary profiles\footnote{In \cite[Prop. \ref{OC-EQUIVFNCON PROP}]{BP_FCOP}, we use the alternative description
	$\mathcal{G}_{\mathfrak{C}} = G \ltimes \Sigma_{\mathfrak{C}}^{op}$.}.
	Equation \eqref{SSYMGSTMAP EQ} then identifies 
	$\mathsf{sSym}^G_{\mathfrak{C}} \simeq 
	\mathsf{sSet}^{\mathcal{G}_{\mathfrak{C}}}$.	
\end{remark}

Before describing the model structure on $\mathsf{sOp}^G$
we need to recall two more ingredients.

First, 
a subgroup $\Gamma \leq G \times \Sigma_n^{op}$
is called a \emph{$G$-graph subgroup}
if $\Gamma \cap \Sigma_n^{op} = \{*\}$.
Equivalently, one can readily show
that we must have
$\Gamma = \sets{(h,\phi(h)^{-1})}{h \in H}$
for some subgroup $H \leq G$
and homomorphism $\phi \colon H \to \Sigma_n$.

Second, one has functors (compare with Definition \ref{HOMGENOP DEF} and the subsequent remarks)
\[
	\mathsf{sOp} \xrightarrow{\pi_0}
	\mathsf{Op} \xrightarrow{\iota^{\**}}
	\mathsf{Cat}
\]
where $\pi_0$ is computed levelwise,
i.e. 
$(\pi_0 \O)(\vect{C}) = 
\pi_0(\O(\vect{C}))$
and
$\iota^{\**}$ forgets non-unary operations.

Generalizing \cite{Ber07b,CM13b},
we showed in \cite{BP_ACOP}
that $\mathsf{sOp}^G$ 
has a Dwyer-Kan style model structure.
More precisely, we have the following result, which is
\cite[Thm. \ref{AC-THMA}]{BP_ACOP},
with the alternative characterization of fibrations 
provided by 
\cite[Prop. \ref{AC-ISOFIBHARD PROP}]{BP_ACOP}.

\begin{theorem}\label{SOPG_THM}
      The category $\sOp^G$ has a cofibrantly generated model structure with weak equivalences (resp. fibrations) those maps
      $F \colon \O \to \P$ such that:
\begin{itemize}
\item $F$ is \emph{fully faithful} (resp. a \emph{local fibration}), i.e. the induced maps
	\begin{equation}\label{DKEQUIV_EQ}
		\O ( \vect{C})^\Gamma \longto 
		\P(F \vect{C})^\Gamma
	\end{equation}
	are Kan equivalences (resp. Kan fibrations) in $\sSet$
	for all $\mathfrak C_\O$-profiles $\vect C = (\mathfrak c_1, \dots, \mathfrak c_n; \mathfrak c_0)$
	and all graph subgroups $\Gamma \leq G \times \Sigma_n^{op}$ which stabilize $\vect C$;
\item $F$ is \emph{essentially surjective} (resp. an \emph{isofibration}), i.e. the induced maps of usual categories
	\begin{equation}\label{ESSSURJ EQ}
		\iota^{\**}\pi_0 \O^H \longto \iota^{\**}\pi_0 \P^H
	\end{equation}
	are essentially surjective (resp. isofibrations) for all $H \leq G$.
\end{itemize}
\end{theorem}

The sets of generating (trivial) cofibrations in $\sOp^G$
are recalled in Remark \ref{GENCOF_SOPG_REM} in \S \ref{NERTENSCOOP SEC},
after discussing the fibered simplicial cotensoring 
in $\sOp^G$.

For a fixed $G$-set of colors $\mathfrak{C}$,
the category of $\mathfrak C$-colored symmetric sequences
$\mathsf{sSym}_{\mathfrak{C}}^G$
admits an auxiliary model structure defined analogously to \eqref{DKEQUIV_EQ}.

\begin{proposition}
        \label{AUXSYM PROP}
        For all $G$-sets $\mathfrak C$,
        $\mathsf{sSym}_{\mathfrak C}^G$
        has a model structure where
        $X \to Y$ is a (trivial) fibration (resp. weak equivalence)
        iff 
        $X(\vect{C})^{\Gamma} \to Y(\vect{C})^{\Gamma}$
        is a (trivial) Kan fibration (resp. weak equivalence)
        for all $n$-ary $\mathfrak C$-profiles $\vect C$ and
        all graph subgroups $\Gamma \leq G \times \Sigma_n^{op}$ stabilizing $\vect C$.
\end{proposition}

\begin{proof}
        Using the identification
        $\mathsf{sSym}^{G}_{\mathfrak{C}}
        \simeq \mathsf{sSet}^{\mathcal{G}_{\mathfrak{C}}}$
        in Remark \ref{SYMGCPRESH REM},
        this is the $\F^\Gamma_{\mathfrak C}$-model structure 
        from \cite[Defn. \ref{OC-FCMODEL_DEF}]{BP_FCOP},
        where
        $\mathcal{F}^{\Gamma}_{\mathfrak C}$
        is the family of subgroups of $\mathcal{G}_{\mathfrak C}$ (\cite[Defn. 4.11]{BP_FCOP})
        defined by
        $\mathcal{F}^{\Gamma}_{\mathfrak C} = \pi_{\mathfrak C}^{\**} \mathcal F^\Gamma$
        for $\pi_\mathfrak C \colon G \ltimes \Sigma_{\mathfrak C}^{op} \to G \times \Sigma^{op}$
        (\cite[Defn. \ref{OC-GSFAM_DEF}, Remark \ref{OC-FAMC_DEF_REM}]{BP_FCOP});
        explicitly, for an $n$-ary $\mathfrak C$-profile $\vect C$,
        $\left(\mathcal F^\Gamma_{\mathfrak C}\right)_{\vect C}$ is the collection of
        graph subgroups $\Gamma \leq G \times \Sigma_n^{op}$ which stabilize $\vect C$.
        
        As the genuine model structure on $\sSet^G$ exists by e.g. \cite[Prop. 4.10]{BP_FCOP},
        this $\F^\Gamma_{\mathfrak C}$-model structure exists by \cite[Prop. \ref{OC-ALLEQ PROP}]{BP_FCOP}.
\end{proof}

\begin{proposition}\label{SSYMCOFCH PROP}
	Let $f \colon A \to B$ be a map in
	$\mathsf{sSym}^G_{\mathfrak{C}} \simeq \mathsf{sSet}^{\mathcal{G}_{\mathfrak{C}}}$.
	The following are equivalent:
\begin{enumerate}
	\item[(i)] $f$ is a cofibration;
	\item[(ii)] $f$ is a monomorphism and the stabilizer of every
	$x \in B \setminus f(A)$
	is a graph subgroup;
	\item[(iii)] $f$ is a monomorphism and every 
	$x \in B \setminus f(A)$ is $\Sigma$-free
	(cf. Notation \ref{SIGFREE NOT}).
\end{enumerate}
\end{proposition}

\begin{proof}
By \cite[Rem. \ref{OC-VGFGEN REM}]{BP_FCOP}
the generating cofibrations in 
$\mathsf{sSym}^G_{\mathfrak{C}}$
then have the form
$\mathcal{G}_{\mathfrak{C}}
(\vect{C},-)/\Gamma \cdot (\partial \Delta[k] \to \Delta[k])$
with $\Gamma \in \mathcal{F}^{\Gamma}_{\vect{C}}$,
$k \geq 0$,
so that (i) $\Leftrightarrow$ (ii)
follows by adapting 
\cite[Prop. 2.16]{Ste16}
or
\cite[Prop. 6.5]{Per18}.
(ii) $\Leftrightarrow$ (iii)
is straightforward.
\end{proof}

In light of (iii) in the previous result,
a color fixed map of operads
$\O \to \mathcal{P}$
such that the underlying map of symmetric sequences is a cofibration
is called a \emph{$\Sigma$-cofibration}.
Combining Proposition \ref{SSYMCOFCH PROP} with 
\cite[Prop. \ref{OC-SIGMAG_COF PROP}]{BP_FCOP}
(also, see \cite[Prop. \ref{AC-FIBERGLMOD PROP}(ii)]{BP_ACOP})
yields the following.

\begin{proposition}\label{COPOFSIGCOF PROP}
	If $f \colon \O \to \mathcal{P}$
	is a color fixed cofibration in $\mathsf{sOp}^G$
	and $\O$ is $\Sigma$-cofibrant
	then $f$ is a $\Sigma$-cofibration.
	In particular, cofibrant operads are $\Sigma$-cofibrant.
\end{proposition}

\begin{remark}[{cf. \cite[Rem. \ref{OC-OP_MAP REM}]{BP_FCOP}}]
	\label{CHECKF REM}
As in \eqref{PREOPCOLCH EQ},
for $f \colon \mathfrak{C} \to \mathfrak{D}$
a map of $G$-sets of colors one has adjunctions
\begin{equation}
f_{!} \colon
\mathsf{sSym}^{G}_{\mathfrak{C}}
\rightleftarrows
\mathsf{sSym}^{G}_{\mathfrak{D}}
\colon f^{\**}
\qquad
\check{f}_{!} \colon
\mathsf{sOp}^{G}_{\mathfrak{C}}
\rightleftarrows
\mathsf{sOp}^{G}_{\mathfrak{D}}
\colon f^{\**}.
\end{equation}
The right adjoints, 
both denoted $f^{\**}$, are given by 
$f^{\**}Y(\vect{C}) = Y(f\vect{C})$
(cf. Definition \ref{CFOREST_DEF}).
The left adjoint $f_!$ on symmetric sequences is
$f_!X = 
\mathsf{Lan}_{\Sigma^{op}_{\mathfrak{C}}\to \Sigma^{op}_{\mathfrak{D}}} X$.
The left adjoint  $\check{f}_!$ on operads
is given on free operads by 
$\check{f}_! \mathbb{F} \O = \mathbb{F} f_! \O$
and thus
(since any operad $\O$ is a coequalizer
$\mathop{\mathrm{coeq}}
\left(
\mathbb{F}\mathbb{F} \O 
\rightrightarrows
\mathbb{F} \O
\right)$)
on general operads $\O$ by
\begin{equation}\label{CHECKF EQ}
\mathop{\mathrm{coeq}}
\left(
\mathbb{F}\left(f_! \mathbb{F} \O \right)
\rightrightarrows
\mathbb{F} f_! \O
\right).
\end{equation}
\end{remark}

\subsection{The nerve and the fibered simplicial (co)tensor on operads}\label{NERTENSCOOP SEC}

This section discusses two necessary constructions: 
the nerve functor
$
N \colon \mathsf{sOp} \to
\mathsf{PreOp}$
and, in analogy with
the fibered tensoring
$\otimes_{\mathfrak{C}_{\bullet}}$ in 
$\mathsf{PreOp}$ from
\S \ref{FIBTENS_SEC},
a similar tensoring
$\otimes_{\mathfrak{C}_{\bullet}}$
on $\mathsf{sOp}$.

We first recall and extend the 
\emph{operadification-nerve functor}
adjunction in \eqref{TAUADJ EQ} to
\begin{equation}\label{TAUNER EQ}
	\tau\colon \mathsf{dSet}
	\rightleftarrows
	\mathsf{Op} \colon N
\qquad
	\tau\colon \mathsf{PreOp}
	\rightleftarrows
	\mathsf{sOp} \colon N
\end{equation}
where the rightmost adjunction 
simply applies the leftmost adjunction
to each simplicial level
(indeed, by Definition \ref{PREOP DEF}, 
Remark \ref{CTTCOLSET REM}
both $\mathsf{PreOp}$ and $\mathsf{sOp}$
are characterized by demanding that
the color sets are constant in the 
simplicial direction).
One then has the formulas
\[
\tau X = \colim_{\Omega[U] \to X} \Omega(U),
\qquad
(N \O) (U) = \mathsf{Op}(\Omega(U),\O),\quad U \in \Omega
\]
for $X \in \mathsf{dSet}$,
$\O \in \mathsf{Op}$,
and with $\Omega(U)$ the free operad determined by a tree 
$U \in \Omega$, that we now recall.
In fact, we define $\Omega(-)$ slightly more generally.
For $F \in \Phi$ a forest, 
$\Omega(F)$ is the $\boldsymbol{E}(F)$-colored operad
which, evaluated at
a $\boldsymbol{E}(F)$-colored corolla
$\vect{C} = 
(C,\mathfrak{c} \colon \boldsymbol{E}(C) \to \boldsymbol{E}(F))$,
is given by
\begin{equation}\label{OMEGADEF_EQ}
\Omega(F)(\vect C) =
	\begin{cases}
		\** \qquad & 
		\mbox{if
			$\mathfrak{c} \colon \boldsymbol{E}(C) \to \boldsymbol{E}(F)$
			defines a map $C \to F$ in $\Phi$}
	\\
		\varnothing & \text{otherwise.}
	\end{cases}
\end{equation}
Note that, as the levels of 
$\Omega(F)$ are $\**$ or $\emptyset$,
one has at most one way to compose operations,
and thus at most one possible operad structure on $\Omega(F)$.
That this structure exists 
(i.e. that composition is well defined)
follows since,
for a tree $U\in \Omega$,
a coloring
$\mathfrak{c} \colon \boldsymbol{E}(U) \to \boldsymbol{E}(F)$
defines a map
$U \to F$ in $\Phi$
iff
the restrictions
$\mathfrak{c}_v \colon \boldsymbol{E}(U_v) \to \boldsymbol{E}(F)$,
$v \in \boldsymbol{V}(U)$
define maps
$U_v \to F$ in $\Phi$.

\begin{remark}\label{NERVESIMDES REM}
	Recalling the $X_{\mathfrak{c}}(U)$ notation in 
	Notation \ref{XUDECALT NOT},
	one has that
	$\left(N\Omega(F)\right)_{\mathfrak{c}}(U)$
	is given exactly as in 
	\eqref{OMEGADEF_EQ} with the corolla $C$ replaced by a general tree $U$,
	so that
	$N\Omega(F) = \Omega[F]$.
\end{remark}

We will also make use of an alternative description of $\Omega(F)$, as follows.

First, if $\vect{C}\in \Sigma_{\mathfrak{C}}$ is a
$\mathfrak{C}$-corolla,
we denote its representable functor
in 
$\Sym_{\mathfrak C} = \mathsf{Set}^{\Sigma^{op}_{\mathfrak{C}}}$
by 
$\Sigma_{\mathfrak C}[\vect C] =
\Sigma_{\mathfrak C}^{op}(\vect C, -)$.
Second, for
$\vect{F}\in \Phi_{\mathfrak{C}}$ a
$\mathfrak{C}$-forest, 
we extend the 
$\Sigma_{\mathfrak C}[-]$ notation via
\[
	\Sigma_{\mathfrak C}[\vect F] = \sideset{}{^{\mathfrak C}}\coprod_{v \in \boldsymbol{V}(F)} \Sigma_{\mathfrak C}[\vect F_v],
\]
where $\amalg^{\mathfrak C}$ is the coproduct
in $\Sym_{\mathfrak C}$ (rather than in the larger category $\Sym$).
Third, for 
$F \in \Phi$
an (uncolored) forest we write
$F^{\tau}$ for $F$ with its tautological 
$\boldsymbol{E}(F)$-coloring, i.e.
$F^{\tau} = 
(F,\mathfrak{t}\colon \boldsymbol{E}(F) \xrightarrow{=} \boldsymbol{E}(F))$,
and abbreviate
$\Sigma_{\tau}[F] = \Sigma_{\boldsymbol{E}(T)}[F^{\tau}]$.
All together, one then has an identification
\begin{equation}\label{OMFFREE EQ}
	\Omega(F) = \mathbb{F} \Sigma_{\tau}[F]
\end{equation}
which, informally, says that 
``$\Omega(F)$ is freely generated by the vertices of $F$''.

We now recall \cite[Prop. 5.3 and Thm. 6.1]{MW09}
that the nerve
$N \colon \mathsf{Op} \to \mathsf{dSet}$
is then a fully faithful inclusion
whose (essential) image can be characterized
as those
dendroidal sets $X \in \mathsf{dSet}$
with the strict right lifting property against inner horn inclusions
$\Lambda^e[U] \to \Omega[T]$ for $U\in\Omega,e\in \boldsymbol{E}(U)$.
Next,
following either \cite[Prop. 2.5 and Cor. 2.6]{CM13a}
or \cite[Props. 3.22 and 3.31]{BP20},
this is in turn equivalent to the strict right lifting property of $X$
against Segal core inclusions
$Sc[U] \to \Omega[T]$ for $U \in \Omega$,
which is in turn equivalent to the strict Segal conditions
(cf. Definition \ref{SEGCOLCHAR DEF}) below,
demanding that the maps
\begin{equation}\label{STRSEGCON EQ}
	X(U)
	\xrightarrow{\simeq}
	X(Sc[U]),
	\quad
	U \in \Omega
\qquad \qquad
	X_{\mathfrak{c}}(U) 
	\xrightarrow{\simeq}
	\prod_{v \in \boldsymbol{V}(U)}
	X_{\mathfrak{c}_v}(U_v),
	\quad
	U \in \Omega,
	\mathfrak{c} \colon 
	\boldsymbol{E}(U) \to X(\eta)
\end{equation}
are all isomorphisms.
Moreover, one then has the following alternate formula for the nerve
$N \O$ evaluated at
$U\in \Omega,
\mathfrak{c} \colon \boldsymbol{E}(U) \to \mathfrak{C}_{\O}$.
\begin{equation}\label{ALTNER EQ}
	(N \O)_{\mathfrak{c}} (U) 
	= 
	\prod_{v \in \boldsymbol{E}(U)}
	\O(U_v,\mathfrak{c}_v)
	=
	\prod_{v \in \boldsymbol{E}(U)}
	\O(\vect{U}_v).
\end{equation}

\begin{remark}\label{TAUSC REM}
	Setting $X=N\O$ and unpacking \eqref{SDSET_EQ},
	the left isomorphisms in \eqref{STRSEGCON EQ}
	become
	$\mathsf{dSet}(\Omega[U],N\O)
	\xrightarrow{\simeq} 
	\mathsf{dSet}(Sc[U],N\O)$,
	yielding isomorphisms
	$\tau(Sc[U]) \xrightarrow{\simeq} \tau(\Omega[U])=\Omega(U)$.
\end{remark}

Mirroring \S \ref{FIBTENS_SEC},
we next describe the fibered simplicial (co)-tensoring
$\otimes_{\mathfrak{C}_{\bullet}}$
on $\mathsf{sOp}^G$.

First, for $\O \in \mathsf{sOp}^G_{\mathfrak{C}}$
and 
$K \in \mathsf{sSet}$
we define the \emph{fiber cotensor}
$\{K,\O\}_{\mathfrak{C}_{\bullet}} 
\in \mathsf{sOp}^G_{\mathfrak{C}}$
via the pointwise simplicial cotensor, i.e.
\begin{equation}\label{KOCOPER EQ}
	\set{K, \O}_{\mathfrak{C}_{\bullet}} (\vect C) = \O(\vect C)^K.
\end{equation}

\begin{remark}\label{NERTENID REM}
	The fact that 
	$\set{K, \O}_{\mathfrak{C}_{\bullet}}$
	as described above has an operad structure
	can be seen by considering nerves. 
	Indeed, 
	one readily checks that the strict Segal condition
	\eqref{STRSEGCON EQ} for $N\O$ 
	implies the same condition for 
	the preoperad
	$\{K,N \O\}_{\mathfrak{C}_{\bullet}}$
	(defined as in \eqref{EASYCOTEN EQ}).
	Thus, \eqref{KOCOPER EQ}
	describes the levels
	of the unique (up to isomorphism) operad
	$\set{K, \O}_{\mathfrak{C}_{\bullet}}$
	such that
	$N\set{K, \O}_{\mathfrak{C}_{\bullet}} 
	\simeq
	\set{K, N \O}_{\mathfrak{C}_{\bullet}}$.
\end{remark}

We now turn to the \emph{fiber tensor} 
$(-)\otimes_{\mathfrak{C}_{\bullet}} K$,
left adjoint to \eqref{KOCOPER EQ}.
First, note that 
\eqref{KOCOPER EQ}
still makes sense 
at the level of symmetric sequences, i.e. 
with 
$\O \in \mathsf{sOp}^G_{\mathfrak{C}}$
replaced with
$X \in \mathsf{sSym}^G_{\mathfrak{C}}$.
Then, at the level of symmetric sequences, the left adjoint construction
$X \times K \in \mathsf{sSym}^G_{\mathfrak{C}}$
is simply given pointwise by
$(X \times K)(\vect{C}) = X(\vect{C}) \times K$.
It is now formal that, 
on a free operad $\mathbb{F} X$,
the tensor $\mathbb{F} X \otimes_{\mathfrak{C}_{\bullet}} K$
adjoint to \eqref{KOCOPER EQ} is given by 
$(\mathbb{F} X) \otimes_{\mathfrak{C}_{\bullet}} K
= \mathbb{F} (X \times K)$
so that,
for a general $\O \in \mathsf{sOp}^G$
(which has a description
$\O \simeq \mathop{\mathrm{coeq}} (\mathbb{F}\mathbb{F}\O \rightrightarrows \mathbb{F}\O)$
as a coequalizer of free algebras)
it is given by (cf. \eqref{CHECKF EQ})
\begin{equation}\label{OTIMESC EQ}
\mathcal{O} \otimes_{\mathfrak{C}_{\bullet}} K
\simeq
\mathop{\mathrm{coeq}}
\left(
\mathbb{F} (\mathbb{F} \O \times K) 
\rightrightarrows \mathbb{F} (\O \times K) 
\right).
\end{equation}
\begin{remark}\label{BYHAND2 REM}
	In \cite[\S 7.1]{CM13b}, 
	the objects $\Omega(T)  \otimes_{\mathfrak{C}_{\bullet}} K$
	were denoted $T[K]$ and built by hand.
\end{remark}

\begin{remark}
	The analogues of
	Remarks \ref{TENSCOADJ REM},\ref{NOTTWOVARADJ REM}
	apply mutatis mutandis to the operadic fiber tensor.
	In particular, one has that the canonical map
\begin{equation}\label{COCARTAR EQ}
	\amalg_i \O \otimes_{\mathfrak{C}_{\bullet}} K_i
\to
	\O \otimes_{\mathfrak{C}_{\bullet}} (\amalg_i K_i)
\end{equation}
	is a cocartesian arrow over the fold map
	$\nabla \colon \amalg_i \mathfrak{C} \to \mathfrak{C}$,
	i.e. 
	$\nabla_!\left(\amalg_i \O \otimes_{\mathfrak{C}_{\bullet}} K_i \right)
		\simeq
	\O \otimes_{\mathfrak{C}_{\bullet}} (\amalg_i K_i)$.
\end{remark}

\begin{proposition}\label{TAUOTIMES_PROP}
	For all $X \in \mathsf{PreOp}^G$ and $K \in \sSet$, 
	one has a natural identification
	\[\tau(X \otimes_{\mathfrak{C}_{\bullet}} K) \simeq \tau(X) \otimes_{\mathfrak{C}_{\bullet}} K\]
	with the first (resp. second) $\otimes_{\mathfrak{C}_{\bullet}}$ is 
	the fiber simplicial tensoring of $\mathsf{PreOp}^G$ 
	(resp. $\sOp^G$).
\end{proposition}

\begin{proof}
	This is equivalent to the 
	already established (cf. Remark \ref{NERTENID REM}) adjoint identification
	$
	N \{K,\O\}_{\mathfrak{C}_{\bullet}}
	\simeq
	\{K,N \O\}_{\mathfrak{C}_{\bullet}} 
	$
	for $\O \in \mathsf{sOp}^G$.	
\end{proof}

\begin{remark}
	Proposition \ref{TAUOTIMES_PROP} is a slight generalization of 
	\cite[Prop. 7.2]{CM13b},
	which establishes the case
	$X = \Omega[U]$, $U\in \Omega$
	by direct inspection
	(cf. Remarks \ref{BYHAND1 REM},\ref{BYHAND2 REM}).
\end{remark}

\begin{remark}\label{SIGMATAUQUOT REM}
Let $\Gamma \leq G \times \Sigma_n^{op}$
be the graph subgroup given by
$\Gamma = \{(h,\phi(h)^{-1})|h\in H\}$
for $H\leq G$, $\phi\colon H \to \Sigma_n$.
Writing
$C_{n}$ for the $n$-corolla, 
$\phi$ defines a left $H$-action on $C_{n}$,
so that one obtains an associated
$G$-corolla $C = G \cdot_{H} C_{n}$.
It is straightforward to check that there are natural identifications
(here we view the natural 
left $G^{op}\times \Sigma_n$-action on 
$G \cdot C_{n}$ as a 
right $G\times \Sigma_n^{op}$-action)
\begin{equation}\label{SIGMATAUQUOT EQ}
	(G \cdot C_{n})/\Gamma
\simeq
	G \cdot_H C_{n}
= 
	C
\qquad
	\Sigma_{\tau}[G \cdot C_{n}]/\Gamma
\simeq
	\Sigma_{\tau}[G \cdot_H C_{n}]
= 
	\Sigma_{\tau}[C]
\end{equation}
in $\Phi^G_{\bullet}$ and $\mathsf{Sym}^G$, respectively.
\end{remark}

We end the section by discussing the generating sets
for the model category $\mathsf{sOp}^G$ in Theorem \ref{SOPG_THM},
as given in \cite[Def. \ref{AC-OPGENCOF DEF}]{BP_ACOP}.
We need one last ingredient,
cf. \cite[Def. \ref{AC-VINTER DEF}]{BP_ACOP}, Definition \ref{PSEUINT DEF}.

\begin{definition}
	Let $\widetilde{[1]}$ denote the free isomorphism category,
	i.e. the contractible groupoid with two objects $0,1$.
	An \textit{interval} is a cofibrant simplicial category $\mathbb I \in \mathsf{sCat}_{\set{0,1}}$ equivalent to $\widetilde{[1]}$.
\end{definition}

\begin{remark}\label{GENCOF_SOPG_REM}
	Specifying \cite[Def. \ref{AC-OPGENCOF DEF}]{BP_ACOP}
	for the category $\mathsf{sSet}$,
	the notation therein becomes
\[
	\mathbb{F}
	\left(\Sigma_{\tau}[G \cdot C_{n}]/\Gamma \cdot f\right)
\simeq
	\mathbb{F}
	\left(\Sigma_{\tau}[C] \cdot f\right)
\simeq
	\mathbb{F}
	\left(\Sigma_{\tau}[C] \right) \otimes_{\mathfrak{C}_{\bullet}} f
\simeq 
	\Omega(C) \otimes_{\mathfrak{C}_{\bullet}} f
\]
where the first identification is
\eqref{SIGMATAUQUOT EQ},
the second follows by definition of
$\otimes_{\mathfrak{C}_{\bullet}}$,
and the third is \eqref{OMFFREE EQ}.
We thus have that the generating cofibrations in $\mathsf{sOp}^G$
are the maps
	\begin{itemize}
		\item[(C1)] $\emptyset \to G/H \cdot \Omega(\eta)$ for $H \leq G$;
		\item[(C2)] $\Omega(C) \otimes_{\mathfrak{C}_{\bullet}} (\partial \Delta[n] \to \Delta[n])$
		for $C \in \Sigma_G$, $n \geq 0$.
	\end{itemize}
	while the generating trivial cofibrations are 
	(for the countability condition, 
	see \cite[Rem. \ref{AC-SSETINT_REM}]{BP_ACOP})
	\begin{itemize}
		\item[(A1)] 
		$G/H \cdot \left(\eta \to \mathbb{G}\right)$ 
		for $H \leq G$ and $\mathbb G$ an interval with countably many simplices;
		\item[(A2)] 
		$\Omega(C) \otimes_{\mathfrak{C}_{\bullet}} 
		(\Lambda^i[n] \to \Delta[n])$
		for $C \in \Sigma_G$, $0 \leq i \leq n$, $1 \leq n$.
	\end{itemize}
\end{remark}

In (C2),(A2) above 
the group $G$ acts only on $\Omega(C)$
and not on the featured simplicial sets.
The following lemma considers the case where the simplicial sets also have a $G$-action
(for a discussion of the genuine model structure on 
$\mathsf{sSet}^G$, see \cite[Def. \ref{OC-GENMOD DEF}]{BP_FCOP}).
\begin{lemma}
	\label{OPTENSCOF_LEM}
	For $C \in \Sigma_G$ and $A \to B$ a genuine (trivial) cofibration in $\sSet^G$,
	$\Omega(C) \otimes_{\mathfrak{C}_{\bullet}} (A \to B)$ is a (trivial) cofibration in $\sOp^G$.
\end{lemma}

\begin{proof}
	Since $\Omega(C) \otimes (-) \colon \sSet^G \to \sOp^G_{\boldsymbol{E}(T)}$ preserves colimits, 
	by \cite[\eqref{OC-GENGENSETEQ}]{BP_FCOP}
	it suffices to consider the case 
	$(A \to B) = G/H \cdot (K \to L)$ for $K \to L$ a
	generating (trivial) cofibration in $\sSet$.
	Now consider the following diagram.
\begin{equation}
\begin{tikzcd}[column sep = small, ampersand replacement = \&]
	\left(G/H \cdot \Omega(C)\right) \otimes_{\mathfrak{C}_{\bullet}} K 
	\arrow[r] \arrow[d]
\&
	\Omega(C) \otimes_{\mathfrak{C}_{\bullet}} \left(G/H \cdot K\right)
	\arrow[d]
\\
	\left(G/H \cdot \Omega(C)\right) \otimes_{\mathfrak{C}_{\bullet}} L
	\arrow[r] 
\&
	\Omega(C) \otimes_{\mathfrak{C}_{\bullet}} \left(G/H \cdot L\right)
\end{tikzcd}
% }
\end{equation}
By \eqref{COCARTAR EQ} the horizontal arrows are cocartesian, while the vertical arrows fix colors, 
so this is a pushout square.
The result now follows since 
$G/H \cdot \Omega(C)$ decomposes as a coproduct 
$\amalg_i \Omega(C_i)$ with $C_i\in \Sigma_G$,
so that the left vertical map above 
is a coproduct of maps in (C2).
\end{proof}

%\begin{lemma}[cf. Lemma \ref{OMEGATTAME_LEM}]
%	For all $T \in \Omega_G$, $\Omega(T) \in \sOp^G$ is cofibrant.
%\end{lemma}
%
%\begin{proof}
%	This follows since the square below is a pushout in $\sOp^G$,
%	with $\coprod_{\boldsymbol{E}(T)} \Omega(\eta)$ cofibrant by (C1),
%	while the left vertical map is a cofibration by (C2) with $n=0$.
%	\[
%	\begin{tikzcd}
%	\displaystyle{
%		\coprod_{Gv \in \boldsymbol{V}_G(T)} \partial \Omega(T_{Gv})
%	}
%	\arrow[d] \arrow[r]
%	&
%	\displaystyle{
%		\coprod_{\boldsymbol{E}(T)} \Omega(\eta)
%	}
%	\arrow[d]
%	\\
%	\displaystyle{
%		\coprod_{Gv \in \boldsymbol{V}_G(T)} \Omega(T_{Gv})
%	}
%	\arrow[r]
%	&
%	\Omega(T)
%	\end{tikzcd}
%	\]
%\end{proof}

\subsection{Equivalence between preoperads and operads}
\label{PREOPOPEQUIV SEC}

Our goal in this subsection is to prove
Theorem \ref{PREQUIEQUIV THM},
establishing the Quillen equivalence between
preoperads $\mathsf{PreOp}^G$
and operads $\mathsf{sOp}^G$.
The key to proving this result is given by 
Lemma \ref{UNITEQUIV LEM}
and the subsequent
Corollary \ref{KEYEQUIV COR},
which allow us to understand the counit of the adjunction.
These latter results in turn depend on the following key 
result, whose proof is deferred to \S \ref{KEYRES SEC}.

\begin{lemma}\label{KEYPRVAR LEM}
	Suppose that $\mathcal{O} \in \mathsf{Op}^{G}$
	is $\Sigma$-cofibrant.
	Further, let $C \in \Sigma_G$ be any $G$-corolla,
	$r \geq 1$ a positive integer and consider 
	a pushout in $\mathsf{Op}^{G}$ of the form below.
\begin{equation}\label{PUSHOUTPROPVAR EQ}
	\begin{tikzcd}
	\partial \Omega(C)^{\amalg r} \ar{r} \ar{d} 
	& \mathcal{O} \ar{d}
\\
	\Omega(C)^{\amalg r} \ar{r} & \mathcal{P}
	\end{tikzcd}
\end{equation}
	Then the induced map
\begin{equation}\label{ANODYNEVAR EQ}
	\Omega[C]^{\amalg r} 
	\amalg_{\partial \Omega[C]^{\amalg r}} N\mathcal{O} \to N\mathcal{P}
\end{equation}
	is $G$-inner anodyne.
\end{lemma}

\begin{remark}\label{KEYPRVAR REM}
	Both \eqref{PUSHOUTPROPVAR EQ} and \eqref{ANODYNEVAR EQ}
	are unchanged if the copowers $(-)^{\amalg r}$ in 
	$\mathsf{Op}^G$, $\mathsf{dSet}^G$
	are replaced with fibered copowers 
	$(-)^{\amalg_{\mathfrak{C}_{\bullet}} r} \simeq 
	(-) \otimes_{\mathfrak{C}_{\bullet}} \{1,\cdots,r\}$
	in 
	$\mathsf{Op}^G_{\boldsymbol{E}(C)}$,
	$\mathsf{dSet}^G_{\boldsymbol{E}(C)}$.
	
	In addition, since 
	$\partial \Omega(C) = \Omega(C) \otimes_{\mathfrak{C}_{\bullet}} \emptyset$,
	one is moreover free to replace
	the left vertical map in \eqref{PUSHOUTPROPVAR EQ}
	with 
	$\Omega(C) \otimes_{\mathfrak{C}_{\bullet}} K
	\to
	\Omega(C) \otimes_{\mathfrak{C}_{\bullet}} L$
	for $K\to L$ any inclusion of sets.
\end{remark}

\begin{remark}
	The integer $r \geq 1$ in 
	Lemma \ref{KEYPRVAR LEM}
	is included to match our required application in Lemma \ref{UNITEQUIV LEM}.
	However, it readily follows by induction on $r$ that one needs only prove the $r=1$ case.
	Indeed, writing $\O_r$ for the pushout in 
	\eqref{PUSHOUTPROPVAR EQ}
	for each $r$,
	one has a diagram below
\begin{equation}
\begin{tikzcd}
	\Omega[C]^{\amalg r} \amalg_{\partial \Omega[C]^{\amalg r}} N\mathcal{O} \ar{r} \ar{d}
	\arrow[dr, phantom, "\ulcorner", very near start] 
&
	N \mathcal{O}_r \ar{d}
\\
	\Omega[C]^{\amalg r+1} \amalg_{ \partial \Omega[C]^{\amalg r+1}} N \mathcal{O} \ar{r}
&
	\Omega[C] \amalg_{\partial \Omega[C]} N \mathcal{O}_r \ar{r}
&
	N \mathcal{O}_{r+1}
\end{tikzcd}
\end{equation}
	where the square is a pushout so that induction on $r$ and the $r=1$ case yield that all horizontal maps
	are $G$-inner anodyne.
	The proof of the interesting $r=1$ case will occupy the entirety of \S \ref{KEYRES SEC}.
\end{remark}

\begin{lemma}[{cf. \cite[Lemma 8.2]{CM13b}}]
	\label{UNITEQUIV LEM}
	Let $A \to B$ be a tame cofibration in $\mathsf{PreOp}^G$, 
	$\mathcal{O} \in \mathsf{sOp}^G$ a $\Sigma$-cofibrant 
	$G$-operad,
	and consider a pushout diagram in $\mathsf{sOp}^G$ of the form below.
\begin{equation}\label{THEPUSH EQ}
\begin{tikzcd}
	\tau A \ar{r} \ar{d} & \mathcal{O} \ar{d}
\\
	\tau B \ar{r} & \mathcal{P}
\end{tikzcd}
\end{equation}
	Then $\mathcal{O} \to \mathcal{P}$ is a $\Sigma$-cofibration and 
	\begin{equation}\label{UNITEQUIV EQ}
	B \amalg_{A} N \mathcal{O}
	\to 
	N \mathcal{P}
	\end{equation}
	is a weak equivalence.
\end{lemma}

Setting $A = \emptyset $, $\mathcal{O}= \emptyset$ in the previous result yields the following.

\begin{corollary}\label{KEYEQUIV COR}
	If $B \in \mathsf{PreOp}^G$ is tame cofibrant, then 
	$B \xrightarrow{\sim} N \tau B$ is a weak equivalence.
\end{corollary}

\begin{proof}[Proof of Lemma \ref{UNITEQUIV LEM}]
	We first consider the case where $A\to B$ is in one of (TC1),(TC2),(TC3).
	
	The (TC1) case is immediate, 
	since then
	$\mathcal{O} \to \mathcal{P}$
	is the $\Sigma$-cofibration
	$\mathcal{O} \to \mathcal{O} \amalg G/H \cdot \Omega(\eta)$   and \eqref{UNITEQUIV EQ}
	is the isomorphism
	$N\mathcal{O} \amalg G/H\cdot \Omega[\eta] \simeq 
	N\left( \mathcal{O} \amalg G/H \cdot \Omega(\eta) \right)$.
	
	The (TC3) case is also straightforward:
	since $\tau A \to \tau B$ is an isomorphism
	(Remark \ref{TAUSC REM}), one can take 
	$\mathcal{O}=\mathcal{P}$, so that 
	\eqref{UNITEQUIV EQ}
	becomes a section of the map
	$N \mathcal{O} \to B \amalg_{A} N \mathcal{O}$, which is a trivial cofibration (as it is a pushout of $A \to B$),
	and 2-out-of-3 thus implies that \eqref{UNITEQUIV EQ} is a weak equivalence.

	The most interesting case is then (TC2).
	Firstly, by Proposition \ref{TAUOTIMES_PROP} the functor $\tau$ sends maps in (TC2) to maps in (C2), 
	so $\O \to \P$ is indeed a $\Sigma$-cofibration
	by \cite[Prop. \ref{OC-SIGMAG_COF PROP}]{BP_FCOP}.
	Next, fixing a simplicial level $m\geq 0$,
	$A_m \to B_m$ then has the form
	$\Omega[C] \otimes_{\mathfrak{C}_{\bullet}} \left(\partial \Delta[n]_m \to \Delta[n]_m\right)$ so that
	$\tau A_m \to \tau B_m$ has the form
	$\Omega(C) \otimes_{\mathfrak{C}_{\bullet}} \left(\partial \Delta[n]_m \to \Delta[n]_m\right)$.
	But then (following the discussion in 
	Remark \ref{KEYPRVAR REM})
	Lemma \ref{KEYPRVAR LEM}
	yields that all levels
	$(B \amalg_{A} N \mathcal{O})_m
	\to 
	(N \mathcal{P})_m$
	for $m \geq 0$
	are equivalences in $\mathsf{dSet}^G$,
	showing that 
	$B \amalg_{A} N \mathcal{O}
	\to 
	N \mathcal{P}$
	is indeed a joint equivalence in
	$\mathsf{PreOp}^G$.

	We now turn to the case of $A \to B$ a general tame cofibration.
	As usual, $A \to B$ is a retract of a transfinite composition of pushouts of generating cofibrations.
	Since the conclusions of the result are invariant under retracts,
	we are free to assume that $A \to B$ is a transfinite composite
	\[
	A = A_0 \to A_1 \to A_2 \to \cdots \to 
	\colim_{\beta < \kappa} A_{\beta} = B.
	\]
	where each map $A_{\beta} \to A_{\beta +1}$
	for $\beta < \kappa$ is a pushout of a map in one of (TC1),(TC2),(TC3).

	Defining $\mathcal{O}_{\beta}$ by replacing $A \to B$ with $A \to A_{\beta}$ in the pushout \eqref{THEPUSH EQ},
	$\mathcal{O} \to \mathcal{P}$ becomes the transfinite composite of the maps $\mathcal{O}_{\beta} \to \mathcal{O}_{\beta + 1}$
	and \eqref{UNITEQUIV EQ} becomes
	$
	\colim_{\beta < \kappa} \left( 
	A_{\beta} \amalg_{A} N \mathcal{O}
	\to 
	N \mathcal{O}_{\beta}
	\right)
	$.

	It thus suffices to show, by induction on $\beta < \kappa$, 
	that the maps $\mathcal{O}_{\beta} \to \mathcal{O}_{\beta + 1}$ are $\Sigma$-cofibrations and that the maps 
	$A_{\beta} \amalg_{A} N \mathcal{O}
	\to 
	N \mathcal{O}_{\beta}$
	are weak equivalences
	(sufficiency of the latter condition uses the fact that 
	filtered colimits of weak equivalences in $\mathsf{PreOp}^G$ are weak equivalences, cf. Theorem \ref{PREOPMS THM}).
	Consider now the following diagrams.
	\[
	\begin{tikzcd}
	\tau A \ar{r} \ar{d} & \mathcal{O} \ar{d}
	&&
	A_{\beta} \amalg_{A} N \mathcal{O}
	\ar[>->]{r} \ar{d}[swap]{\sim} &
	A_{\beta+1} \amalg_{A} N \mathcal{O}
	\ar{d}[swap]{\sim}
	\\
	\tau A_{\beta} \ar{r} \ar{d} & \mathcal{O}_{\beta} \ar{d}
	&&
	N \mathcal{O}_{\beta} \ar[>->]{r} &
	A_{\beta+1} \amalg_{A_{\beta}} N \mathcal{O}_{\beta} \ar{d}
	\\
	\tau A_{\beta + 1} \ar{r} & \mathcal{O}_{\beta + 1}
	&&
	&
	N \mathcal{O}_{\beta+1}
	\end{tikzcd}
	\]
	The induction hypothesis states that
	$\mathcal{O} \to \mathcal{O}_{\beta}$ is a $\Sigma$-cofibration and that the map
	$A_{\beta} \amalg_A N \mathcal{O} \to N \mathcal{O}_{\beta}$ is a weak equivalence.
	Therefore, $\mathcal{O}_{\beta}$ is $\Sigma$-cofibrant 
	and both vertical maps marked $\sim$ in the rightmost diagram above are weak equivalences 
	(by left properness of $\mathsf{PreOp}^G_{tame}$),
	and thus the induction step will follow provided that the result holds for
	$A_{\beta} \to A_{\beta + 1}$ and $\mathcal{O}_{\beta}$.
	But $A_{\beta} \to A_{\beta + 1}$ is assumed to be a pushout of a map in (TC1),(TC2),(TC3), 
	in which case the result is already known, and thus noting that the result is invariant under pushouts finishes the proof.
\end{proof}

Before proving Theorem \ref{PREQUIEQUIV THM}, 
we recall the two following results,
which are adapted from \cite{JT07} (see Proposition 7.15 therein).

\begin{proposition}
	A cofibration $A \to B$ in a model category
	is a weak equivalence 	
	iff it has the left lifting property against all fibrations between fibrant objects.
\end{proposition}

%\begin{proof}
%	Let $B \xrightarrow{\sim} \tilde{B}$ be a fibrant replacement and
%	let $A \xrightarrow{\sim} \tilde{A} \twoheadrightarrow \tilde{B}$
%	be a factorization of the composite $A \to \tilde{B}$ 
%	as a trivial cofibration followed by a fibration.
%	One then has a lift in the diagram
%	\[
%	\begin{tikzcd}
%	A \ar{r}{\sim} \ar[>->]{d} & \tilde{A} \ar[->>]{d}
%	\\
%	B \ar{r}{\sim} \ar[dashed]{ru} & \tilde{B}
%	\end{tikzcd}
%	\]
%	where the top and bottom horizontal maps are weak equivalences. 
%	But then the 2-out-of-6 property for weak equivalences says that all maps are weak equivalences.
%\end{proof}

\begin{corollary}\label{SIMPLQUILL COR}
	An adjunction 
\[
	F \colon \mathcal{C}
	\rightleftarrows
	\mathcal{D} \colon G
\]
	between model categories is a Quillen adjunction
	provided that $F$ preserves cofibrations
	and $G$ preserves fibrations between fibrant objects.
\end{corollary}

\begin{theorem}\label{PREQUIEQUIV THM}
	The adjunction
\begin{equation}\label{PREQUIEQUIV EQ}
	\tau \colon \mathsf{PreOp}^G_{\text{tame}}
	\rightleftarrows 
	\mathsf{sOp}^G \colon N
\end{equation}
	is a Quillen equivalence.
\end{theorem}

\begin{proof}
	We first show that $N$ preserves and detects weak equivalences.
	To see this, note first that all objects in the image of $N$ are Segal operads
	(cf. Remark \ref{SEGCOLCHAR_REM}, \eqref{STRSEGCON EQ}) % \eqref{SEGCOLCHAR EQ}
	 so that, by Theorem \ref{FIBPREOP THM}, 
	a map in the image of $N$ is a weak equivalence iff it is a Dwyer-Kan equivalence.
	But it is clear that $N$ preserves and reflects fully-faithful maps,
	and since
	$N \left(\iota^{\**} \O^H \right)
	\simeq
	\iota^{\**} \left( (N \O)^H \right)$
	for $H\leq G$
	one likewise has that 
	$N$ preserves and reflects essentially surjective maps.
	
	Next, we use Corollary \ref{SIMPLQUILL COR}
	to show that \eqref{PREQUIEQUIV EQ}
	is a Quillen adjunction.
	First,
	$\tau$ preserves cofibrations since,
	by Proposition \ref{TAUOTIMES_PROP},
	$\tau$ sends maps in (TC1),(TC2) to maps in (C1),(C2)
	and maps in (TC3) to isomorphisms (Remark \ref{TAUSC REM}).
	Second, to show that $N$ preserves fibrations between fibrant objects,
	by using the characterization in Theorem \ref{TAMEMS_THM}
	it suffices, thanks to an adjunction argument,
	to show that $\tau$
	sends the maps in (TA1),(TA2),(TA3)
	to trivial cofibrations. 
	Moreover, as we already know that $\tau$ preserves cofibrations, we need only show 	that $\tau$
	sends (TA1),(TA2),(TA3)
	to weak equivalences.
	The cases (TA2),(TA3) are again immediate 
	from Proposition \ref{TAUOTIMES_PROP},
	but (TA1) requires a different argument
	(which could also be used for (TA2),(TA3)).
	Writing $A \to B$ for a map in (TA1), one always has that $A,B$ are tame cofibrant, so that
	Corollary \ref{KEYEQUIV COR}
	and $2$-out-of-$3$ imply that 
	$N \tau A \to N \tau B$ is a weak equivalence
	and thus, since $N$ reflects weak equivalences,
	that $\tau A \to \tau B$ itself is a weak equivalence.
	This shows that \eqref{PREQUIEQUIV EQ}
	is a Quillen adjunction.
	
	Lastly, for the Quillen equivalence claim, 
	let $B \in \mathsf{PreOp}^G$ be tame cofibrant and
	$\mathcal{O} \in \mathsf{sOp}^G$ be fibrant.
	We must show the left map below is a weak equivalence iff 
	%its adjoint, 
	the right composite is.
	\[
	\tau B \to \mathcal{O},
	\qquad
	B \xrightarrow{\sim} N \tau B \to N \mathcal{O}
	\]
	This now follows from Corollary \ref{KEYEQUIV COR}
	and the fact that $N$ preserves and detects weak equivalences.
\end{proof}

\subsection{The homotopy coherent nerve and the proof of the main result}
\label{PFMNTHM SEC}

This section proves Theorem \ref{QE THM}.
We first recall the
$W_!\colon \mathsf{dSet}^G 
\rightleftarrows 
\mathsf{sOp}^G \colon hcN$
adjunction in \eqref{QE_EQ}.

In the categorical setting,
the left adjoint $W_!$ admits 
an explicit description, due to Dugger and Spivak \cite{DS11},
in terms of so called \emph{necklaces},
which we extend to the operadic setting in 
\cite{BP_WCONS}.
We now summarize the results therein we will need.

\begin{definition}
        \label{WU_DEF}
        For a tree $U \in \Omega$ there is 
        a simplicial operad
        $W(U) \in \mathsf{sOp}$
        with set of colors $\boldsymbol{E}(U)$
        and whose $n$-simplices evaluated at
        a $\boldsymbol{E}(U)$-corolla
        $\vect{C} = (C,\mathfrak{c})$ are
        \begin{equation} % \label{WU_EQ}
                W(U)_n(\vect C) =
                \begin{cases}
                        \left\{\text{factorizations }
                                C \xrightarrow{t} 
                                F_0 \xrightarrow{i,p} 
                                \cdots \xrightarrow{i,p}
                                F_n \xrightarrow{f,p} U
                        \right\}
                        &
                        \mbox{if $\boldsymbol{E}(C) \xrightarrow{\mathfrak{c}} \boldsymbol{E}(U)$ gives a map in $\Omega$}
                        \\
                        \varnothing
                        &
                        \mbox{otherwise.}
                \end{cases}
        \end{equation}
        where we label maps in $\Omega$ as
        $t/i/f/p$
        to indicate they are 
        tall/inner faces/faces/planar
        (cf. \S \ref{FORESTS_SEC}).
\end{definition}

\begin{remark}
	The factorization description in Definition \ref{WU_DEF} % \eqref{WU_EQ}
	reflects our approach in 
	\cite{BP_WCONS},
	which makes heavy use of the factorizations in 
	Proposition \ref{TREEFACT_PROP}.
	However, there is a simpler and more familiar description of $W(U)(\vect{C})$.
	If one lets
	$C \xrightarrow{t} U_C \xrightarrow{o,p} U$
	denote the unique ``tall map followed by planar outer face'' factorization, 
	repeated use of Proposition \ref{TREEFACT_PROP}
	shows that the 
	$F_0 \to \cdots \to F_n$
	strings in Definition \ref{WU_DEF} % \eqref{WU_EQ}
	are precisely the strings of planar inner faces of $U_C$.
	And since the latter are in bijection with strings of subsets of inner edges $\boldsymbol{E}^{\mathsf{i}}(U_C)$, 
	we have 
\begin{equation}\label{WU_EQ2}
	W(U)(\vect C) =
	\begin{cases}
	\Delta[1]^{\times \boldsymbol{E}^{\mathsf{i}}(U_{C})}
	\qquad
&
	\mbox{if $\boldsymbol{E}(C) \xrightarrow{\mathfrak{c}} \boldsymbol{E}(U)$ gives a map in $\Omega$}
\\
\varnothing
&
	\mbox{otherwise,}
	\end{cases}
\end{equation}
which recovers the description in \cite[\S 4]{CM13b}.
\end{remark}

\begin{remark}
	One neat feature of the description in 
	Definition \ref{WU_DEF} % \eqref{WU_EQ}
	is that the nerve $N W(U)$
	can be defined identically 
	(cf. \cite[Def. \ref{W-NWTNS DEF}]{BP_WCONS}; 
	compare with Remark \ref{NERVESIMDES REM}).
%	It is then straightforward to check that the nerve thus defined has the required functoriality and satisfies the strict Segal condition.
%	allowing us to deduce many properties of $W_!$
%	via systematic use of 
%	Proposition \ref{TREEFACT_PROP}.
	As an aside, one can verify directly that
	\cite[Def. \ref{W-NWTNS DEF}]{BP_WCONS}
	defines a preoperad satisfying the strict Segal condition
	\cite[Rem. \ref{W-NWTNS_REM}]{BP_WCONS},
	thereby inducing the operad structure on $W(U)$.
\end{remark}

The adjunction
\begin{equation} \label{SOPDSET_EQ}
	W_! \colon \dSet \rightleftarrows \sOp \colon h c N           
\end{equation}
is then defined by
\[
	W_!X = \colim_{\Omega[U] \to X}W(U)
	\qquad
	hcN\O(U) = \sOp(W(U), \O)
\]
with the analogous equivariant adjunction \eqref{QE_EQ}
obtained by taking $G$-objects.

For a $G$-tree 
$T = \amalg_i T_i = G \cdot_H T_{\**}$
in $\Omega_G$
we abbreviate
$W(T) = W_!(\Omega[T])$.
Note that, since the $T_i$ have disjoint edge sets, we thus have
\[
	W(T) = \amalg_i W(T_i) \simeq G \cdot_H W(T_{\**}).
\]

The following formalizes some key observations in the proof of
\cite[Prop. 4.5]{CM13b}.

\begin{lemma}\label{WLEFTQPUSH LEM}
      For $\eta \neq T \in \Omega^G$
      a tree with a $G$-action
      and $G$-subset $\emptyset \neq E \subseteq \boldsymbol{E}^{\mathsf{i}}(T)$,
      one has pushout diagrams in $\sOp^G$
      (for $C=\mathsf{lr}(T)$
      the corolla with the same number of leaves as $T$)
\begin{equation}\label{WLEFTQPUSH_EQ}
\begin{tikzcd}[column sep = small]
	\Omega(C) \otimes_{\mathfrak C_\bullet}
	\partial \left(\Delta[1]^{\times \boldsymbol{E}^{\mathsf{i}}(T)}\right)
	\ar{r} \ar{d}
&
	W_! \left(\partial \Omega[T]\right) 
	\arrow[d]
& % ----------
	\Omega(C) \otimes_{\mathfrak C_\bullet}
	\lambda^E \left(
	\Delta[1]^{\times \boldsymbol{E}^{\mathsf{i}}(T)}
	\right)
	\ar{r} \ar{d}
&
	W_! \left(\Lambda^E[T]\right) 
	\arrow[d]
\\
	\Omega(C) \otimes_{\mathfrak C_\bullet}
	\Delta[1]^{\times \boldsymbol{E}^{\mathsf{i}}(T)}
	\ar{r}
&
	W(T)
& % ----------
	\Omega(C) \otimes_{\mathfrak C_\bullet}
	\Delta[1]^{\times \boldsymbol{E}^{\mathsf{i}}(T)}
	\ar{r}
&
	W(T)
\end{tikzcd}
\end{equation}
      where
      $\partial \left(\Delta[1]^{\times \boldsymbol{E}^{\mathsf{i}}(T)}\right)
      \to
      \Delta[1]^{\times \boldsymbol{E}^{\mathsf{i}}(T)}$
      and
      $\lambda^E
      \left(
      \Delta[1]^{\times \boldsymbol{E}^{\mathsf{i}}(T)}
      \right)
      \to \Delta[1]^{\times \boldsymbol{E}^{\mathsf{i}}(T)}$
      are the pushout products
\[
	\left(
	\partial\Delta[1] \to \Delta[1]
	\right)^{\square \boldsymbol E^{\mathsf{i}}(T)},
\qquad
	\left(
	\partial \Delta[1] \to \Delta[1]
	\right)^{\square (\boldsymbol{E}^{\mathsf{i}}(T) \setminus E)}
	\square
	\left(
	\{1\} \to \Delta[1]
	\right)^{\square E}
\]
      with $G$-action induced by the action on $\boldsymbol{E}^{\mathsf{i}}(T)$.
\end{lemma}

\begin{proof}
Note first that, by \eqref{OMFFREE EQ} and the discussion above
\eqref{OTIMESC EQ},
	\[\Omega(C) \otimes_{\mathfrak{C}_{\bullet}} K
	\simeq
	\left(
	\mathbb{F} \Sigma_{\tau}[C]	
	\right)	\otimes_{\mathfrak{C}_{\bullet}} K
	\simeq
	\mathbb{F} (\Sigma_{\tau}[C] \times K).\]
Let us write $\vect{C} = (C,\mathfrak{c})$
for the corollas colored by either
$\boldsymbol{E}(C)$ or $\boldsymbol{E}(T)$
via the (planar) coloring sending leaves to leaves and the root to the root. By definition of $\mathbb{F}$, cf. \eqref{FREEOP_EQ},
one has
	\[
	\left(\mathbb{F} (\Sigma_{\tau}[C] \times K)\right)(\vect{C}) =
	(\Sigma_{\tau}[C] \times K)(\vect{C}) = K,
	\]
so that the horizontal maps in 
\eqref{WLEFTQPUSH_EQ}
are the unique maps given by the identity at the 
$\vect{C}$ level,
as per the calculations of
$W_!\left(\partial \Omega[T] \right)(\vect{C})$,
$W_!\left(\Lambda^E[T] \right)(\vect{C})$
in 
\cite[Examples \ref{W-WPARTIALT_EX},\ref{W-WPARTIALT2_EX}]{BP_WCONS}.

Lastly, to see that the squares in
\eqref{WLEFTQPUSH_EQ} are pushout squares note that,
after taking nerves, it is clear that the left vertical inclusions attach precisely those dendrices 
missing from the right vertical inclusions.
I.e., upon applying the nerve functor, \eqref{WLEFTQPUSH_EQ}
induces pushouts in $\mathsf{dSet}^G$.
The result now follows since the nerve reflects colimits,
due to being a fully faithful right adjoint.
\end{proof}

\begin{proposition}[{cf. \cite[Prop. 4.9]{CM13b}}]
      \label{W!_LEFTQ_PROP}
      $W_! \colon \dSet^G \rightleftarrows \sOp^G\colon hcN$
      is a Quillen adjunction.
\end{proposition}

\begin{proof}
	We verify the conditions in Corollary \ref{SIMPLQUILL COR}.
	Combining the pushouts in 
	Lemmas \ref{WLEFTQPUSH LEM} and \ref{OPTENSCOF_LEM}
	yields that 
	$W_!$ preserves cofibrations
	and sends $G$-inner anodyne extensions to trivial cofibrations.
	By adjunction,
	the latter claim implies that,
	if $f\colon \mathcal{O} \to \mathcal{P}$ is a fibration between fibrant objects,
	then 
	$hcN (f)\colon hcN \mathcal{O} \to hcN\mathcal{P}$
	is a $G$-inner fibration between $G$-$\infty$-operads.
	Thus, the remaining claim that 
	$hcN$ preserves fibrations between fibrant objects
	reduces to showing that the maps
	$\tau \iota^{\**} (h c N_d(f)^H) = \iota^{\**} \tau (h c N_d(f^H))$
	for $H \leq G$
	are isofibrations of (usual) categories
	(cf. Theorem \ref{DSETGMOD THM}).
	But since by definition 
	of fibration in $\mathsf{sOp}^G$
	the maps $\iota^{\**} \pi_0 f^H$ in \eqref{ESSSURJ EQ}
	are isofibrations,
	the result follows by the identification
	$\pi_0 \mathcal{Q} \simeq 
	\tau \left(h c N (\mathcal{Q}) \right)$ for fibrant operads $\mathcal{Q} \in \mathsf{sOp}$,
	cf. \cite[Prop. 4.8]{CM13b}.
\end{proof}

\begin{remark}\label{TWOHOMOP REM}
	The identification
	$\pi_0 \mathcal{Q} \simeq 
	\tau \left(h c N (\mathcal{Q}) \right)$ 
	in \cite[Prop. 4.8]{CM13b}
	used above
	identifies two procedures for discretizing 
	a simplicial operad $\mathcal{Q} \in \mathsf{sOp}$
	to obtain its \emph{homotopy operad}
	$\pi_0 \mathcal{Q} \in \mathsf{Op}$.

	Notably, however, neither 
	Proposition \ref{W!_LEFTQ_PROP}
	nor the original 
	\cite[Prop. 4.9]{CM13b}
	require the full strength
	of \cite[Prop. 4.8]{CM13b}, 
	as essential surjectivity depends only on the 
	the categories of unary operations within operads.
	Nonetheless, and in light of the fully faithful inclusions in 
	\eqref{ALLFULL EQ},
	it is natural to ask
	if \cite[Prop. 4.8]{CM13b}
	extends to the context of genuine operads.
	The answer to this question is affirmative,
	and is given by Proposition \ref{HOOPID_PROP}
	in Appendix \ref{HGEO AP}.
%	
%	In the equivariant context with 
%	$Z \in \mathsf{sOp}^G$,
%	and allowing $H\leq G$ to range over subgroups,
%	one obtains two procedures to produce a 
%	\emph{coefficient system of homotopy operads
%	$(G/H \mapsto \pi_0 \mathcal{Q}^H) \in \mathsf{Op}^{\mathsf{O}_G^{op}}$}.
%	However, such coefficient systems of operads ignore
%	all the non trivial norm maps
%	of $\mathcal{Q}$,
%	as they only depend on the fixed points
%	$\mathcal{Q}(\vect{C})^{H}$,
%	which correspond to demanding
%	$\Gamma \leq G \leq G \times \Sigma_n^{op}$
%	in \eqref{DKEQUIV_EQ}.
%	
%	In Appendix \ref{HGEO AP} we discuss a more complete equivariant generalization of \cite[Prop. 4.8]{CM13b},
%	by identifying to procedures of discretizing 
%	$\mathcal{Q} \in \mathsf{sOp}^G$
%	to obtain its so called 
%	``homotopy genuine equivariant operad'',
%	which is an enlargement of the coefficient systems of operads that also accounts for norm map data.
\end{remark}

We now turn to the proof of our main result,
Theorem \ref{QE THM}.

Recall that, given an object $X$ in a model category $\mathcal{M}$,
a simplicial frame for $X$ is a fibrant replacement
$c_!(X) \to \widetilde{X}(\bullet)$ of the constant 
simplicial object $c_!(X)$ in the Reedy model structure on $\mathcal{M}^{\Delta^{op}}$.
Moreover, if $X$ was already fibrant,
one is free to assume that $\widetilde{X}(0) = X$.

In addition, we will need some variants
of \cite[Prop. 4.5]{BP20},
regarding the categories
$\mathsf{sSet}^{\Delta^{op}}$
and 
$(\mathsf{sdSet}^G)^{\Delta^{op}}$.
We set some notation: for $X$ in one of these categories,
$X(m)$ (resp. $X_n$) denotes a level obtained by fixing the 
$(-)^{\Delta^{op}}$ simplicial coordinate
(resp. a simplicial coordinate in 
$\mathsf{sSet}$, $\mathsf{sdSet}^G$).
Further, $\delta^{\**}X$
denotes the object $X_n(n),n\geq 0$
given by the diagonal in the simplicial coordinates.

\begin{remark}\label{JOINTFIB REM}
	The proof of \cite[Prop. 4.5(ii)]{BP20}
	(or, alternatively, adapting \cite[Prop. 4.24(ii)]{BP20})
	shows that a Reedy fibrant $X(\bullet) \in \mathsf{sSet}^{\Delta^{op}}$
	is joint fibrant (i.e. its transpose swapping the two simplicial directions is also Reedy fibrant)
	iff the vertex maps $X(m) \to X(0)$
	(induced by the vertices $[0] \to [m]$)
	are Kan equivalences.
\end{remark}

\begin{lemma}\label{DIAGWE LEM}
If $X \in (\mathsf{sdSet}^G)^{\Delta^{op}}$ is
Reedy fibrant 
over the dendroidal Reedy model structure 
on $\mathsf{sdSet}^G$
and the vertex maps 
$X(m) \to X(0)$ are simplicial equivalences in $\mathsf{sdSet}^G$
then the two maps
\[
X(0) \to \delta^{\**} X \leftarrow X_0
\]
are also simplicial equivalences in $\mathsf{sdSet}^G$.
\end{lemma}

\begin{proof}
By definition,
we need to show that, for each $T \in \Omega_G$, the maps 
\[
	X(0)(\Omega[T]) \to 
	\delta^{\**} X (\Omega[T]) \leftarrow
	X_0 (\Omega[T])
\]	
are Kan equivalences in $\mathsf{sSet}$.
Both of these equivalences will follow from 
\cite[Prop. 4.5(iv)]{BP20}
provided we show that
$X(\Omega[T])$ is a joint fibrant object in $\mathsf{sSet}$.
And since the vertex maps 
$X(\Omega[T])(m) \to X(\Omega[T])(0)$
are Kan equivalences by assumption on $X$,
by Remark \ref{JOINTFIB REM} it remains only to check that
$X(\Omega[T])$
is Reedy fibrant in $\mathsf{sSet}^{\Delta^{op}}$.
For this last claim, 
note first that the Reedy fibrancy 
assumption on $X$ is that the matching maps 
(see, e.g., \cite[Thm. A.8]{BP20})
$X(m) \to M_m X(\bullet)$
are dendroidal fibrations in 
$\mathsf{sdSet}^G$.
Unpacking definitions,
this means that, for every normal monomorphism 
$A \to B$ in $\mathsf{dSet}^G$, the maps
\[
X(m)(B) \to 
X(m)(A) \underset{{M_m X(\bullet)(A)}}{\times} M_m X(\bullet)(B)
\]
are Kan fibrations.
But setting $A \to B$ to be $\emptyset \to \Omega[T]$
we get that the maps
$X(\Omega[T])(m) \to 
M_m X(\Omega[T])(\bullet)$
are Kan fibrations, i.e. that 
$X(\Omega[T])$ is Reedy fibrant
in $\mathsf{sSet}^{\Delta^{op}}$.
\end{proof}

\begin{proof}[Proof of Theorem \ref{QE THM}]
Consider the square of adjunctions on the left below 
(where we depict only the right adjoints).
We already know that all four adjunctions therein are Quillen,
and that those adjunctions other than the
$(W_!,hcN)$
adjunction are Quillen equivalences
(\cite[Thms. 4.30 and 4.41]{BP20} and Theorem \ref{PREQUIEQUIV THM}).
Next, 
we consider the induced diagram of homotopy categories
and derived functors on the right.
Crucially,
note that while 
the right Quillen functors $N$ and $hcN$
must be right derived, 
the left Quillen functors $\gamma^{\**}$ and $c_!$ do not, 
since they preserve all weak equivalences
(this holds for $\gamma^{\**}$ by definition,
cf. Theorem \ref{TAMEMS_THM},
and for $c_!$ since it sends weak equivalences to dendroidal equivalences, cf. Theorem \ref{JB_THM}\ref{SDEQUIV_LBL}).
\begin{equation}\label{SQINPROOF EQ}
\begin{tikzcd}
	\mathsf{PreOp}^G_{tame} 
&
	\mathsf{sOp}^G 
	\ar{l}[swap]{N}
	\ar{d}{hcN}
&&%%%
	\Ho \mathsf{PreOp}^G 
	\ar{d}[swap]{\gamma^{\**}}{\sim}
&
	\Ho \mathsf{sOp}^G 
	\ar{l}[swap]{R N}{\sim}
	\ar{d}{R hcN}
\\
	\mathsf{sdSet}^G
	\ar{u}{\gamma_{\**}}
	\ar{r}[swap]{c^{\**}}
&
	\mathsf{dSet}^G
&&%%%
	\Ho \mathsf{sdSet}^G
&
	\Ho \mathsf{dSet}^G
	\ar{l}{c_!}[swap]{\sim}
\end{tikzcd}
\end{equation}
Recalling that a Quillen adjunction
is a Quillen equivalence iff the induced adjunction of homotopy categories is an equivalence adjunction,
the desired claim that
$(W_!,hcN)$ is a Quillen equivalence will thus follow
provided we show that the right square in 
\eqref{SQINPROOF EQ}
commutes up to natural isomorphism.
In other words, we reduce to showing that
for fibrant operads
$\O \in \mathsf{sOp}^G$
there is a natural zigzag of joint equivalences
between 
$\gamma^{\**} N \O$ and
$c_! hcN \O$.

We now discuss this zigzag. Assume 
$\mathcal{O} \in \mathsf{sOp}^G$ is fibrant.
First, choose a (functorial) fibrant simplicial frame
$\widetilde{\mathcal{O}}(\bullet) \in (\mathsf{sOp}^G)^{\Delta^{op}}$, where we assume $\widetilde{\mathcal{O}} (0) = \mathcal{O}$.
Next, let 
$\gamma^{\**} N \widetilde{\mathcal{O}}(\bullet) 
\to \widetilde{Q}(\bullet)$
be a Reedy fibrant replacement in  
$(\mathsf{sdSet}^G)^{\Delta^{op}}$.
We note that 
both $\widetilde{\O}$ and $\widetilde{Q}$
have two simplicial directions:
the frame direction, whose levels are written as
$\widetilde{\O}(m),\widetilde{Q}(m)$,
and an internal direction
(determined by the simplicial levels in $\mathsf{sOp}^G,\mathsf{sdSet}^G$),
whose levels are written 
$\widetilde{\O}_n,\widetilde{Q}_n$.
Our desired zigzag of joint equivalences in $\mathsf{sdSet}^G$
will have the form below.
\begin{equation}\label{BIGZIG EQ}
\begin{tikzcd}[column sep =20pt]
	\gamma^{\**} N \mathcal{O} =
	\gamma^{\**} N \widetilde{\mathcal{O}} (0)
	\ar{r}{\sim}[swap]{(a)}
&
	\widetilde{Q}(0)
	\ar{r}{\sim}[swap]{(b)}
&
	\delta^{\**} \widetilde{Q}
	\ar[leftarrow]{r}{\sim}[swap]{(c)}
&	
	\widetilde{Q}_0
	\ar[leftarrow]{r}{\sim}[swap]{(d)}
&
	\left(\gamma^{\**} N \widetilde{\mathcal{O}}\right)_0
		\ar{r}{\sim}[swap]{(e)}
&
	hcN \widetilde{\mathcal{O}}
	\ar[leftarrow]{r}{\sim}[swap]{(f)}
&
	c_{!} hcN \mathcal{O}
\end{tikzcd}
\end{equation}

Firstly, the map (a) is a joint equivalence by definition of 
$\widetilde{Q}$.

Next, the maps 
(b),(c) are joint equivalences (in fact, simplicial equivalences)
by Lemma \ref{DIAGWE LEM}. 
Here, we note that though, a priori,
the maps $\widetilde{Q}(m) \to \widetilde{Q}(0)$
are only joint equivalences,
they are in fact simplicial equivalences,
since the levels $\widetilde{Q}(m)$ are joint fibrant,
cf. Theorem \ref{JB_THM}\ref{SFIB_JEQ_LBL},
\cite[Lemma A.29(i)]{BP20}.

To see that (d) is a joint equivalence,
note that one has identifications 
($\bullet$ tracks the simplicial index)
\begin{equation}\label{EQDMAPSP EQ}
\begin{tikzcd}[row sep=0,column sep=7pt]
	\widetilde{Q}_0(\Omega[T])(\bullet)
	\ar[equal]{r}
&
	\mathsf{sdSet}^G(\Omega[T],\widetilde{Q}(\bullet))
	\ar[equal]{r}
&
	\mathsf{PreOp}^G(\Omega[T],\gamma_{\**}\widetilde{Q}(\bullet))
\\
	\left(\gamma^{\**} N \widetilde{\mathcal{O}}\right)_0(\Omega[T])(\bullet)
	\ar[equal]{r}
&
	\mathsf{sdSet}^G(\Omega[T],\gamma^{\**} N \widetilde{\mathcal{O}}(\bullet))
	\ar[equal]{r}
&
	\mathsf{PreOp}^G(\Omega[T], N \widetilde{\mathcal{O}}(\bullet))
\end{tikzcd}
\end{equation}
in $\sSet$ for each $T \in \Omega_G$.
Next, since the counit maps 
$\gamma^{\**} \gamma_{\**} \widetilde{Q}(\bullet) \to \widetilde{Q}(\bullet)$
are joint equivalences
(due to $\gamma^{\**} \gamma_{\**} \widetilde{Q}(\bullet)$ already computing a derived functor)
in 
$\mathsf{sdSet}^G$,
one has that the maps
$N\widetilde{\O}(\bullet) \to \gamma_{\**}
\widetilde{Q}(\bullet)$
are joint equivalences in $\mathsf{PreOp}^G$.
Therefore, and since $\gamma_{\**}$ is right Quillen,
both $N\widetilde{\O}$ and $\gamma_{\**}\widetilde{Q}$
are simplicial frames for $N \O$
in the tame model structure
$\mathsf{PreOp}^G_{tame}$.
Thus, the fact that (d) is a joint equivalence
follows since both halves of \eqref{EQDMAPSP EQ}
compute the mapping space from 
$\Omega[T]$ to $N \O$
in the tame model structure
(this uses the observation that
$\Omega[T]$ is tame cofibrant, cf. Lemma \ref{OMEGATTAME_LEM}).

For (e), in which case we need also describe the map,
we consider the identifications
\begin{equation}\label{EQEMAPSP EQ}
\begin{tikzcd}[row sep=0,column sep=7pt]
	\left(\gamma^{\**} N \widetilde{\mathcal{O}}\right)_0(\Omega[T])(\bullet)
	\ar[equal]{r}
&
	\mathsf{PreOp}^G(\Omega[T], N \widetilde{\mathcal{O}}(\bullet))
	\ar[equal]{r}
&
	\mathsf{sOp}^G(\Omega(T), \widetilde{\mathcal{O}}(\bullet))
\\
	\left(hcN \widetilde{\mathcal{O}} \right)(\Omega[T])(\bullet)
	\ar[equal]{r} 
&
	\mathsf{sOp}^G(W(T),  \widetilde{\mathcal{O}}(\bullet))
\end{tikzcd}
\end{equation}
in $\mathsf{sSet}$ for each $T \in \Omega_G$.
The map (e) is then induced by the unique color preserving map
$W(T) \to \Omega(T)$ (recall that the mapping spaces of $\Omega(T)$ are either $\**$ or $\emptyset$).
Thus, since $W(T) \to \Omega(T)$ is a weak equivalence of cofibrant operads in $\sOp^G$
and $\widetilde{\O}$ is a simplicial frame for $\O$,
it follows that 
\eqref{EQEMAPSP EQ}
likewise computes mapping spaces, 
and thus (e) is indeed a weak equivalence. 

Lastly, that (f)
is a weak equivalence follows from the identification
$c_! hcN \mathcal{O} = hcN c_! \mathcal{O}$,
together with the fact that
$c_! \mathcal{O}(\bullet) \to \widetilde{\O}(\bullet)$
is a levelwise equivalence of levelwise fibrant operads
(where ``levelwise'' means ``for fixed $\bullet$'')
and $hcN \colon \sOp^G \to \dSet^G$ being right Quillen.
\end{proof}

\begin{remark}
	The previous proof is a close variation of 
	the proof of \cite[Thm. 8.14]{CM13b},
	although the equivariant context 
	forces us to use a more formal argument.
	
	More precisely, the given proof of \cite[Thm. 8.14]{CM13b}
	relies on 
	\cite[Thm 5.9(v)]{CM13b},
	which states that a preoperad
	$X \in \mathsf{PreOp}$
	is equivalent in $\mathsf{sdSet}$
	to the presheaf
	$T \mapsto \mathsf{Map}(\Omega[T],X)$
	for $T \in \Omega$
	(here $\mathsf{Map}(-,-)$
	denotes the homotopy space of maps).
	However, in the equivariant context
	the assignment 
	$T \mapsto \mathsf{Map}(\Omega[T],X)$
	for $T \in \Omega_G$
	does not produce a presheaf in 
	$\mathsf{sdSet}^G$
	(since the levels of such presheaves are indexed by
	$U \in \Omega$)
	but rather a presheaf in the category
	$\mathsf{sdSet}_G$
	of simplicial objects on $\mathsf{dSet}_G$ (cf. \eqref{UPSILONADJ EQ}),
	which does not appear in 
	\eqref{SQINPROOF EQ}.
	
	As such, rather than attempt to formulate and use
	an analogue of \cite[Thm 5.9(v)]{CM13b},
	our proof replaces the role of that result with 
	an explicit analysis of the simplicial framings
	needed to define the homotopy mapping spaces
	appearing in \cite[Thm 5.9(v)]{CM13b}.
\end{remark}

\begin{remark}
	There is a natural way to attempt to simplify the zigzag
	\eqref{BIGZIG EQ}
	in the a previous proof. 
	Namely, one may attempt to replace the first four maps therein
	with the simpler two map zigzag
\begin{equation}\label{SIMPZIG EQ}
	\gamma^{\**} N \widetilde{\O}(0)
\to 
	\delta^{\**} \gamma^{\**} N \widetilde{\O}
\leftarrow
	\left(\gamma^{\**} N \widetilde{\O}\right)_0.
\end{equation}
As it turns out, it can be shown that 
\eqref{SIMPZIG EQ} consists of weak equivalences, 
but our argument for this is substantially more involved than the argument given for \eqref{BIGZIG EQ}.

Briefly, if $\widetilde{X}$
in $\left(\mathsf{PreOp}^G_{tame}\right)^{\Delta^{op}}$
is a simplicial frame for $X$ in 
$\mathsf{PreOp}^G$,
one can find a levelwise simplicial equivalence
$\widetilde{X} \xrightarrow{\sim} \widetilde{Y}$
with $\widetilde{Y}$
a simplicial frame in
$\left(\mathsf{PreOp}^G_{normal}\right)^{\Delta^{op}}$.
One then has that 
$\widetilde{X}(0) \to \delta^{\**}\widetilde{X} \leftarrow \widetilde{X}_0$ 
consists of weak equivalences iff
$\widetilde{Y}(0) \to \delta^{\**}\widetilde{Y} \leftarrow \widetilde{Y}_0$ does,
and the latter can be shown by following the
(rather involved) proof of 
\cite[Prop. 5.41]{BP20}
with $\widetilde{Y}$ taking the role
of $X^{J^{\bullet}}$ therein.

As a side note, 
\cite[Prop. 5.41]{BP20} is one of the keys to the proof of
\cite[Thm. 5.48]{BP20}, 
which establishes the DK description of weak equivalences between fibrant objects in
$\mathsf{PreOp}^G$, $\mathsf{sdSet}^G$,
so our proof of Theorem \ref{QE THM}
does indirectly rely on \cite[Prop. 5.41]{BP20}.
\end{remark}

\section{Nerves of free extensions are homotopy pushouts}
\label{KEYRES SEC}

This section will prove the following key lemma,
which is the equivariant analogue of
\cite[Prop. 3.2]{CM13b}.
For a comparison with the work in \cite[\S 3]{CM13b},
see Remarks \ref{CHAREDCOMP REM}, \ref{CH01SUCKS REM}.

\begin{lemma}\label{KEYPR LEM}
	Suppose that $\mathcal{O} \in \mathsf{Op}^{G}$
	is $\Sigma$-cofibrant.
	Further, let $C \in \Sigma_G$ be any $G$-corolla and consider 
	a pushout in 
	$\mathsf{Op}^{G} = \mathsf{Op}^{G}(\mathsf{Set})$ 
	of the form below.
	\begin{equation}\label{PUSHOUTPROP EQ}
	\begin{tikzcd}
	\partial \Omega(C) \ar{r} \ar{d} & \mathcal{O} \ar{d}
	\\
	\Omega(C) \ar{r} & \mathcal{P}
	\end{tikzcd}
	\end{equation}
	Then the induced map
	\begin{equation}\label{ANODYNE MAP}
	\Omega[C] \amalg_{\partial \Omega[C]} N\mathcal{O} \to N\mathcal{P}
	\end{equation}
	is $G$-inner anodyne.
\end{lemma}

\subsection{The characteristic edge lemma}

The proof of Lemma \ref{KEYPR LEM}
will use the \emph{characteristic edge lemma}
\cite[Lemma 3.4]{BP20},
repackaged below as 
Lemma \ref{CHAREDGE LEM}
in a form better suited for our application.
We start with some notation.
For a discussion of (non-)degenerate 
dendrices in a dendroidal set,
see \cite[Prop. 5.62]{Per18}.

\begin{notation}
	Let $Y \in \mathsf{dSet}^G$ be a $G$-dendroidal set and 
	$y \colon \Omega[U^y] \to Y$
	a dendrex, $U^y \in \Omega$.
	
	We write $\langle y \rangle = y\left(  \Omega[U^y] \right)$
	and refer to
	$\langle y \rangle \subseteq Y$
	as the \emph{principal subpresheaf generated by $y$}.
	
	Moreover, if some (and thus any)
	non-degenerate representative $y$ is free 
	with respect to the $\mathsf{Aut}(U^y)$-action (via precomposition),
	we say $y$ and $\langle y \rangle$ are \emph{$\Sigma$-free}.
	If all dendrices $y$ are $\Sigma$-free, we say $Y$ itself is \textit{$\Sigma$-free}.
	
	Given a map of trees $V \to U^y$ we write $\partial_V y$ for the composite $\Omega[V] \to \Omega[U^y] \xrightarrow{y} Y$.
\end{notation}

\begin{remark}
	Note that
	$\langle y \rangle = \langle \bar{y} \rangle$
	iff $y,\bar{y}$ are both degeneracies of a common non-degenerate dendrex.
	In particular, if the chosen representatives $y,\bar{y}$ are both nondegenerate,
	there must exist an isomorphism
	$\varphi \colon U^y \xrightarrow{\simeq} U^{\bar{y}}$
	(which is unique if 
	$\langle y \rangle$ is $\Sigma$-free)
	such that $y= \bar{y} \circ \varphi$.
\end{remark}

\begin{notation}
	Given a $\Sigma$-free $\langle y \rangle$,
	a \emph{coherent inner edge set $E^{\langle y \rangle}$ for $\langle y \rangle$}
	is a collection of subsets 
	$E^y \subseteq \boldsymbol{E}^{\mathsf{i}}(U^y)$
	for each non-degenerate representative $y$ of $\langle y \rangle$, and such that 
	$E^{\bar{y}}  = \varphi \left(E^y \right)$
	for the unique $\varphi$ with $y= \bar{y} \circ \varphi$.
	Note that $E^{\langle y \rangle} = \left\{E^y \right\}$
	is entirely determined by any of the $E^y$.
	%
	%Note that for any representatives
	%$x,\bar{x}$ one has
	%$\langle \partial_{U^x - E^x} x\rangle
	%=
	%\langle \partial_{U^{\bar{x}} - E^{\bar{x}}} \bar{x}\rangle$,
	%so that we abbreviate this presheaf as
	%$\partial_{E^{\langle x \rangle}} \langle x\rangle$.
	%Moreover, given coherent inner edge sets
	%$E^{\langle x \rangle},F^{\langle x \rangle}$
	%we write
	%$E^{\langle x \rangle} \subseteq F^{\langle x \rangle}$
	%if
	%$E^{x} \subseteq F^{x}$
	%for some (and thus all) non-degenerate representatives $x$.
\end{notation}

\begin{remark}\label{PRINGACT REM}
	For $Y \in \mathsf{dSet}^G$, the group $G$
	acts on dendrices by postcomposition,
	i.e. $gy$ is the composite
	$\Omega[U^y] \xrightarrow{y} Y \xrightarrow{g} Y$.
	In particular, note that $U^{gy} = U^{y}$.
	Moreover, this action extends to principal subpresheaves as
	$g \langle y \rangle = \langle g y \rangle$.
	
	As such, if $\langle y\rangle$ is $\Sigma$-free, a coherent inner edge set 
	$E^{\langle y \rangle} = \{E^y \subseteq \boldsymbol{E}^{\mathsf{i}}(U^y)\}$
	for $\langle y \rangle$
	gives rise to a coherent inner edge set 
	$g E^{\langle y \rangle} = \{E^y \subseteq \boldsymbol{E}^{\mathsf{i}}(U^{gy})\}$
	for $g\langle y \rangle$
	with the same edge sets $E^y$.
\end{remark}

The following essentially replicates \cite[Def. 3.1]{BP20} as generalized in \cite[Rem. 3.7]{BP20},
except with dendrices
$y \colon \Omega[U^y] \to Y$
mostly replaced with the principal presheaves
$\langle y \rangle \subseteq Y$. 
The reformulation of (Ch0.2) and the descending chain condition
are discussed in Remarks \ref{CH02 REM}, \ref{DCC REM}.

\begin{definition}\label{CHAREDGE DEF}
	Let $f\colon X \to Y$ be a monomorphism in 
	$\mathsf{dSet}^G$ and 
	$\left\{ \langle y \rangle\right\}$
	a set of $\Sigma$-free principal subpresheaves of $Y$. 
	Suppose further that 
	$\left\{ \langle y \rangle \right\}$ is equipped 
	with a poset structure compatible with the $G$-action
	in Remark \ref{PRINGACT REM}
	and which satisfies the descending chain condition.
	For each $\langle y \rangle$ denote
	\[
	X_{< \langle y \rangle} = X \cup 
	\bigcup_{\langle\bar{y}\rangle < \langle y \rangle} \langle \bar{y} \rangle
	\]
	Given a coherent inner edge set 
	$
	\Xi^{\langle y \rangle} =
	\left\{ \Xi^y \subseteq \boldsymbol{E}^{\mathsf{i}}(U^y)\right\}$,
	non-degenerate representative
	$y \colon \Omega[U^y] \to Y$, and a subface $V \hookrightarrow U^y$,
	we write
	$\Xi^y_V = \Xi^y \cap \boldsymbol{E}^{\mathsf{i}}(V)$.
	
	We say
	$
	\left\{ \Xi^{\langle y \rangle} \right \} 
	%=\left\{ \Xi^y \subseteq \boldsymbol{E}^{\mathsf{i}}(U^y)\right\}
	$
	is a \emph{characteristic inner edge collection 
	of $\left\{ \langle y \rangle \right\}$ with respect to $X$} if,
	for some (and thus any) choice of non-degenerate representatives $y\colon \Omega[U^y] \to Y$,
	one has that:
	\begin{enumerate}
		\item[(Ch0.1)] $y \colon \Omega[U^y] \to Y$ is injective away from
		$y^{-1}\left( X_{< \langle y \rangle} \right)$; 
		\item[(Ch0.2)]
		$\{\langle y\rangle\}$ and
		$\{\Xi^{\langle y \rangle}\}$ are $G$-equivariant, in the sense that
		$g\langle y\rangle \in \{\langle y\rangle\}$ and 
		$g \Xi^{\langle y \rangle} =
		\Xi^{g \langle y \rangle}$,
		i.e. $\Xi^y = \Xi^{gy}$;
		\item[(Ch1)] if $V \hookrightarrow U^y$ is an outer face and $\Xi^y_V = \emptyset$,
		then $\langle \partial_V (y) \rangle \subseteq X_{< \langle y \rangle}$;
		\item[(Ch2)] if $V \hookrightarrow U^y$ is any face and
		$\langle \partial_{V-\Xi^y_V} (y)\rangle \subseteq X$,
		then
		$\langle \partial_V (y)\rangle \subseteq X_{< \langle y \rangle}$;
		\item[(Ch3)] if $\langle \bar{y} \rangle \not \geq \langle y \rangle$,
		$V \hookrightarrow U^y$,
		and
		$\langle \partial_{V-\Xi^y_V} (y)\rangle \subseteq \langle \bar{y} \rangle$,
		then
		$\langle \partial_V (y)\rangle \subseteq X_{< \langle y \rangle}$.
	\end{enumerate}
\end{definition}

\begin{remark}\label{CH02 REM}
	In \cite[Rem. 3.7]{BP20} the role of each presheaf 
	$\langle y \rangle$ is played by a special chosen representative,
	which we here denote by
	$y^{\mathsf{pl}} \in \langle y \rangle$. 
	The motivation for this is that in some key examples, such as in \cite[Ex. 3.9]{BP20}, one can choose preferred ``planar representatives for principal presheaves'',
	allowing for a pictorial depiction of the dendrices and poset as in
	\cite[Fig. 3.1]{BP20}. 
	
	There is then a bijection $\{\langle y \rangle\} = \{ y^{\mathsf{pl}}\}$ between principal presheaves and the set of representatives, but while the former has a $G$-action the latter a priori does not (as $gy^{\mathsf{pl}}$ may not be planar). Translating the $G$-action along this bijection one has that the action of $g$ on $y^{\mathsf{pl}}$ gives
	$(g y^{\mathsf{pl}})^{\mathsf{pl}}$ and (ii),(iii) in 
	(Ch0.2) of \cite[Rem. 3.7]{BP20} precisely 
	encode this action on planar representatives.
\end{remark}

\begin{remark}\label{DCC REM}
	Recall that a poset satisfies the 
	\emph{descending chain condition}
	
        As such, while  \cite[Lemma 3.4]{BP20}
	assumed that the poset $\{\langle y \rangle\}$ was finite,
	since its proof follows by iteratively adding elements to $G$-equivariant convex subsets of the poset (cf. the last paragraph of the proof in loc. cit.), 
	the argument generalizes to any poset satisfying the descending chain condition.
\end{remark}

\begin{lemma}[{cf. \cite[Lemma 3.4]{BP20}}]
	\label{CHAREDGE LEM}
	If
	$
	\left\{ \Xi^{\langle y \rangle} \right \} 
	%=\left\{ \Xi^y \subseteq \boldsymbol{E}^{\mathsf{i}}(U^y)\right\}
	$
	is a \emph{characteristic inner edge collection 
	of $\left\{ \langle y \rangle \right\}$ with respect to $X$} then
	\begin{equation}\label{CHAREDGE EQ}
	X \to X \cup \bigcup_{\{\langle y \rangle\}} \langle y \rangle
	\end{equation}
	is $G$-inner anodyne.
\end{lemma}

\subsection{Proof of the key lemma}

This section will prove Lemma \ref{KEYPR LEM}
by applying the characteristic edge lemma, 
Lemma \ref{CHAREDGE LEM}.
This will require a fair bit of preparation,
starting with a description of the pushout operad
$\mathcal{P}$ in \eqref{PUSHOUTPROP EQ}.

First, we let $\mathfrak{C} = \mathfrak{C}_{\O}$
and write 
$\vect{C} = (C,\mathfrak{c})$
for 
$\mathfrak{c} \colon
\boldsymbol{E}(C) \to \mathfrak{C}$
the map of colors induced
by the top horizontal map in \eqref{PUSHOUTPROP EQ}.
One then has identifications
(cf. \eqref{OMFFREE EQ} and Remark \ref{CHECKF REM})
\[
\check{\mathfrak{c}}_{!} \Omega(C) 
\simeq 
\check{\mathfrak{c}}_{!} 
\left( \mathbb{F} \Sigma_{\tau}[C] \right)
\simeq 
\mathbb{F} 
\left(\mathfrak{c}_{!}  \Sigma_{\tau}[C] \right)
\simeq
\mathbb{F} 
\left( \Sigma_{\mathfrak{C}}[\vect{C}] \right)	
\]
\[
\check{\mathfrak{c}}_{!} \partial \Omega(C) 
\simeq 
\check{\mathfrak{c}}_{!} 
\left( \mathbb{F} \emptyset_{\boldsymbol{E}(C)} \right)
\simeq 
\mathbb{F} 
\left(\mathfrak{c}_{!} \emptyset_{\boldsymbol{E}(C)} \right)
\simeq
\mathbb{F} 
\left( \emptyset_{\mathfrak{C}} \right)	
\]
(for $\emptyset_{\mathfrak{C}}$
the initial object in 
$\mathsf{sSym}^G_{\mathfrak{C}}$)
allowing us to rewrite \eqref{PUSHOUTPROP EQ}
as the following pushout in 
$\mathsf{Op}^G_{\mathfrak{C}}$
\begin{equation}
\begin{tikzcd}
\mathbb{F} ( \emptyset_{\mathfrak{C}} ) \ar{r} \ar{d}
&
\mathcal{O} \ar{d}
\\
\mathbb{F} \left( 
\Sigma_{\mathfrak{C}}[\vect{C}] \right) \ar{r}
&
\mathcal{P}.
\end{tikzcd}
\end{equation}
By \cite[Lemma \ref{OC-OURE LEM}]{BP_FCOP}
with $u\colon X \to Y$
the map $\emptyset_{\mathfrak{C}} \to \Sigma_{\mathfrak{C}}[\vect{C}]$
(and as further detailed in 
\cite[Rem. \ref{OC-FILTPUSH REM}]{BP_FCOP})
one then has,
for each $\mathfrak{C}$-corolla
$\vect{D} \in \Sigma_{\mathfrak{C}}$,
a formula
\begin{equation}\label{PUSHOPPR EQ}
\mathcal{P}(\vect{D}) 
	\simeq 
\coprod_{
	[\vect{T}] \in \mathsf{Iso}
	\left(\vect{D} \downarrow \Omega_{\mathfrak{C}}^a \right)
}
\left(
\prod_{v \in \boldsymbol{V}^{ac}(T)} \mathcal{O}(\vect{T}_v)
\times
\prod_{v \in \boldsymbol{V}^{in}(T)} \Sigma_{\mathfrak{C}}[\vect{C}](\vect{T}_v)
\right)
\cdot_{\mathsf{Aut}_{\Omega^a_{\mathfrak{C}}}(\vect{T})} \mathsf{Aut}_{\Sigma_{\mathfrak{C}}}(\vect{D})
\end{equation}
(where,
since $X = \emptyset_{\mathfrak{C}}$,
it is likewise always
$Q^{in}_{\vect{T}}[u] = \emptyset_{\mathfrak{C}}$
in \cite[Lemma \ref{OC-OURE LEM}]{BP_FCOP}).

To explain the notation in \eqref{PUSHOPPR EQ},
we need to recall the category
$\Omega^a$ of \emph{alternating trees} 
\cite[Def. 5.52]{BP21}.
First, we consider trees whose vertices are partitioned 
$\boldsymbol{V}(T) = 
\boldsymbol{V}^{ac}(T) \amalg \boldsymbol{V}^{in}(T)$
into $\bullet$-labeled vertices, called active,
and $\circ$-labeled vertices, called inert,
as in Example \ref{ALTTREES_EX}. % \eqref{ALTTREES EQ}.
Such trees are called \emph{alternating} if each edge is adjacent to exactly one $\bullet$-labeled vertex
(i.e. if adjacent vertices have different labels $\bullet,\circ$ and vertices adjacent to the root and leaves are $\bullet$-labeled).
% In \eqref{ALTTREES EQ}, $T$ is alternating while $S$ is not.
$\Omega^a_{\mathfrak{C}}$
is then defined by coloring edges, 
cf. Definition \ref{CFOREST_DEF}.
Further, as in \eqref{FREEOP_EQ},
one has a leaf-root functor
$\mathsf{lr} \colon 
\Omega^a_{\mathfrak{C}} \to \Sigma_{\mathfrak{C}}$
given by 
$\vect{T} = (T,\mathfrak{c})
\mapsto
(\mathfrak{c}(l_1),\cdots,\mathfrak{c}(l_n);
\mathfrak{c}(r))$
for $l_i$ the leaves and $r$ the root,
which gives rise to the undercategory
$\vect{D} \downarrow \Omega^a_{\mathfrak{C}}$
appearing in \eqref{PUSHOPPR EQ}.

\begin{example}
        In the following, $T$ is alternating while $S$ is not.
        
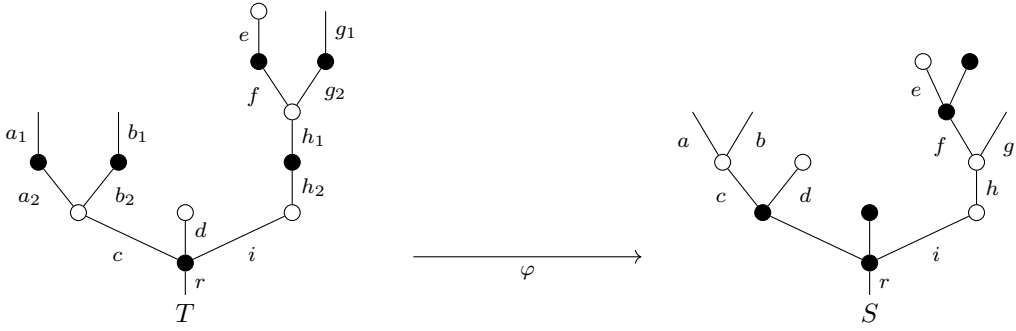
\label{ALTTREES_EX}
        \begin{equation}% \label{ALTTREES EQ}
                \begin{tikzpicture}[grow=up,auto,level distance=1.9em,
                        every node/.style = {font=\footnotesize},
                        dummy/.style={circle,draw,inner sep=0pt,minimum size=2.1mm}]
                        \tikzstyle{level 2}=[sibling distance = 4em]
                        \tikzstyle{level 3}=[sibling distance = 3em]
                        \tikzstyle{level 4}=[sibling distance = 4em]
                        \tikzstyle{level 5}=[sibling distance = 2.5em]
                        \node at (0,0) [font=\normalsize]{$T$}
                        child{node [dummy,fill = black] {}
                          child{node [dummy,fill=white] {}
                            child{node [dummy,fill=black] {}
                              child{node [dummy,fill=white] {}
                                child{node [dummy,fill = black] {}
                                  child{
                                    edge from parent node [swap] {$g_1$}}
                                  edge from parent node [swap,near end] {$g_2$}}
                                child{node [dummy,fill = black] {}
                                  child{node [dummy,fill=white] {}
                                    edge from parent node  {$e$}}
                                  edge from parent node [near end] {$f$}}
                                edge from parent node [swap] {$h_1$}}
                              edge from parent node [swap] {$h_2$}}
                            edge from parent node [swap] {$i$}}
                          child{node [dummy,fill=white] {}
                            edge from parent node [swap, near end] {$d\phantom{1}$}}
                          child{node [dummy,fill=white] {}
                            child{node [dummy,fill=black] {}
                              child{
                                edge from parent node [swap] {$b_1$}}
                              edge from parent node [swap,near end] {$b_2$}}
                            child{node [dummy,fill=black] {}
                              child{
                                edge from parent node {$\phantom{b}a_1$}}
                              edge from parent node [near end] {$\phantom{b}a_2$}}
                            edge from parent node {$\phantom{1}c$}}
                          edge from parent node [swap] {$r$}};
                        
                        \begin{scope}[level distance=1.9em]
                                \tikzstyle{level 2}=[sibling distance = 4em]
                                \tikzstyle{level 3}=[sibling distance = 3em]
                                \tikzstyle{level 4}=[sibling distance = 2.25em]
                                \tikzstyle{level 5}=[sibling distance = 1.75em]
                                \node at (9,0) [font = \normalsize] {$S$}
                                child{node [dummy,fill = black] {}
                                  child{node [dummy,fill = white] {}
                                    child{node [dummy,fill=white] {}
                                      child{
                                        edge from parent node [swap] {$g$}}
                                      child{node [dummy,fill = black] {}
					child{node [dummy,fill=black] {}
					}
					child{node [dummy,fill=white] {}
                                          edge from parent node [near end] {$e$}}
                                        edge from parent node [near end] {$f$}}
                                      edge from parent node [swap] {$h\phantom{1}$}}
                                    edge from parent node [swap] {$i$}}
                                  child{node [dummy,fill = black] {}
                                  }
                                  child{node [dummy,fill = black] {}
                                    child{node [dummy,fill=white] {}
                                      edge from parent node [swap, near end] {$d$}}
                                    child{node [dummy,fill=white] {}
                                      child{
                                        edge from parent node [swap,near end]{$b$}}
                                      child{
                                        edge from parent node [near end]{$\phantom{b}a$}}
                                      edge from parent node [near end] {$\phantom{d}c$}}
                                  }
                                  edge from parent node [swap] {$r$}};
                        \end{scope}
                        \draw [->] (3,0.75) -- node [swap] {$\varphi$} (6,0.75);
                \end{tikzpicture}
        \end{equation}
\end{example}

Before proceeding, we recall two key properties of the map $\varphi$ above that will be used in what follows.
First, for a vertex $T_v$ of $T$, write $S_v$ for the outer face of $S$ whose outer edges are the image of $T_v$
(alternatively,
$T_v \to S_v \to S$ is the ``tall-outer factorization''
of $T_v \to T \to S$).
Then, $\varphi$ is a \emph{label map},
meaning that if $T_v$ has a $\bullet$-label 
(resp. $\circ$-label) then so do all vertices in $S_v$.
Moreover, $\varphi$ is \emph{$\circ$-inert}, 
meaning that if $T_v$ has a $\circ$-label then $S_v$ is a corolla
(on a terminological note, $\circ$-labeled nodes are called inert precisely because we work only with $\circ$-inert maps).

The decomposition \eqref{PUSHOPPR EQ}
will be the key to verifying the 
characteristic edge conditions in Definition \ref{CHAREDGE DEF}.
To do so, we will first find it useful to discuss a number of special types of dendrices and principal subpresheaves of $N \mathcal{P}$, suggested by \eqref{PUSHOPPR EQ}.
Recall that, by the strict Segal condition characterization of nerves \cite[Cor. 2.6]{CM13a},
a dendrex $p \colon \Omega[U] \to N \mathcal{P}$
is uniquely specified by the tree $U \in \Omega$ together with
a coloring
$\boldsymbol{E}(U) \to \mathfrak{C}$
and a choice of operations
$\{p_v \in \mathcal{P}(\vect{U}_v)\}_{v \in \boldsymbol{V}(U)}$.
Moreover,
for $\vect{D} = (D,\mathfrak{d})$ a $\mathfrak{C}$-corolla,
we will throughout make use of the decomposition 
\begin{equation}\label{DECONCOR EQ}
	\left(
	\Omega[C] \amalg_{\partial \Omega[C]} N \O
	\right)_{\mathfrak{d}} (D)
\simeq
	\Sigma_{\mathfrak{C}}[\vect{C}](\vect{D}) \amalg \O(\vect{D})
\end{equation}
where we recall that the left side 
is an instance of the $X_{\mathfrak{c}}(U)$ notation
in Notation \ref{XUDECALT NOT}.

\begin{definition}\label{DENDTYPES DEF}
	A dendrex $p\colon \Omega[U] \to N \mathcal{P}$ 
	is called:
	\begin{itemize}
		\item \emph{elementary} if for each vertex $U_v \hookrightarrow U$,
                        one has $ \partial_{U_v} p \in \mathcal{O} \amalg \Sigma_{\mathfrak{C}}[\vect{C}]$, 
                        cf. \eqref{DECONCOR EQ};
                        % $ \O \amalg_{\mathfrak{C}}B$;
		\item \emph{alternating} if $U \in \Omega^a$ is an alternating tree
                        and for each active (resp. inert) vertex 
		$U_v \hookrightarrow U$ one has
		$\partial_{U_v} p \in \O$
		(resp. $ \partial_{U_v} p \in \Sigma_{\mathfrak{C}}[\vect{C}]$);
		\item \emph{canonical} if it is non-degenerate and has a degeneracy which is alternating.
	\end{itemize}
\end{definition}

\begin{definition}
	Let $\langle p \rangle \subseteq N \mathcal{P}$ be 
	a principal subpresheaf. 
	We say $\langle p \rangle$ is:
	\begin{itemize}
		\item \emph{unital} if there is a representative
		$p\colon \Omega[U] \to N \mathcal{P}$ with $U=\eta$ the stick tree;
		\item \emph{reduced} if there is a representative
		$p\colon \Omega[U] \to N \mathcal{P}$ with $U \in \Sigma$, i.e. with $U$ a corolla;
		\item \emph{elementary} 
		if there is an elementary representative
		$p\colon \Omega[U] \to N \mathcal{P}$;
		\item \emph{canonical} 
		if there is a canonical (equivalently, alternating) representative
		$p\colon \Omega[U] \to N \mathcal{P}$.
	\end{itemize}
\end{definition}

\begin{remark}
	A dendrex is elementary iff its degeneracies are elementary,
	so the definition of elementary subpresheaf does not depend on the choice of representative.
\end{remark}

\begin{notation}\label{REDUCT NOT}
	Recalling that any tree $U$
	has an associated corolla $\mathsf{lr}(U)$ (counting leaves of $U$) together with a map
	$\mathsf{lr}(U) \to U$, 
	we abbreviate
	$\partial_r p = \partial_{\mathsf{lr}(U)}p$
	and call
	$\langle \partial_r p \rangle$
	the \emph{reduction} of  
	$\langle p \rangle$.
\end{notation}

\begin{remark}\label{PUSHOPPRRST REM}
        A reduced principal subpresheaf $\langle r \rangle \subseteq N \mathcal P$
        is simply the data of a single operation in $\mathcal P$,
        and so the coproduct decomposition of 
        \eqref{PUSHOPPR EQ} 
	implies that
        for each such $\langle r \rangle$,
	there exists an alternating dendrex $a \colon \Omega[T] \to N \mathcal P$, 
	\emph{unique up to isomorphism}, 
	such that
	$\langle \partial_r a \rangle = \langle r \rangle$.
        % \eqref{PUSHOPPR EQ} 
	% implies
	% that, for each reduced principal subpresheaf
	% $\langle r \rangle \subseteq N \mathcal{P}$,
	% there exists an alternating dendrex $a$ of $N \mathcal P$, 
	% \emph{unique up to isomorphism}, 
	% such that
	% $\langle \partial_r a \rangle = \langle r \rangle$.
	% 
	Moreover, $\langle a \rangle$ is thus the only canonical subpresheaf with reduction 
	$\langle r \rangle$,
	and we write
	$\langle r \rangle_{\chi} = \langle a \rangle$
	to denote this.
	
	Lastly, note that one thus has that 
	$\langle p \rangle$ is canonical iff
	$\langle p \rangle = \langle \partial_r p \rangle_{\chi}$.
\end{remark}

\begin{remark}\label{UNITALCASE REM}
	A reduced subpresheaf $\langle r \rangle$
	is unital iff 
	$\langle r \rangle_{\chi}$ is unital, in which case
	$\langle r \rangle = \langle r \rangle_{\chi}$.
\end{remark}

\begin{remark}\label{ELEMLABEL REM}
	If $e \colon \Omega[U^e] \to N \mathcal{P}$
	is an elementary dendrex, 
	the tree $U^e$ can be naturally regarded as an
	$\{\O,\vect{C}\}$-labeled tree by labeling each vertex $U^e_v \into U^e$ according to whether
	$\partial_{U^e_v} e \in \O $ or
	$ \partial_{U^e_v} e \in \Sigma_{\mathfrak{C}}[\vect{C}]$.
	
	By the alternating tree analogue of
%	({\color{blue} the alternating tree analogue of})
	%\cite[Prop. 5.49]{BP_geo},
        \cite[Prop. 5.57]{BP21},
        there is hence a unique alternating tree $U^a$ together with a tall planar $\vect{C}$-inert label map $U^a \to U^e$,
	and it then follows that
	$\partial_{U^a} e$ is an alternating dendrex so that
	$\langle \partial_{U^a} e \rangle = \langle \partial_r e \rangle_{\chi}$.
	In particular, this shows that
	$\langle \partial_r e \rangle_{\chi} \subseteq 
	\langle e \rangle$.
\end{remark}

\begin{example}
	Should $U^e$ in the previous remark be the tree $S$ in 
	Example \ref{ALTTREES_EX},
        % \eqref{ALTTREES EQ}
        then $U^a \to U^e$ is the map
	$\varphi\colon T \to S$ depicted therein. 
\end{example}

\begin{definition}\label{XIEDGES DEF}
	Let $e \colon \Omega[U^e] \to N \mathcal{P}$ be a non-degenerate elementary dendrex. 
	We write
	$\Xi^e \subseteq \boldsymbol{E}^{\mathsf{i}}(U^e)$
	for the subset of inner edges that are adjacent to at least one $\vect{C}$-labeled vertex.
\end{definition}

\begin{remark}\label{XIEREDEF REM}
	Let $e \colon \Omega[U^e] \to N \mathcal{P}$ be a non-degenerate elementary dendrex, 
	$U^a \to U^e$ be as in Remark \ref{ELEMLABEL REM}, 
	and write $a = \partial_{U^a} e$.
	Since $U^a$ is alternating, all of its inner edges are adjacent to a 
	$\vect{C}$-labeled vertex. 
	Therefore, the fact that $U^a \to U^e$ is a tall $\vect{C}$-inert label map
	implies that $\Xi^e$ consists of those inner edges which are in the image of $U^a$.
\end{remark}

\begin{proposition}\label{CANIFFXIE PROP}
	Let $e \colon \Omega[U^e] \to N \mathcal{P}$ be a non-degenerate elementary dendrex.
	Then $\langle e \rangle$
	is canonical iff $\Xi^e = \boldsymbol{E}^{\mathsf{i}}(U)$.
\end{proposition}

\begin{proof}
	We use the notation in Remark \ref{XIEREDEF REM}.
	Since $\langle a\rangle = \langle \partial_r e \rangle_{\chi}$, 
	$\langle e \rangle$
	is canonical 
	iff
	$\langle a\rangle = \langle e \rangle$, i.e.
	iff
	$U^a \to U^e$ is a degeneracy. 
	But since a map of trees is a degeneracy iff it is tall and surjective,
	it follows that $U^a \to U^e$ is a degeneracy
	iff its image includes all inner edges, i.e. iff $\Xi^e = \boldsymbol{E}^{\mathsf{i}}(U)$.
\end{proof}

\begin{remark}\label{WHENLRINN REM}
	If $U\neq \eta$ is not the stick tree,
	then $\mathsf{lr}(U) \simeq U - \boldsymbol{E}^{\mathsf{i}}(U)$
	is the inner face removing all inner edges.
\end{remark}

\begin{lemma}\label{ELEMEXIST LEM}
	For any principal subpresheaf $\langle p \rangle \subseteq N \mathcal{P}$
	there exists an elementary subpresheaf
	$\langle e \rangle \subseteq N \mathcal P$, 
	non-degenerate representative 
	$e \colon \Omega[U^e] \to N \mathcal{P}$,
	and a subset $E \subseteq \Xi^{e}$
	such that
	$\partial_{U^e-E} (e)$ is non-degenerate and 
	$\langle p \rangle = \langle \partial_{U^e-E} (e) \rangle$.
	In particular,  
	$\langle p \rangle \subseteq \langle e \rangle$.
\end{lemma}

\begin{proof}
	Let $p\colon \Omega[U] \to N \mathcal{P}$
	be a non-degenerate representative. We first build $e$.
	
	For each vertex $U_v \hookrightarrow U$, write
	$p_v = \partial_{U_v} p$
	and,
	noting that $\langle p_v \rangle$ is reduced,
	we further write 
	$e_v \colon \Omega[U^e_v] \to N \mathcal{P}$
	for some canonical representative of 
	$\langle p_v \rangle_{\chi}$
	(cf. Remark \ref{PUSHOPPRRST REM}).

	The identity
	$\langle p_v \rangle = \langle \partial_r e_v \rangle$
	implies that
	$U_v \simeq \mathsf{lr}(U^e_v)$,
	so that by choosing tall maps $U_v \to U^e_v$ 
	one obtains a $U$-substitution datum
	\cite[Def. 3.43]{BP21}
	which by 
	\cite[Prop. 3.46]{BP21}
	can be assembled into a tree $U^e$
	together with a tall map
	$U \to U^e$ such that for every vertex $U_v$
	the ``tall map followed by outer face'' factorization of
	the composite
	$U_v \to U \to U^e$
	is given by
	$U_v \to U^e_v \to U^e$.
	Since each vertex of $U^e$ is in exactly one of the outer trees $U^e_v$,
	we define $e \colon \Omega[U^e] \to N \mathcal{P}$
	as the unique dendrex such that
	$\partial_{U^e_v} e = e_v$.
	Note that since the $e_v$ are non-degenerate 
	%and canonical, $e$ is non-degenerate and elementary. %
	then so is $e$.

	Since $\partial_{U}e = p$ was chosen to be non-degenerate,
	%is non-degenerate by assumption,
	it remains to show that
	$U \to U^e$ identifies $U \simeq U^e - E$
	for some $E \subseteq \Xi^{e}$.
	%But since $p$ is non-degenerate,
	Since $\langle p_v \rangle, \langle p_v \rangle_{\chi}$ are non-unital
	(by assumption on $p$ and Remark \ref{UNITALCASE REM}),
	the $U^e_v$ are not stick trees,
	and Remark \ref{WHENLRINN REM} implies
	$U \simeq U^e-E$ for 
	$E = \amalg_{v \in \boldsymbol{V}(U)} \boldsymbol{E}^{\mathsf{i}}(U^e_v)$.
	That
	$E \subseteq \Xi^{e}$, i.e. that any edge in $E$ is adjacent to a $\vect{C}$-labeled vertex,
	follows from Proposition \ref{CANIFFXIE PROP}
	applied to the $\langle e_v \rangle$.
\end{proof}

\begin{lemma}\label{FORSAKEN LEM}
	Suppose $e \colon \Omega[U^e] \to N \mathcal{P}$ is elementary
	and $\langle \partial_r e\rangle$ is not unital.
	Then there exists an inner face map
	$U^c \to U^e$ such that $\partial_{U^c} e$ is canonical.
\end{lemma}

\begin{proof}
	Let $U^a \to U^e$ be as in Remark \ref{ELEMLABEL REM}.
	
	Writing
	$a = \partial_{U^a} e$,
	$r = \partial_r e$, 
	and letting
	$c \colon \Omega[U^c] \to N \mathcal{P}$
	be a canonical representative of
	$\langle a \rangle = \langle r \rangle_{\chi}$,
	%and hence by Remark \ref{PUSHOPPRRST REM}
	one has that $a$ is a degeneracy of $c$,
	i.e. there is a degeneracy map $U^a \to U^c$
	such that $\partial_{U^a} c = a$.
	Since $\langle r \rangle = \langle \partial_r e\rangle$ is not unital,
	%by assumption,
	neither is $\langle r \rangle_{\chi}$ (cf. Remark \ref{UNITALCASE REM}), 
	so that $U^c$ can not be the stick tree $\eta$,
	and thus $U^a \to U^c$ has a section which is an inner face 
	(this follows from \cite[Cor. 5.38]{Per18} since no edge of $U^c$ is both a root and a leaf).
	But then the composite 
	$U^c \to U^a \to U$
	must be a face (or else $c = \partial_{U^c} e$ would be degenerate) and is tall, and is hence an inner face.
\end{proof}

\begin{lemma}\label{UNIQINAN LEM}
	Let 
	$c \colon \Omega[U^c] \to N \mathcal{P}$ 
	be a non-unital canonical dendrex,
	$e \colon \Omega[U^{e}] \to N \mathcal{P}$
	an elementary dendrex,
	and $\mathsf{lr}(U^c) \to U^{e}$ a tall map.

	Then, 
	if the solid diagram below commutes, there exists a tall dashed map making the diagram commute.
	%Moreover, if $C \to U^{e}$ is tall the dashed map is also tall.
	%if $e$ is canonical the dashed map is an isomorphism.
	\begin{equation}\label{UNIQINAN EQ}
	\begin{tikzcd}
	\Omega\left[\mathsf{lr}(U^c)\right] \ar{r} \ar{rd}&
	\Omega[U^c] \ar{r}{c} \ar[dashed]{d} &
	N \mathcal{P}
	\\
	&
	\Omega[U^{e}] \ar{ru}[swap]{e} 
	\end{tikzcd}
	\end{equation}
\end{lemma}

\begin{remark}
	The requirement that $\langle c \rangle$ is non-unital is essential, as there may exist non-unital $\langle e \rangle \subseteq N \O$
	such that $\langle \partial_r e\rangle = \langle \partial_r c\rangle$ is unital, in which case no dashed arrow as in \eqref{UNIQINAN EQ} can exist.
\end{remark}

\begin{proof}
	Since commutativity of \eqref{UNIQINAN EQ} implies
	$\langle \partial_r e \rangle =
	\langle \partial_r c \rangle$,
	which is not unital by assumption, 
	by Lemma \ref{FORSAKEN LEM} there is an inner face $U^{\bar{c}} \to U^{e}$
	such that $\bar{c} = \partial_{U^{\bar{c}}} (e)$ is canonical.
	
	By definition of canonical dendrex, there are
	degeneracies
	$U^a \to U^c$,
	$U^{\bar{a}} \to U^{\bar{c}}$
	with $U^a,U^{\bar{a}}$
	alternating trees
	and such that the composites 
	$\Omega[U^a] \to \Omega[U^c] \xrightarrow{c} N \mathcal{P}$,
	$\Omega[U^{\bar{a}}] \to \Omega[U^{\bar{c}}] \xrightarrow{\bar{c}} N \mathcal{P}$  
	are alternating dendrices. And since $\mathsf{lr}$ sends tall maps to isomorphisms, we can form the diagram
	\[
	\begin{tikzcd}
	\Omega[\mathsf{lr}(U^c)] \ar{r}{\simeq} \ar{rd}[swap]{\simeq} &
	\Omega\left[\mathsf{lr}(U^a)\right] \ar[dashed]{d}[swap]{\simeq} \ar{r} &
	\Omega[U^a] \ar[dashed]{d}[swap]{\simeq} \ar{r} &
	\Omega[U^c] \ar{r}{c} \ar[dashed]{d}[swap]{\simeq} &
	N \mathcal{P}
	\\
	&
	\Omega\left[\mathsf{lr}(U^{\bar{a}})\right] \ar{r} &
	\Omega[U^{\bar{a}}] \ar{r} &
	\Omega[U^{\bar{c}}] \ar{ru}[swap]{\bar c} \ar{r} &
	\Omega[U^{e}] \ar{u}[swap]{e}. &
	\end{tikzcd}
	\]
	We will argue that all dashed vertical isomorphisms exist.
	That the first vertical isomorphism exists is trivial.
	The existence of the second vertical isomorphism follows from
	\eqref{PUSHOPPR EQ} which implies that, for $a,\bar{a}$ alternating dendrices, all isomorphisms 
	$\partial_r a \simeq \partial_r \bar{a}$
	are induced from an isomorphism $a \simeq \bar{a}$.
	Lastly, the existence of the third isomorphism follows 
	from the fact that the factorization of degenerate dendrices through non-degenerate dendrices is unique up to (unique) isomorphism \cite[Prop. 5.62]{Per18}.
	%
	%The moreover claim is clear. 
\end{proof}

\begin{lemma}\label{UNIQINAN2 LEM}
	Let 
	$e \colon \Omega[U^e] \to N \mathcal{P}$ 
	be a non-degenerate elementary dendrex,
	$\bar{e} \colon \Omega[U^{\bar{e}}] \to N \mathcal{P}$
	be an elementary dendrex,
	and 
	$U^e-E \to U^{\bar{e}}$ be a map with $E \subseteq \Xi^e$.
	
	Then, 
	if the solid diagram below commutes, there exists a dashed map making the diagram commute.
	\begin{equation}\label{UNIQINAN2 EQ}
	\begin{tikzcd}
	\Omega[U^e-E] \ar{r} \ar{rd}&
	\Omega[U^e] \ar{r}{e} \ar[dashed]{d} &
	N \mathcal{P}
	\\
	&
	\Omega[U^{\bar{e}}] \ar{ru}[swap]{\bar{e}} 
	\end{tikzcd}
	\end{equation}
\end{lemma}

\begin{proof}
	We abbreviate $U' = U^e -E$. 
	Note first that, 
	by applying the
	``tall map followed by outer face'' factorization to
	$U' \to U^{\bar{e}}$
	to obtain
	$U' \to \tilde{U} \to U^{\bar{e}}$,
	the dendrex $\partial_{\tilde{U}} \bar{e}$
	is still elementary (being an outer face of an elementary dendrex),
	so we reduce to the case where $U' \to U^{\bar{e}}$ is a tall map.
	
	For each vertex $U'_v \hookrightarrow U'$ we apply the 
	``inner face followed by outer face factorization''
	to the composites
	$U'_{v} \hookrightarrow U' \to U^{e}$,
	$U'_{v} \hookrightarrow U' \to U^{\bar{e}}$
	to get
	$U'_{v} \to U_{v}^e \hookrightarrow U^e$,
	$U'_{v} \to U_{v}^{\bar{e}} \hookrightarrow U^{\bar{e}}$
	and, further writing
	$e_v = \partial_{U^e_v} e$,
	$\bar{e}_v = \partial_{U^{\bar{e}}_v} \bar{e}$, 
	we obtain solid diagrams
	\begin{equation}\label{UNIQINAN2A EQ}
	\begin{tikzcd}
	\Omega[U'_v] \ar{r} \ar{rd}&
	\Omega[U^e_v] \ar{r}{e_v} \ar[dashed]{d} &
	N \mathcal{P}
	\\
	&
	\Omega[U^{\bar{e}}_v] \ar{ru}[swap]{\bar{e}_v} 
	\end{tikzcd}
	\end{equation}
	We now claim that $e_v$ is canonical. Indeed, $e_v$ is non-degenerate elementary since it is an outer face of $e$, which is also non-degenerate elementary. 
	And, since all inner edges of $U^e_v$ are in 
	$E \subseteq \Xi^e$, they are all adjacent to $\vect{C}$-labeled vertices, 
	so $e_v$
	is indeed canonical by
	Proposition \ref{CANIFFXIE PROP}.

	Since $\langle e_v \rangle$ is non-unital
	(or $U'_v \to U_v^e$ would be a degeneracy),
	by Lemma \ref{UNIQINAN LEM} there is a tall dashed arrow in \eqref{UNIQINAN2A EQ} for each $v \in \boldsymbol{V}(U')$,
	i.e. a $U^e$-substitution datum. Thus
	by
	\cite[Prop. 3.46]{BP21}
	we obtain the desired dashed arrow in \eqref{UNIQINAN2 EQ}.
\end{proof}

\begin{corollary}\label{MINELEMSH COR}
	If $e \colon \Omega[U^e] \to N \mathcal{P}$ is a non-degenerate elementary dendrex
	and $E \subseteq \Xi^e$,
	then 
	$\langle e\rangle$ is the smallest elementary subpresheaf
	containing $\langle \partial_{U^e - E} (e)\rangle$,
	i.e. if 
	$\langle \partial_{U^e - E} (e)\rangle
	\subseteq \langle \bar{e} \rangle$
	with $\langle \bar{e} \rangle$
	elementary then 
	$\langle e\rangle
	\subseteq \langle \bar{e} \rangle$.
\end{corollary}

\begin{proof}
	$\langle \partial_{U^e - E} (e)\rangle
	\subseteq \langle \bar{e} \rangle$
	yields the diagram \eqref{UNIQINAN2 EQ}
	and the dashed arrow therein shows
	$\langle e\rangle
	\subseteq \langle \bar{e} \rangle$.
\end{proof}

\begin{corollary}\label{CANCHAR COR}
	An elementary subpresheaf $\langle e \rangle$
	is canonical iff
	it is the smallest elementary subpresheaf containing
	$\langle \partial_r e \rangle$.
\end{corollary}

\begin{proof}
	As noted in Remark \ref{PUSHOPPRRST REM}, 
	$\langle e \rangle$ is canonical iff
	$\langle e \rangle = \langle \partial_r e \rangle_{\chi}$.
	If $\langle \partial_r e \rangle$ is unital, then 
	$\langle \partial_r e \rangle = \langle \partial_r e \rangle_{\chi}$,
	which is elementary (cf. Remark \ref{UNITALCASE REM}), so the claim is clear.
	Otherwise, letting 
	$c \colon \Omega[U^c] \to N \mathcal{P}$
	be a canonical representative of
	$\langle \partial_r e \rangle_{\chi} $,
	one has
	$\Xi^c = \boldsymbol{E}^{\mathsf{i}}(U^c)$
	and
	$\langle \partial_r e \rangle = 
	\langle \partial_{U^c-\Xi^c} c \rangle$
	so, by Corollary \ref{MINELEMSH COR},
	$\langle \partial_r e \rangle_{\chi} $
	is the smallest elementary containing $\langle \partial_r e \rangle$.
\end{proof}

\begin{corollary}\label{ISODIFCL COR}
	Suppose $N \mathcal{P}$ is $\Sigma$-free,
	and let $e \colon \Omega[U^e] \to N \mathcal{P}$
	be an elementary non-degenerate dendrex and
	$E,E' \subseteq \Xi^e$.
	If 
	$\langle \partial_{U^e-E} (e) \rangle
	=
	\langle \partial_{U^e-E'} (e) \rangle$,
	then in fact $E = E'$.
\end{corollary}

\begin{proof}
	If 
	$\langle \partial_{U^e-E} (e) \rangle
	=
	\langle \partial_{U^e-E'} (e) \rangle$
	then one can find a solid diagram as below
	\[
	\begin{tikzcd}
	\Omega[U^e-E] \ar{d}[swap]{\simeq} \ar{r} &
	\Omega[U^e] \ar{r}{e} \ar[dashed]{d} &
	N \mathcal{P}
	\\
	\Omega[U^e-E'] \ar{r} &
	\Omega[U^e] \ar{ru}[swap]{e} &
	\end{tikzcd}
	\]
	and thus by Lemma \ref{UNIQINAN2 LEM} one can also find the vertical dashed arrow.
	But since $N \mathcal{P}$ is $\Sigma$-free
	by assumption,
	the dashed arrow must be the identity,
	so that $U^e-E=U^e-E'$ and hence $E=E'$.
\end{proof}

\begin{proof}[Proof of Lemma \ref{KEYPR LEM}]
	We will verify the characteristic edge conditions in Definition \ref{CHAREDGE DEF}.
	Note that, 
	as $\partial \Omega[T] \to \Omega[T]$ is a generating cofibration in $\sOp^G$,
	$\O \to \P$ is a cofibration with $\Sigma$-cofibrant source and hence, by Proposition \ref{COPOFSIGCOF PROP},
	$\P$ is $\Sigma$-cofibrant.
	One thus has that 
	$N \mathcal{P}$ is $\Sigma$-free
	(cf. \cite[Lemma 5.9]{BM03})
	so that 
	the $\Sigma$-freeness conditions in 
	both Definition \ref{CHAREDGE DEF}
	and Corollary \ref{ISODIFCL COR}
	are satisfied.

	We set $X = \Omega[C] \amalg_{\partial \Omega[C]} N \mathcal{O}$, 
	with the $G$-poset of principal subpresheaves formed by the 
	elementary subpresheaves 
	$\langle e \rangle$
	under inclusion, and the characteristic inner edge sets
	$\Xi^{\langle e \rangle} = \left\{\Xi^{e}\right\}$ given as in Definition \ref{XIEDGES DEF}.
	Note that by Lemma \ref{ELEMEXIST LEM}
	every dendrex of $N \mathcal{P}$ is in some 
	$\langle e \rangle$, so that the map
	\eqref{CHAREDGE EQ} is indeed
	$\Omega[C] \amalg_{\partial \Omega[C]} N \mathcal{O}
	\to N \mathcal{P}$.
%	and moreover $\P$ being $\Sigma$-cofibrant implies that $N \mathcal P$ is $\Sigma$-free by Proposition \ref{SGS_COF_PROP}.
	
	Let $e\colon \Omega[U^e] \to N \mathcal{P}$
	be a non-degenerate elementary dendrex. We note the following: 
	\begin{itemize}
		\item[(a)] if $\Xi^e = \emptyset$ then 
		either all vertices of $U^e$ are $\O$-labeled, i.e. $e \in N \mathcal{O}$, 
		or $U^e$ is a $\vect{C}$-labeled corolla, 
		i.e. $e \in \Sigma_{\mathfrak{C}}[\vect{C}]$.
		In other words, $\Xi^e = \emptyset$ iff 
		$\langle e \rangle \subseteq X = N \mathcal{O} \amalg_{\partial \Omega[C]} \Omega[C]$.
		\item[(b)] any outer face of $e$ is again elementary,
		as is any inner face $\partial_{U^e-f} (e)$ such that $f \not \in \Xi^e$
		(since then both vertices adjacent to $f$ are $\O$-labeled).
		Therefore, by Corollary \ref{MINELEMSH COR},
		we see that a face of $e$ is \emph{not} in
		some elementary $\langle \bar{e} \rangle \subsetneq \langle e \rangle$
		iff it is of the form
		$\partial_{U^e - E} (e)$
		for some $E \subseteq \Xi^e$.
	\end{itemize}

	We now check the characteristic edge conditions. (Ch0.2) is clear.
	
	For (Ch1), by (b) any proper outer face of $e$ is in $X_{<\langle e\rangle}$, so we need only consider the case of
	$V=U^e$ with $\Xi^e=\emptyset$, in which case
	$\langle e \rangle \subseteq X \subseteq X_{<\langle e\rangle}$ by (a).
	
	For (Ch2),(Ch3), by (a) and the first half of (b) one needs only consider the case of
	$V \simeq U^e - E$ where $E \subseteq \Xi^e$ and $\Xi^e \neq \emptyset$.
	But then
	$V - \Xi^e_V \simeq U^e- \Xi^e$
	and, by the second half of (b),
	one has
	$\langle \partial_{U^e-\Xi^e}(e) \rangle \subseteq  X_{<\langle e\rangle}$
	iff
	$\langle e \rangle \subseteq  X_{<\langle e\rangle}$
	iff
	$\langle \partial_{U^e-E}(e) \rangle \subseteq  X_{<\langle e\rangle}$,
	so (Ch2),(Ch3) follow.

	Lastly, we address (Ch0.1). By (a) we need only consider the case of $\Xi^e \neq \emptyset$,
	so that by (b)
	the complement of
	the preimage
	$e^{-1}(X_{<\langle e\rangle})$
	consists of the faces isomorphic to
	$U^e-E$ for $E \subseteq \Xi^e$.
	Injectivity of $e$ within each isomorphism class of the faces away from % in 
	$e^{-1}(X_{<\langle e\rangle})$
	follows from $N \mathcal{P}$ being $\Sigma$-free,
	while injectivity across distinct isomorphism classes of faces is
	Corollary \ref{ISODIFCL COR}.
\end{proof}

\begin{remark}\label{CH01SUCKS REM}
	Condition (Ch0.1) is, by some margin, the subtlest condition in the previous proof, and the main reason for the chosen formulations of 
	Lemmas \ref{UNIQINAN LEM}, \ref{UNIQINAN2 LEM}.
	In particular, we note that injectivity of 
	$e \colon \Omega[U^e] \to N \mathcal{P}$ will in general fail away from 
	$e^{-1}(X_{< \langle e \rangle})$.
	For example, two edges/vertices of $U^e$
	may be assigned the same color/operation, and similarly for larger outer faces. In fact, injectivity may even fail on inner faces
	$T^e-E$ where $E \not \subseteq \Xi^e$.
\end{remark}

\begin{remark}\label{CHAREDCOMP REM}
	The proof of the characteristic edge lemma
	\cite[Lemma 3.4]{BP20}
	gives a filtration of 
	$\Omega[C] \amalg_{\partial \Omega[C]} N \mathcal{O}
	\to N \mathcal{P}$ by anodyne maps
	which, for $G=\**$,
	matches the filtration in \cite[\S 3]{CM13b}.

	As such, our arguments largely adapt \cite{CM13b},
	though we also patch an apparent gap in the proof therein.
	Namely, we believe that the appeal to  
	\cite[Lemma 3.4]{CM13b} in the proof of
	\cite[Lemma 3.5]{CM13b} is incorrect
	(namely, in the terminology
	of Definition \ref{DENDTYPES DEF},
	the assumption in \cite[Lemma 3.4]{CM13b}
	only holds for \emph{elementary} dendrices).
	To fix this, we replace the role of 
	\cite[Lemma 3.4]{CM13b} with 
	\cite[Lemma 3.4]{BP20}.
	From this point of view, \cite[Lemma 3.5]{CM13b}
	amounts to verifying (Ch1),(Ch2),(Ch3),
	though the tricky condition (Ch0.1)
	appears left unaddressed.
\end{remark}

\section{Indexing system analogue results}\label{INDSYS SEC}

As in \cite[\S 6]{BP20} and \cite[\S 9]{Per18}, 
we dedicate our final main section to 
discussing the variants of our main results 
for the \textit{indexing systems} of Blumberg-Hill \cite{BH15}
or, more precisely,
the slightly more general \textit{weak indexing systems}
introduced by the authors \cite[\S 9]{Per18}, \cite[\S 4.4]{BP21} (and independently identified by Gutierrez-White \cite{GW18}).

Weak indexing systems can regarded as
the operadic generalization of the notion of a family of subgroups 
(cf. Remark \ref{WIS_FAMILY_REM}).
Recall that families $\F$ of a subgroup $G$ are in
bijection with \textit{sieves} 
$\mathsf{O}_\F \subseteq \mathsf{O}_G$ of the orbit category of $G$.

\begin{definition}
      A \textit{sieve} of a category $\mathcal C$ is a full subcategory $\mathcal S \subseteq \mathcal C$ such that,
      for all arrows $A \to B$ in $\mathcal C$, 
      $B \in \mathcal S$ implies that $A$ and $A \to B$ are also in $\mathcal S$.
\end{definition}

\begin{definition}
      \label{WIS_DEF}
      A subcategory $\Omega_\F \subseteq \Omega_G$ is called a \textit{weak indexing system} if:
      \begin{enumerate}
      \item $\Omega_\F$ is a sieve in $\Omega_G$ and;
      \item for $T \in \Omega_G$, 
      one has $T \in \Omega_\F$ iff $T_v \in \Omega_\F$ for all $v \in \boldsymbol{V}_G(T)$.
      \end{enumerate}
\end{definition}

\begin{remark}\label{WIS_FAMILY_REM}
	The $\F$ in $\Omega_\F$ refers to an equivalent presentation of a weak indexing system.
	Fixing $\Omega_\F$, 
	let $\F_n$ for $n \geq 0$ denote the family of those
	graph subgroups $\Gamma \leq G \times \Sigma_n$
	(cf. \eqref{NORMMAP EQ}) such that,
	equipping the $n$-corolla $C_n$
	with the $H$-action given by the associated homomorphism
	$\phi_{\Gamma} \colon H \to \Sigma_n$,
	the induced $G$-corolla $G \cdot_H C_n$ is in $\Omega_\F$.
	The symbol $\F$ then denotes the collection of families $\set{\F_n}_{n \geq 0}$, whose data is equivalent to $\Omega_{\F}$.
	
	For later reference,
	replacing $C_n$ above with a tree $T\in \Omega$
	and $\Sigma_n$ with $\mathsf{Aut}(T)$,
	we write
	$\F_T$ for the family of graph subgroups
	$\Gamma \leq G \times \mathsf{Aut}(T)$ for which the similarly built $G\cdot_H T$ is in $\Omega_{\F}$.
	
	We note that the fact that each $\F_n$ is a family is consequence of $\Omega_\F$ being a sieve. More precisely, this follows from the sieve condition for the quotient maps in $\Omega_G$. 
	However, the $\F_n$ must satisfy additional operadic closure properties
	(resulting from the sieve condition for face maps in $\Omega_G$),
	which are made explicit in 
	\cite[Def. 3.22]{BH15}, \cite[Def. 2.12]{Rub_nf}.
%
%      More details and alternative definitions can be found in \cite[\S 9]{Per18},\cite[\S 4.4]{BP21},\cite[\S 3.4]{BP_ACOP}.
\end{remark}

\begin{remark}\label{ETAF_REM}
	By vacuousness, Definition \ref{WIS_DEF}(ii) 
	implies that any weak indexing system $\Omega_\F$
	contains the equivariant stick trees
	$G/H \cdot \eta$.
	As such, by the sieve condition, 
	all equivariant linear trees
	$G/H \cdot [n]$ for $[n] \in \Delta$
	are likewise in $\Omega_\F$.
	As such, $\F_1$ always contains all subgroups
	$H \leq G \simeq G \times \Sigma_1$.
\end{remark}

For both of the categories appearing in Theorem \ref{QE THM},
each weak indexing system $\F$
gives rise to model structures (on the same underlying categories)
$\mathsf{dSet}^G_{\F}$, 
$\mathsf{sOp}^G_{\F}$
which, loosely speaking, are built by replacing the role
of $\Omega_G$ (or the families of graph subgroups $\F^{\Gamma}_n$)
with $\Omega_\F$ (the families $\F_n$) throughout.
With minor changes to our proofs, discussed in the remainder of the section, one then obtains an $\F$-variant of Theorem \ref{QE THM},
stated below as Theorem \ref{QEF THM}. 

First, we recall that $\F$-variant model structures on 
all the categories in the square \eqref{ADJSQ EQ}
have already been built,
with some of the adjunctions therein already known to be Quillen equivalences.

% ---------- PREVIOUS $\F$-MODEL STRUCTURES ----------
\begin{itemize}
	\item Replacing $\Omega_G$ with $\Omega_\F$
	in Definitions \ref{GINNERANN DEF},\ref{GINNERFIB DEF}
	one obtains the classes of 
	$\F$-normal monomorphisms, $\F$-inner anodyne maps,
	and $\F$-inner fibrations in $\mathsf{dSet}^G$.
	Using these alternative classes,
	the $\F$-model structure
	$\mathsf{dSet}^G_{\F}$
	\cite[Thm. 2.2]{Per18}
	is described as in Theorem \ref{DSETGMOD THM}.

\item As in the discussion before Theorem \ref{JB_THM},
	since $\Delta^{op}$ is a Reedy category,
	the identification $\sdSet^G = (\dSet^G)^{\Delta^{op}}$
	yields a \emph{$\F$-simplicial Reedy model structure} on $\sdSet^G$
	with weak equivalences are the 
	\textit{$\F$-dendroidal equivalences},
	i.e. maps $X\to Y$ for which the maps 
	$X_n \to Y_n$, $n\geq 0$ are weak equivalences in $\dSet^G_\F$.
	
	Similarly, as $\Omega^{op}\times G$ is generalized Reedy and the collection of families $\{\F_T\}_{T \in \Omega}$
	is Reedy admissible,
	the identification $\sdSet^G = (\sSet)^{\Omega^{op} \times G}$
	yields a \emph{$\F$-dendroidal Reedy model structure} on $\sdSet^G$
	with weak equivalences the 
	\textit{$\F$-simplcial equivalences},
	i.e. maps $X\to Y$ for which the maps 
	$X(\Omega[T]) \to Y(\Omega[T])$ for $T \in \Omega_\F$
	are Kan equivalences in $\sSet$.	
	
	The \emph{$\F$-joint/complete/Rezk model structure} on 
	$\mathsf{sdSet}^G$
	(cf. the discussion above \cite[Thm. 6.7]{BP20}), 
	denoted $\mathsf{sdSet}^G_{\F}$,
	then combines the $\F$-simplicial and $\F$-dendroidal Reedy model structures as in Theorem \ref{JB_THM}.
	
	\item The $\F$-\emph{normal model structure} on $\PreOp^G$,
	denoted $\PreOp^G_{normal,\F}$
	(cf. the discussion above \cite[Thm. 6.8]{BP20}),
	is obtained from the $\F$-joint model structure 
	$\mathsf{sdSet}^G_{\F}$
	as in Theorem \ref{PREOPMS THM}.

	\item The adjunctions
	$c_! \colon \mathsf{dSet}^G_{\F} 
	\rightleftarrows
	\mathsf{sdSet}^G_{\F} \colon c^{\**}$
	and
	$\gamma^{\**} \colon \mathsf{PreOp}^G_{normal,\F} 
	\rightleftarrows
	\mathsf{sdSet}^G_{\F} \colon \gamma_{\**}$
	are Quillen equivalences \cite[Thms. 6.7 and 6.8]{BP20}.
	
	\item The $\F$-model structure on $\sOp^G$, 
		denoted $\sOp^G_{\F}$,
		is an instance of 
		\cite[Thm. \ref{AC-THMA}]{BP_ACOP} with $\F_n$ therein as 	discussed in Remark \ref{WIS_FAMILY_REM}.
		This model structure is described as in 
		Theorem \ref{SOPG_THM} by restricting the 
		$\Gamma \leq G \times \Sigma_n^{op}$ to $\F_n$.
		Likewise, the generating sets are as in
		Remark \ref{GENCOF_SOPG_REM}
		with $C \in \Sigma_G$ restricted to 
		$C \in \Sigma_{\F} = \Sigma_G \cap \Omega_{\F}$.
\end{itemize}

Next, the notions of $\F$-Segal operad and 
$\F$-Dwyer Kan equivalence are obtained 
by restricting $T\in \Omega_G$, $C \in \Sigma_G$
in Definitions \ref{SEGCOLCHAR DEF},\ref{DKEQUIV DEF}
to $T\in \Omega_\F$, $C \in \Sigma_\F$.
As in Theorem \ref{FIBPREOP THM},
one has the following,
which is left unstated in \cite[\S 6]{BP20},
but follows from \cite[Thm. 6.9]{BP20}
and the $\F$-version of \cite[Cor. 5.51]{BP20}
(in fact, we need only the easier ``only if'' half of the latter,
which follows by the same proof, 
restricting $T$ therein to $\Omega_\F$).
\begin{remark}\label{FIBPREOPF REM}	
	A map between $\F$-Segal operads is a 
	$\F$-joint equivalence iff it is an $\F$-Dwyer-Kan equivalence.
\end{remark}

Adapting the work in \S \ref{TAMEDEFEX SEC},
and writing 
($\F$TC2),($\F$TC3),($\F$TA2),($\F$TA3)
for the maps in Definitions \ref{TAMEGENCOF DEF},\ref{TAMEGENANO DEF}
restricted to $C \in \Sigma_\F$, $T \in \Omega_\F$ we now get the following extension of Theorem \ref{TAMEMS_THM}.
Note that, by Remark \ref{ETAF_REM},
there is no need to alter (TC1),(TA1).

\begin{theorem}\label{TAMEMSF_THM}
	There is a left proper model structure on 
	$\mathsf{PreOp}^G$,
	called the \emph{$\F$-tame model structure}
	and denoted $\mathsf{PreOp}^G_{\F}$,
	such that:
	\begin{itemize}
		\item weak equivalences are the 
		$\F$-joint equivalences 
		(detected by inclusion into $\mathsf{sdSet}^G_{\F}$);
		\item the generating cofibrations are the maps (TC1),($\F$TC2),($\F$TC3);
		\item $X \in \mathsf{PreOp}^G_{\F}$ is fibrant iff
		$X \to \**$ has the right lifting property against 
		(TA1),($\F$TA2),($\F$TA3);
		\item a map $X \to Y$ between fibrant objects is a fibration iff
		it has the right lifting property against 
		(TA1),($\F$TA2),($\F$TA3).
	\end{itemize}
	Moreover, the identity adjunction
	$
	\mathsf{PreOp}^G_{tame,\F}
	\rightleftarrows
	\mathsf{PreOp}^G_{normal,\F} 
	$
	is a Quillen equivalence.
\end{theorem}

\begin{proof}
	All the arguments in the proof of Theorem \ref{TAMEMS_THM},
	as well as in the proofs of necessary 
	Lemmas \ref{TAMECOFCOF_LEM},\ref{TAMETRIVFIB LEM},\ref{SLIMOD LEM}
	carry through, 
	with the biggest differences being that 
	Lemma \ref{TAMECOFCOF_LEM} 
	uses the obvious $\F$-variant of Remark \ref{SQUAREEQUI REM}
	and that the verification of C5 in the proof of Theorem \ref{TAMEMS_THM} replaces the appeal to 
	Theorem \ref{FIBPREOP THM} with an appeal to 
	Remark \ref{FIBPREOPF REM}.
\end{proof}

% ---------- COMPARISON WITH OPERADS ----------
%\subsubsection*{Comparison with operads}

We now turn to the $\F$-variants of the Quillen equivalences established in
\S \ref{QE_SEC}. 

The $\Sigma$-cofibrations 
$\O \to \mathcal{P}$ in $\mathsf{sOp}^G$ (or $\mathsf{Op}^G$)
discussed after Proposition \ref{SSYMCOFCH PROP}
generalize to the notion of $\Sigma_\F$-cofibration,
demanding that any point 
$x \in \mathcal{P} \setminus \O$
at a profile $\vect{C}$
of arity $n$ must have isotropy in $\F_n$.
We note that this is connected to the characterization, 
given in the discussion after \cite[Def. 9.8]{Per18},
of the $\F$-normal monomorphisms $X\to Y$ in $\mathsf{dSet}^G$
as the maps such that points
$x \in Y\setminus X$ at a tree $T$ have isotropy in $\F_T$.
Indeed, the nerve formula \eqref{ALTNER EQ}
and Definition \ref{WIS_DEF}(ii)
imply that, for $\O \to \mathcal{P}$ a 
$\Sigma_{\F}$-cofibration between 
$\Sigma_{\F}$-cofibrant objects in $\mathsf{Op}^G$,
the nerve map
$N\O \to N\mathcal{P}$
is an $\F$-normal monomorphism between
$\F$-normal objects in $\mathsf{dSet}^G$.

The following is the extension of Theorem \ref{PREQUIEQUIV THM}.

\begin{theorem}\label{PREQUIEQUIVF THM}
	The following adjunction is a Quillen equivalence.
	\begin{equation}\label{PREQUIEQUIVF EQ}
	\tau \colon \mathsf{PreOp}^G_{\text{tame},\F}
	\rightleftarrows 
	\mathsf{sOp}^G_{\F} \colon N
	\end{equation}
\end{theorem}

\begin{proof}
	Most of the proof of Theorem \ref{PREQUIEQUIV THM}
	extends, with the role of  
	Theorem \ref{FIBPREOP THM} 
	replaced with
	Remark \ref{FIBPREOPF REM}.
	The trickiest part concern the appeals 
	to Corollary \ref{KEYEQUIV COR}.
	To obtain an $\F$-version of the latter, 
	we need an $\F$-version of Lemma \ref{UNITEQUIV LEM}.
	Analyzing its proof, 
	this comes down to having $\F$-variants of 
	Lemma \ref{FCOLIM_WE_LEM} and of Lemma \ref{KEYPR LEM}.
	For Lemma \ref{FCOLIM_WE_LEM}, its proof immediately generalizes.
	
	As for Lemma \ref{KEYPR LEM},
	\cite[Prop. \ref{OC-SIGMAG_COF PROP}]{BP_FCOP}
	ensures that the map $\O \to \P$ in \eqref{PUSHOUTPROP EQ}
	is a $\Sigma_\F$-cofibration betwen $\Sigma_\F$-cofibrant objects.
	Thus, by the discussion above,
	the map
	$\Omega[C]\amalg_{\partial \Omega[C]} N \O \to N\P$ in
	\eqref{ANODYNE MAP} has $\F$-normal target. 
	Hence, since this map is shown to be $G$-anodyne via the characteristic edge lemma \cite[Lemma 3.4]{BP20},
	whose proof follows by attaching $G$-inner horns determined by $N\P$,
	the fact that $N\P$ is $\F$-normal implies that 
	\eqref{ANODYNE MAP} must be built by attaching only 
	$\F$-inner horns, and thus be an $\F$-anodyne map
	(this is identical to the argument before  
	\cite[Thm. 6.7]{BP20}).
\end{proof}

Lastly, we have the following, which is the 
$\F$-variant of the main result, Theorem \ref{QE THM}.

\begin{theorem}\label{QEF THM}
	Let $\F$ be a weak indexing system.
	One then has a Quillen equivalence
	\begin{equation}
	\label{QEF_EQ}
	W_! \colon 
	\dSet^G_{\F} \rightleftarrows \sOp^G_{\F}
	\colon hcN.
	\end{equation}
\end{theorem}

\begin{proof}
	First, the fact that \eqref{QEF_EQ} is a Quillen adjunction
	follows as in the proof of Proposition \ref{W!_LEFTQ_PROP}
	where we note that, in the appeal to 
	Lemma \ref{OPTENSCOF_LEM},
	the fact that for $C \in \Sigma_\F$ the decomposition
	$G/H \cdot \Omega(C) \simeq \amalg_i \Omega(C_i)$
	satisfies $C_i \in \Sigma_\F$ 
	follows from the sieve condition in Definition \ref{WIS_DEF}(i).

	The remainder of the Quillen equivalence proof follows as written,
	using the same zigzag \eqref{BIGZIG EQ}
	(the $\F$ version of Lemma \ref{DIAGWE LEM}
	works as expected, using $\F$-simplicial equivalences)
	with the only notable point being that, for the map (d) therein,
	one needs to know that 
	$\Omega[T]$ for $T\in \Omega_\F$
	is $\F$-tame fibrant, which is clear from the proof of
	Lemma \ref{OMEGATTAME_LEM}.
\end{proof}

% ------------------------------ APPENDIX ----------------------------------------
\appendix

\section{The homotopy genuine equivariant operad}
\label{HGEO AP}

The goal of this appendix is to establish
Proposition \ref{HOOPID_PROP},
which compares two procedures of discretizing 
an equivariant operad $\O \in \mathsf{sOp}^G$,
and is the full equivariant generalization of
\cite[Prop. 4.8]{CM13b},
cf. Remark \ref{TWOHOMOP REM}.

Most of the work will be spent 
describing the ``genuine operadification'' functor
$\tau_G \colon \mathsf{dSet}_G \to \mathsf{Op}_G$
first mentioned in
\eqref{TAUFUNCTS EQ}.
In general, $\tau_G Z$ for some $Z \in \mathsf{dSet}_G$
is given by the formula 
\eqref{TAUGFORM EQ} below,
though this formula is cumbersome in practice
(and included mostly for completeness).
Instead, our focus will be on the special case of
$Z \in \mathsf{dSet}_G$
a \emph{genuine $G$-$\infty$-operad}
(Definition \ref{GENINFOP DEF}),
for which $\tau_G Z$ admits a simpler description as a 
\emph{homotopy operad}, 
which we denote
$\mathop{\mathrm{ho}}(Z)$ (Definition \ref{XTENDSIM DEF}).
This alternative $\mathop{\mathrm{ho}}(Z)$
construction mimics the similar description 
of $\tau X$ for $X\in \mathsf{dSet}$
an $\infty$-operad given in \cite[\S 6]{MW09},
so that the majority of this appendix is a fairly direct generalization of the work, 
although with one interesting nuance.
Namely, \cite[\S 6]{MW09} makes heavy use of tree diagrams,  
and associated faces/horns, 
whose equivariant generalization turns out to involve
\emph{orbital representations of $G$-trees}
rather than expanded representations
(cf. \eqref{GTREE_EQ}).
As such, we will need to recall the notions of
orbital faces (Def. \ref{ORBITALFACE_DEF}) and 
orbital horns (Eq. \eqref{ORBITALHORN_EQ})
introduced in 
\cite[\S2.2,\S2.3]{BP20},
which are distinct from any of the notions recalled in 
\S \ref{EDS_SEC}.

We first recall the notion of orbital face.

\begin{definition}
        \label{ORBITALFACE_DEF}
	A map $S \to T$ of $G$-trees
	that is injective on edges is called an
	\emph{orbital face map}.
\end{definition}

\begin{example}\label{ORBFACE EQ}
	Let $T$ be the $G$-tree in \eqref{GTREE_EQ},
	whose edge orbits we abbreviate as $Ga,Gb,Gc,Gd$.
	
	$T$ has orbital faces
	$T-Gd = T-\{d,\rho d\}$
	and 
	$T-Gc = T-\{c,\sigma c, \rho c, \rho \sigma c\}$,
	where we note that the expanded representation of
	$T-Gd$ has four tree components
	(in contrast to $T$ and $T-Gc$, whose expanded representations have only two components).
	The term ``orbital'' refers to the fact that orbital faces 
	correspond to (usual) faces of the underlying tree in the orbital representation.
\end{example}

As one would expect, an orbital outer face 
$S \to T$ is called \emph{inner} if it is given by removing inner edges (e.g. $T-Gc$ in Example \ref{ORBFACE EQ}).
Orbital faces allow for a general description of $\tau_G(Z)$ for any $Z \in \dSet_G$, cf. \eqref{TAUGFORM EQ} below.
This description makes use of the following, 
which is the key definition in \cite{BP_WCONS},
and is connected to the description of the $W$-construction,
Definition \ref{WU_DEF},
% \eqref{WU_EQ}
built therein.

\begin{definition}\label{DENDNECK DEF}
	The category $\mathsf{Nec}^t_G$ of $G$-dendroidal necklaces and tall maps has:
	\begin{enumerate}
		\item objects planar orbital inner faces 
		$\mathfrak{n} \colon J \to T$, 
		called \emph{necklaces};
		\item maps from 
		$\mathfrak{n} \colon J \to T$ to
		$\mathfrak{n}^{\**} \colon J^{\**} \to T^{\**}$
		given by tall maps $\varphi \colon T \to T^{\**}$
		such that $\varphi(J) \supseteq J^{\**}$.
	\end{enumerate}
\end{definition}

The category $\mathsf{Nec}^t_G$ above is part of a larger category $\mathsf{Nec}_G$ where maps need not be tall,
though defining the latter requires more care
(see \cite[Def. \ref{W-NECKREP_DEF} and Prop. \ref{W-MAPNECK PROP}]{BP_WCONS}).
Nonetheless, by restricting to 
$\mathsf{Nec}^t_G$
one has a \emph{leaf-root functor}
$\mathsf{lr} \colon \mathsf{Nec}^t_G \to \Sigma_G$
given by $(J \to T) \mapsto \mathsf{lr}(T)$
where (cf. \eqref{FREEOP_EQ}, \eqref{PUSHOPPR EQ}, Notation \ref{REDUCT NOT})
the $G$-corolla
$\mathsf{lr}(T)$ replaces each tree component of $T$
with the corolla with the same number of leaves
(e.g. in Example \eqref{ORBFACE EQ}
one has $\mathsf{lr}(T) = T-Gc$).
\cite[Rem. \ref{W-GTAUFUNEX REM}]{BP_WCONS}
can then be restated as follows.

\begin{remark}\label{TAUGFORM REM}
	Let $Z \in \mathsf{dSet}_G$. 
	Then, at a $G$-corolla $C$ one has the formula
\begin{equation}\label{TAUGFORM EQ}
	\tau_G Z(C) = 
	\colim_{(C \to \mathsf{lr}(J \to T))\in(C \downarrow \mathsf{Nec}_G^{t,op})}
	\prod_{v \in \boldsymbol{V}_G(J)} Z(T_v).
\end{equation}
\end{remark}

Noting that any necklace $(J\to T)$ 
receives a map 
$(T \to T) \to (J \to T)$
in $\mathsf{Nec}^t_G$,
one has that every element in the colimit
\eqref{TAUGFORM EQ} has a representative in 
$\prod_{v \in \boldsymbol{V}_G(T)} Z(T_v)$
for some $G$-tree $T\in \Omega_G$
(though general necklaces are needed to encode all the relations).

We now turn our attention to the special case of 
$G$-$\infty$-operads $Z \in \mathsf{dSet}_G$, 
defined as follows,
and for which all elements of
$\tau_G Z(C)$ in \eqref{TAUGFORM EQ}
will be represented by elements of $Z(C)$ itself.

\begin{definition}\label{GENINFOP DEF}
        $Z \in \mathsf{dSet}_G$
        is a \emph{genuine $G$-$\infty$-operad}
        if it has the right lifting property against all maps
        $\upsilon_{G,\**} 
        \left(\Lambda^E[T] \to \Omega[T]\right)$
        for $T \in \Omega_G$,
        $G$-subset
        $E \subseteq \boldsymbol{E}^{\mathsf{i}}(T)$,
        and $\upsilon_{G,\**}$
        as defined in \eqref{UPSILONADJ EQ}.
\end{definition}

\begin{remark}
	Genuine $G$-operads (Definition \ref{OPG_DEF})
	can also be defined via a \emph{strict}
	lifting property against the maps
	$\upsilon_{G,\**} \left(\Lambda^E[T] \to \Omega[T]\right)$,
	cf. \cite[Def. 3.35]{BP20}.
	Thus, genuine $G$-$\infty$-operads
	can be viewed as ``weak'' genuine $G$-operads,
	much as the relation between quasicategories
	and categories.
\end{remark}

\begin{remark}
	Since $\upsilon_{G,\**}$
	is fully faithful
	one has that 
	$X \in \mathsf{dSet}^G$
	is a $G$-$\infty$-operad 
	(cf. Definition \ref{GINNERFIB DEF})
	iff
	$\upsilon_{G,\**}X \in \mathsf{dSet}_G$
	is a genuine $G$-$\infty$-operad.
\end{remark}

We now turn to the task of describing
$\tau_G Z$ for $Z \in \mathsf{dSet}_G$
a genuine $G$-$\infty$-operad.

We start with some notation. 
Given a multiset $I$ of edges of a tree $T \in \Omega$
(formally, a function 
$I \colon \boldsymbol{E}(T) \to \mathbb{N}_0$),
we write $\sigma^I T \in \Omega$
for the tree obtained by degenerating $T$ once for each edge in $I$.
More explicitly, $\sigma^I T$ is the unique tree such that there is a planar degeneracy
$\pi \colon \sigma^I T \to T$
with $|\pi^{-1}(e)| = I(e) + 1$.
Moreover,
note that, if $T\in \Omega_G$ is a $G$-tree, 
then $\sigma^{I} T \in \Omega_{G}$
can be defined if $I$ is $G$-equivariant
(formally, this means that the multiset $I$ 
factors as
$I \colon \boldsymbol{E}(T) \to \boldsymbol{E}_G(T)
\to \mathbb{N}_0$;
i.e. if $I(e)=n$ then so is $I(ge)=n$ for all $g \in G$).

Our main interest will be in degeneracies of $G$-corollas. 
Up to isomorphism, 
a $G$-corolla $C \in \Sigma_G$ is determined by the number $0 \leq k$ of leaf orbits
and isotropy subgroups
$H_i \leq H_0 \leq G$ for $0 \leq i \leq k$,
where $H_0$ is the isotropy of a (chosen) root edge.
Pictorially, such a $G$-corolla has the orbital representation (cf. \eqref{GTREE_EQ})
given on the left below,
but in this section we will find it more convenient to label edge orbits using coset notation as on the right below,
so that $[e_i] = G e_i$ denotes the $G$-orbit of $e_i$.
\begin{equation}\label{GCOR EQ}
\begin{tikzpicture}
[grow=up,auto,level distance=2.3em,every node/.style = {font=\footnotesize},dummy/.style={circle,draw,inner sep=0pt,minimum size=1.75mm}]
	\node at (0,0) [font=\normalsize]{$C$}
		child{node [dummy] {}
			child{
			edge from parent node [swap,near end] {$G/H_k$} node [name=Kn] {}}
			child{
			edge from parent node [near end] {$G/H_1$}
node [name=Kone,swap] {}}
		edge from parent node [swap] {$G/H_0$}
		};
		\draw [dotted,thick] (Kone) -- (Kn) ;
	\node at (5,0) [font=\normalsize]{$C$}
		child{node [dummy] {}
			child{
			edge from parent node [swap,near end] {$[e_k]$} node [name=Kn] {}}
			child{
			edge from parent node [near end] {$[e_1]$}
node [name=Kone,swap] {}}
		edge from parent node [swap] {$[e_0]$}
		};
		\draw [dotted,thick] (Kone) -- (Kn) ;
\end{tikzpicture}
\end{equation}
We will then abbreviate $\sigma^i C = \sigma^{[e_i]} C$, and write $e_i$, $e_i'$ for the two edges of $\sigma^i C $ that degenerate the edge $e_i$ of $C$,
with $e_i$ denoting the inner edge and $e'_i$ the outer
edge.
\begin{equation}\label{DEGGCOR EQ}
\begin{tikzpicture}
[grow=up,auto,level distance=3em,
every node/.style = {font=\footnotesize},
dummy/.style={circle,draw,inner sep=0pt,minimum size=1.75mm}]
	\node at (0,0) [font=\normalsize]{$\sigma^0 C$}
		child{node [dummy] {}
			child{node [dummy] {}
				child{
				edge from parent node [swap,near end] {$[e_k]$} node [name=Kn] {}}
				child{
				edge from parent node [near end] {$[e_1]$}
node [name=Kone,swap] {}}
			edge from parent node [swap] {$[e_0]$}}
		edge from parent node [swap] {$[e'_0]$}
		};
		\draw [dotted,thick] (Kone) -- (Kn) ;
	\node at (5,0) [font=\normalsize]{$\sigma^i C$}
		child{node [dummy] {}
			child{
			edge from parent node [swap,near end] {$[e_k]$} node [near start,inner sep=1pt,name=Kn] {}}
			child[level distance=3.4em]{node [dummy] {}
				child[level distance=2.7em]{
				edge from parent node [swap] {$[e'_i]$}
}
			edge from parent node [near end,swap] {$[e_i]$}
node [near start,inner sep=1pt,name=Kone,swap] {}
node [near start,inner sep=1pt,name=Kone1] {}}
			child{
			edge from parent node [near end] {$[e_1]$}
node [swap] {}
node [near start,inner sep=1pt,name=Kn1,swap]{}}
		edge from parent node [swap] {$[e_0]$}
		};
		\draw [dotted,thick] (Kone) -- (Kn) ;
		\draw [dotted,thick] (Kone1) -- (Kn1) ;
\end{tikzpicture}
\end{equation}
The $G$-tree $\sigma^i C$ then has an orbital inner face
$\sigma^i C - [e_i]$ obtained by removing $[e_i]$
as well as an orbital outer face obtained by removing $e'_i$,
which we denote $\sigma^i C - [e'_i]$.
Moreover, note that we have natural identifications
$C \simeq \sigma^i C - [e_i]$,
$C \simeq \sigma^i C - [e'_i]$.

In what follows, we will find it convenient to simplify notation by denoting maps $\upsilon_{G,\**}\Omega[T] \to Z$,
where $T \in \Omega_G$ and $Z \in \mathsf{dSet}_G$,
simply as $T \to Z$.

\begin{definition}\label{HOEQUIVS DEF}
	Let $Z \in \mathsf{dSet}_G$ be a genuine $G$-$\infty$-operad and $C$ a $G$-corolla with edge orbits
	$[e_0],\cdots,[e_k]$.
	Given two operations 
	$f,g\colon C \to Z$,
	we write $f \sim_i g$ if there exists a map
	$H \colon \sigma^i C \to Z$ such that
\begin{itemize}
\item $f$ equals the restriction $H|_{\sigma^i C-[e'_i]}$;
\item $g$ equals the restriction $H|_{\sigma^i C-[e_i]}$;
\item the restriction $H|_{\sigma^i [e_i]}$
is the degeneracy $\sigma^i [e_i] \to [e_i] \to C \to Z$.
\end{itemize}
\end{definition}

\begin{remark}\label{HOMOTBOUND REM}
	Note that, if $f \sim_i g$, then it must be that
	$f|_{\partial C} = g|_{\partial C}$.
\end{remark}

\begin{example}\label{EQUIVSIM EX}
	Let $G = \mathbb{Z}_{/2} = \{\pm 1\}$
	and consider the $G$-corolla with orbital and expanded representations as given on the left below.
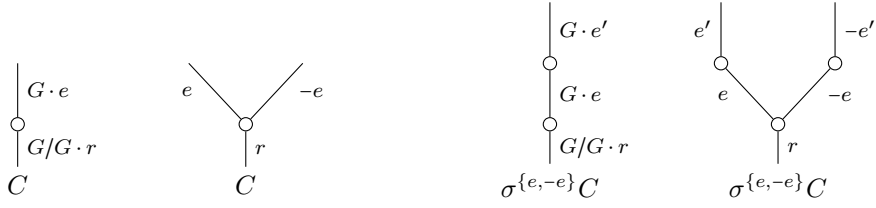
\[
\begin{tikzpicture}
[grow=up,auto,level distance=2.3em,every node/.style = {font=\footnotesize},dummy/.style={circle,draw,inner sep=0pt,minimum size=1.75mm}]
	\node at (0,0) [font=\normalsize]{$C$}
		child{node [dummy] {}
			child{
			edge from parent node [swap] {$G \cdot e$}
node [name=Kone,swap] {}}
		edge from parent node [swap] {$G/G \cdot r$}
		};
	\node at (3,0) [font=\normalsize]{$C$}
		child{node [dummy] {}
			child{
			edge from parent node [swap,near end] {$-e$} node [name=Kn] {}}
			child{
			edge from parent node [near end] {$e$}
node [name=Kone,swap] {}}
		edge from parent node [swap] {$r$}
		};
	\node at (7,0) [font=\normalsize]{$\sigma^{\{e,-e\}} C$}
		child{node [dummy] {}
			child{node [dummy] {}
				child{
				edge from parent node [swap] {$G \cdot e'$}
node [swap] {}}
			edge from parent node [swap] {$G \cdot e$}
node [swap] {}}
		edge from parent node [swap] {$G/G \cdot r$}
		};
	\node at (10,0) [font=\normalsize]{$\sigma^{\{e,-e\}} C$}
		child{node [dummy] {}
			child{node [dummy] {}
				child{
				edge from parent node [swap] {$-e'$} node {}}
			edge from parent node [swap,near end] {$-e$} node {}}
			child{node [dummy] {}
				child{
				edge from parent node {$e'$}
node [swap] {}}
			edge from parent node [near end] {$e$}
node [swap] {}}
		edge from parent node [swap] {$r$}
		};
\end{tikzpicture}
\]
$C$ then has a single $G$-leaf orbit $[e] = G \cdot e$ so that,
for $f,g \colon C \to Z$, % parallel
one has $f \sim_1 g$ if there exists a dendrex
$H \colon \sigma^{\{e,-e\}}C \to Z$
such that 
\begin{equation}\label{EQUIVHOMOT EQ}
	f = H|_{\sigma^{\{e,-e\}}C - \{e',-e'\}}
\qquad
	g = H|_{\sigma^{\{e,-e\}}C - \{e,-e\}}
\qquad
	H_{\sigma^e e}, H|_{\sigma^{-e}-e} \text{ are degenerate}.
\end{equation}
It is worthwhile to compare this equivariant relation 
$f \sim_1 g$
with the relations obtained if one forgets the $G$-actions. Indeed, while \eqref{EQUIVHOMOT EQ} implicitly assumes that all of $f,g,H$ are $G$-equivariant,
by omitting that assumption one can reinterpret 
\eqref{EQUIVHOMOT EQ}
as defining a relation
$f \sim_{[e]} g$ between not necessarily $G$-equivariant maps $f,g \colon C \to Z$.

A priori, the $\sim_{[e]}$ relation differs from the 
non-equivariant 
$\sim_{e}$ and $\sim_{-e}$
relations obtained forgetting the $G$-actions
and regarding $C$ as a non-equivariant corolla.
However, for $f,g,H$ as in \eqref{EQUIVHOMOT EQ} one has
\begin{equation}\label{EQUIVSIM EQ}
f = H|_{\sigma^{\{e,-e\}}C - \{e',-e'\}}
\sim_e H|_{\sigma^{\{e,-e\}}C - \{e,-e'\}}
\sim_{-e} H|_{\sigma^{\{e,-e\}}C - \{e,-e\}} =g
\end{equation}
so that, by Lemma \ref{EQUIVI LEM}(b) below,
one has that $f \sim_{[e]} g$ in fact implies $f \sim_{e} g$. 
Moreover, the converse statement follows immediately by using degeneracies.

More generally, similar considerations show that the $\sim$ relations are compatible with restricting the $G$-actions
to $H$-actions for some $H \leq G$.
\end{example}

Our next task is to establish the key properties
of the $\sim_i$ relations.
This will adapt a slew of arguments in
\cite[\S 6]{MW09}
that make use of lifting properties against horn inclusions.
However, as $\sim_i$ is motivated by the
orbital representations \eqref{GCOR EQ},\eqref{DEGGCOR EQ},
we will need a notion of horn that is likewise motivated by
such representations.
This is the notion of \emph{orbital horn},
introduced in \cite[\S 2.3]{BP20},
and defined as follows.
Letting $T \in \Omega_G$ be a $G$-tree,
$E \subseteq \boldsymbol{E}^{\mathsf{i}}(T)$
be a $G$-subset,
and writing
$\mathsf{Face}_o(T)$
for the poset of planar orbital faces, 
the \emph{orbital $G$-inner horn} 
$\Lambda_o^E[T] \in \mathsf{dSet}^G$ is given by
\begin{equation}
        \label{ORBITALHORN_EQ}
        \Lambda_o^E[T]
        = \mathop{\colim}\limits_{S \in \mathsf{Face}_o(T), (T - E) \not\into S} \Omega[S]
        = \bigcup_{S \in \mathsf{Face}_o(T), (T - E) \not\into S} \Omega[S].
\end{equation}
We caution that this differs from 
the $\Lambda^E[T]$ horns in 
\eqref{GINNERHORN_EQ} and Definition \ref{GENINFOP DEF},
as in general one has
$\Lambda_o^E[T] \subsetneq \Lambda^E[T]$
(see \cite[Ex. 2.34]{BP20} for a detailed example).
Nonetheless, as orbital $G$-inner horn inclusions 
$\Lambda_o^E[T] \to \Omega[T]$
are $G$-inner anodyne \cite[Prop. 3.13]{BP20},
the genuine $G$-$\infty$-operads in
Definition \ref{GENINFOP DEF}
also have the right lifting property against the maps
$\upsilon_{G,\**}(\Lambda_o^E[T] \to \Omega[T])$.

\begin{example}
	For the trees $\sigma^i C$ in \eqref{DEGGCOR EQ},
	the orbital inner horn 
	$\Lambda_o^{[e_i]}[\sigma^i C]$
	is the union of the orbital faces 
	$\sigma^i C - [e'_i]$ and 
	$\sigma^i[e_i]$
	(as expected by treating the orbital
	picture as a usual tree).
\end{example}

\begin{lemma}[{cf. \cite[Prop. 6.3 and Lemma 6.4]{MW09}}]
	\label{EQUIVI LEM}
	Let $Z \in \mathsf{dSet}_G$ be a genuine $G$-$\infty$-operad and $C$ a $G$-corolla with edge orbits
	$[e_0],\cdots,[e_k]$. Then:
\begin{itemize}
	\item[(a)] each of the relations $\sim_i$ 
	in Definition \ref{HOEQUIVS DEF}
	is an equivalence relation;
	\item[(b)] all the equivalence relations $\sim_i$ coincide.
\end{itemize}
\end{lemma}

\begin{proof}
	We first address (a). 
	
	For reflexivity $f \sim_i f$,
	we take the exhibiting homotopy 
	$H$ to be the degeneracy
	$\sigma^i C \xrightarrow{\sigma^i} C \xrightarrow{f} Z$.
	
	Both symmetry and transitivity will use the 
	tree $\sigma^{ii} C$ below, which degenerates $[e_i]$ twice.
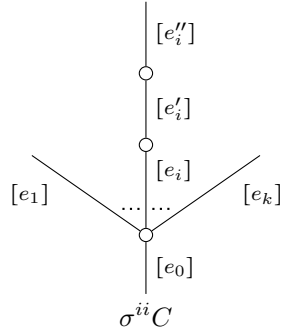
\[
\begin{tikzpicture}
[grow=up,auto,level distance=3em,
every node/.style = {font=\footnotesize},
dummy/.style={circle,draw,inner sep=0pt,minimum size=1.75mm}]
	\node at (0,0) [font=\normalsize]{$\sigma^{ii} C$}
		child{node [dummy] {}
			child{
			edge from parent node [swap,near end] {$[e_k]$} node [near start,inner sep=1pt,name=Kn] {}}
			child[level distance=3.4em]{node [dummy] {}
				child[level distance=2.7em]{node [dummy] {}
					child[level distance=2.7em]{
					edge from parent node [swap] {$[e''_i]$}
}
				edge from parent node [swap] {$[e'_i]$}
}
			edge from parent node [near end,swap] {$[e_i]$}
node [near start,inner sep=1pt,name=Kone,swap] {}
node [near start,inner sep=1pt,name=Kone1] {}}
			child{
			edge from parent node [near end] {$[e_1]$}
node [swap] {}
node [near start,inner sep=1pt,name=Kn1,swap]{}}
		edge from parent node [swap] {$[e_0]$}
		};
		\draw [dotted,thick] (Kone) -- (Kn) ;
		\draw [dotted,thick] (Kone1) -- (Kn1) ;
\end{tikzpicture}
\]
For symmetry, suppose $f \sim_i g$ with 
$H \colon \sigma^{i} C \to Z$ the exhibiting homotopy.
Define a map 
$\bar{H} \colon \Lambda^{[e_i]}_o[\sigma^{ii} C] \to Z$ via
\[
	\bar{H}|_{\sigma^{ii}C - [e''_i]} = H,
		\qquad
	\bar{H}|_{\sigma^{ii}C - [e'_i]} = f \circ \sigma^i,
		\qquad
	\bar{H}|_{\sigma^{ii} [e_i]} = 
	f|_{[e_i]} \circ \sigma^{ii} =
	g|_{[e_i]} \circ \sigma^{ii}.
\]
Since the orbital inner horn inclusion
$\bar{H} \colon \Lambda^{[e_i]}_o[\sigma^{ii} C] \to \Omega[C]$
is $G$-inner anodyne by \cite[Prop. 3.13]{BP20},
$\bar{H}$ admits an extension $\widetilde{H} \colon \sigma^{ii}C \to Z$.
The restriction $\widetilde{H}|_{\sigma^{ii}C - [e_i]}$ then provides the homotopy exhibiting $g \sim_i f$, and symmetry of $\sim_i$ follows.

Next, suppose $f \sim_i g$ and $g \sim_i h$, and let 
$H \colon \sigma^{i} C \to Z$,
$K \colon \sigma^{i} C \to Z$ be the exhibiting homotopies.
Define a map 
$\bar{H} \colon \Lambda^{[e'_i]}_o[\sigma^{ii} C] \to Z$ by
\[
	\bar{H}|_{\sigma^{ii}C - [e''_i]} = H,
		\qquad
	\bar{H}|_{\sigma^{ii}C - [e_i]} = K,
		\qquad
	\bar{H}|_{\sigma^{ii} [e_i]} = 
	f|_{[e_i]} \circ \sigma^{ii} =
	g|_{[e_i]} \circ \sigma^{ii} =
	h|_{[e_i]} \circ \sigma^{ii}.
\]
$\bar{H}$ again admits an extension $\widetilde{H} \colon \sigma^{ii}C \to Z$, 
and the restriction $\widetilde{H}|_{\sigma^{ii}C - [e'_i]}$
provides the homotopy exhibiting $f \sim_i g$, and transitivity of $\sim_i$ follows.

We next turn to (b). Consider the tree $\sigma^{ij} C$ which degenerates $C$ once along each of $[e_i]$ and $[e_j]$.
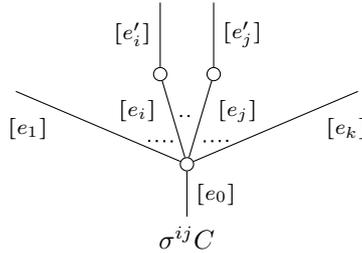
\[
\begin{tikzpicture}
[grow=up,auto,level distance=2.75em,
every node/.style = {font=\footnotesize},
dummy/.style={circle,draw,inner sep=0pt,minimum size=1.75mm}]
	\node at (0,0) [font=\normalsize]{$\sigma^{ij} C$}
		child{node [dummy] {}
			child{
			edge from parent node [swap,near end] {$[e_k]$} node [near start,inner sep=1pt,name=Kn] {}}
			child[level distance=3.4em,sibling distance=2em]{node [dummy] {}
				child[level distance=2.7em]{
				edge from parent node [swap] {$[e'_j]$}
}
			edge from parent node [very near end,swap] {$[e_j]$}
node [near start,inner sep=1pt,name=Kone,swap] {}
node [inner sep=1pt,name=Kn2] {}}
			child[level distance=3.4em,sibling distance=2em]{node [dummy] {}
				child[level distance=2.7em]{
				edge from parent node {$[e'_i]$}
}
			edge from parent node [very near end] {$[e_i]$}
node [inner sep=1pt,name=Kone2,swap] {}
node [near start,inner sep=1pt,name=Kone1] {}}
			child{
			edge from parent node [near end] {$[e_1]$}
node [swap] {}
node [near start,inner sep=1pt,name=Kn1,swap]{}}
		edge from parent node [swap] {$[e_0]$}
		};
		\draw [dotted,thick] (Kn) -- (Kone) ;
		\draw [dotted,thick] (Kone1) -- (Kn1) ;
		\draw [dotted,thick] (Kone2) -- (Kn2) ;
\end{tikzpicture}
\]
Suppose $f \sim_i g$ with $H \colon \sigma^{i} C \to Z$ the associated homotopy.
Define a map 
$\bar{H} \colon \Lambda^{[e_i]}_o[\sigma^{ij} C] \to Z$ by
\[
	\bar{H}|_{\sigma^{ij}C - [e'_j]} = H,
		\qquad
	\bar{H}|_{\sigma^{ij}C - [e_j]} = f \circ \sigma^i,
		\qquad
	\bar{H}|_{\sigma^{ij}C - [e'_i]} = f \circ \sigma^j.
\]
Yet again, $\bar{H}$ admits an extension $\widetilde{H} \colon \sigma^{ij}C \to Z$, and the restriction $\widetilde{H}|_{\sigma^{ij}C - [e_i]}$
provides a homotopy exhibiting $g \sim_j f$. (b) now follows.
\end{proof}

In light of Lemma \ref{EQUIVI LEM},
given operations $f,g \colon C \to Z$ with  % parallel
$C$ a $G$-corolla and $Z$ a genuine $G$-$\infty$-operad,
we will henceforth write $f \sim g$ whenever $f \sim_i g$ for some (and thus all) $i$.
We now extend the $\sim$ relation,
allowing us to define
$\mathop{\mathrm{ho}}(Z)$
for $Z$ a genuine $G$-$\infty$-operad.

\begin{definition}\label{XTENDSIM DEF}
	Let $T \in \Omega_G$ be a $G$-tree
	and $Z \in \mathsf{dSet}_G$ be a 
	genuine $G$-$\infty$-operad.
	
	Given dendrices $x,y\colon T \to Z$ we write
	$x \sim y$ if there are equivalences of restrictions
	$x|_{T_v} \sim y|_{T_v}$ for all $G$-vertices
	$v \in \boldsymbol{V}_G(T)$.
	
	Further, we define $\mathop{\mathrm{ho}}(Z)(T) = Z(T)/\sim$.
\end{definition}

\begin{proposition}
	Let $Z \in \mathsf{dSet}_G$ be a genuine $G$-$\infty$-operad. Then the assignment 
	$T \mapsto \mathop{\mathrm{ho}}(Z)(T)$
	is a contravariant functor on $T \in \Omega_G$, i.e.
	$\mathop{\mathrm{ho}}(Z) \in \mathsf{dSet}_G$.
\end{proposition}

\begin{proof}
	It suffices to show that the $\sim$ equivalence relations are compatible with the generating classes of maps in $\Omega_G$, namely 
	(cf. Proposition \ref{TREEFACT_PROP})
	(planar) degeneracies, inner faces, outer faces, and quotient maps.
	
	The cases of degeneracies and outer faces are obvious. 
	For quotients, 
	since any quotient $\bar{T} \to T$ of $G$-trees induces quotients on $G$-vertices, it suffices to consider the case of a quotient
	$\bar{C} \xrightarrow{\pi} C$ of $G$-corollas.
	But it is straightforward to check that a homotopy exhibiting $f \sim_0 g$ also induces a homotopy exhibiting 
	$f \circ \pi \sim_0 g \circ \pi$
	(notably, this needs not hold for the relations $f \sim_i g$ for $0<i$, 
	since then the exhibiting homotopy 
	may instead exhibit a string of relations 
	$f \circ \pi \sim \cdots \sim g \circ \pi$
	as in \eqref{EQUIVSIM EQ},
	due to quotient maps possibly sending two leaf orbits to the same leaf orbit).

It remains to address the most interesting case,
that of inner faces. Since inner faces can be factored as composites of inner faces that each collapse a singe inner edge orbit,
it suffices to consider the case of faces $D \to T$,
where $T$ has a single inner edge orbit.
That is, we can assume that there are $G$-corollas
$C_1$, $C_2$ such that 
$T = C_1 \amalg_{[e_i]} C_2$ and
$D = T - [e_i]$, as illustrated below.
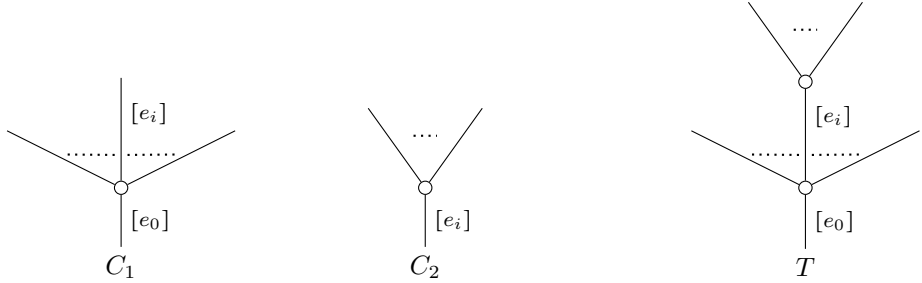
\[
\begin{tikzpicture}
[grow=up,auto,level distance=3em,
every node/.style = {font=\footnotesize},
dummy/.style={circle,draw,inner sep=0pt,minimum size=1.75mm}]
	\node at (0,0) [font=\normalsize]{$C_1$}
		child{node [dummy] {}
			child[level distance=2.15em]{
			edge from parent node [swap,near end] {} node [inner sep=1pt,name=Kn] {}}
			child[level distance=4.5em]{node {}
			edge from parent node [pos=0.65,swap] {$[e_i]$}
node [near start,inner sep=1pt,name=Kone,swap] {}
node [near start,inner sep=1pt,name=Kone1] {}}
			child[level distance=2.15em]{
			edge from parent node [near end] {}
node [swap] {}
node [inner sep=1pt,name=Kn1,swap]{}}
		edge from parent node [swap] {$[e_0]$}
		};
		\draw [dotted,thick] (Kone) -- (Kn) ;
		\draw [dotted,thick] (Kone1) -- (Kn1) ;
	\node at (4,0) [font=\normalsize]{$C_2$}
		child{node [dummy] {}
			child{
			edge from parent node [swap,near end] {} node [name=Kn] {}}
			child{
			edge from parent node [near end] {}
node [name=Kone,swap] {}}
		edge from parent node [swap] {$[e_i]$}
		};
		\draw [dotted,thick] (Kone) -- (Kn) ;
	\node at (9,0) [font=\normalsize]{$T$}
		child{node [dummy] {}
			child[level distance=2.15em]{
			edge from parent node [swap,near end] {} node [inner sep=1pt,name=Kn] {}}
			child[level distance=4em]{node [dummy] {}
				child[level distance=3em]{
				edge from parent node [swap,near end] {} node [name=Kn2] {}}
				child[level distance=3em]{
				edge from parent node [near end] {}
node [name=Kone2,swap] {}}
			edge from parent node [pos=0.7,swap] {$[e_i]$}
node [pos=0.28,inner sep=1pt,name=Kone,swap] {}
node [pos=0.28,inner sep=1pt,name=Kone1] {}}
			child[level distance=2.15em]{
			edge from parent node [near end] {}
node [swap] {}
node [inner sep=1pt,name=Kn1,swap]{}}
		edge from parent node [swap] {$[e_0]$}
		};
		\draw [dotted,thick] (Kone) -- (Kn) ;
		\draw [dotted,thick] (Kone1) -- (Kn1) ;
		\draw [dotted,thick] (Kone2) -- (Kn2) ;
\end{tikzpicture}
\]
The claim is now that,
if $x,y \colon T \to Z$ are such that
$x|_{C_1} \sim y|_{C_1}$ and
$x|_{C_2} \sim y|_{C_2}$,
then one must have 
$x|_{D} \sim y|_{D}$.
This will follow from the following two claims:
\begin{itemize}
\item[(i)] if $x,y \colon T \to Z$ are such that
$x|_{C_1} = y|_{C_1}$ and
$x|_{C_2} = y|_{C_2}$
then $x|_{D} \sim y|_{D}$;
\item[(ii)]
given $x \colon T \to Z$, $f\colon C_1 \to Z$ and
$g \colon C_2 \to Z$ such that
$f \sim x|_{C_1}$, $g \sim x|_{C_2}$,
there exists
$y \colon T \to Z$ such that
$y|_{C_1} = f$, $y|_{C_2} = g$ and
$y|_D = x|_D$.
\end{itemize}
To show (i) and (ii), consider the degeneracies
$\sigma^0 T$ and $\sigma^i T$ pictured below.
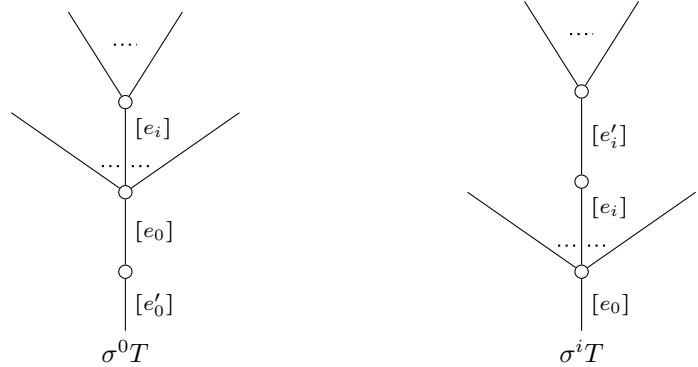
\[
\begin{tikzpicture}
[grow=up,auto,level distance=3em,
every node/.style = {font=\footnotesize},
dummy/.style={circle,draw,inner sep=0pt,minimum size=1.75mm}]
	\node at (0,0) [font=\normalsize]{$\sigma^0 T$}
		child{node [dummy] {}
		child{node [dummy] {}
			child{
			edge from parent node [swap,near end] {} node [near start,inner sep=1pt,name=Kn] {}}
			child[level distance=3.4em]{node [dummy] {}
				child{
				edge from parent node [swap,near end] {} node [name=Kn2] {}}
				child{
				edge from parent node [near end] {}
node [name=Kone2,swap] {}}
			edge from parent node [near end,swap] {$[e_i]$}
node [near start,inner sep=1pt,name=Kone,swap] {}
node [near start,inner sep=1pt,name=Kone1] {}}
			child{
			edge from parent node [near end] {}
node [swap] {}
node [near start,inner sep=1pt,name=Kn1,swap]{}}
		edge from parent node [swap] {$[e_0]$}}
		edge from parent node [swap] {$[e'_0]$}
		};
		\draw [dotted,thick] (Kone) -- (Kn) ;
		\draw [dotted,thick] (Kone1) -- (Kn1) ;
		\draw [dotted,thick] (Kone2) -- (Kn2) ;
	\node at (6,0) [font=\normalsize]{$\sigma^i T$}
		child{node [dummy] {}
			child{
			edge from parent node [swap,near end] {} node [near start,inner sep=1pt,name=Kn] {}}
			child[level distance=3.4em]{node [dummy] {}
			child{node [dummy] {}
				child{
				edge from parent node [swap,near end] {} node [name=Kn2] {}}
				child{
				edge from parent node [near end] {}
node [name=Kone2,swap] {}}
			edge from parent node [swap] {$[e'_i]$}}
			edge from parent node [near end,swap] {$[e_i]$}
node [near start,inner sep=1pt,name=Kone,swap] {}
node [near start,inner sep=1pt,name=Kone1] {}}
			child{
			edge from parent node [near end] {}
node [swap] {}
node [near start,inner sep=1pt,name=Kn1,swap]{}}
		edge from parent node [swap] {$[e_0]$}
		};
		\draw [dotted,thick] (Kone) -- (Kn) ;
		\draw [dotted,thick] (Kone1) -- (Kn1) ;
		\draw [dotted,thick] (Kone2) -- (Kn2) ;
\end{tikzpicture}
\]
Given $x,y$ as in (i), one can now build a map
$H \colon \Lambda_o^{[e_i]}[\sigma^0 T] \to Z$ by
\[
	H|_{\sigma^0 T - [e_0]} = x,
\qquad
	H|_{\sigma^0 T - [e'_0]} = y,
\qquad
	H|_{\sigma^0 C_1} = 
	x|_{C_1} \circ \sigma^0 = 
	y|_{C_1} \circ \sigma^0.
\]
Letting $\widetilde{H}\colon \sigma^0 T \to Z$
be an extension of $H$,
the restriction $H|_{\sigma^0 T - [e_i]}$
provides the desired homotopy 
$x|_{D} \sim y|_{D}$, showing (i).

Lastly, let $x,f,g$ be as in (ii), 
and let
$K \colon \sigma^i C_1 \to Z$ exhibit the relation
$f \sim_i x|_{C_1}$
and 
$ \bar{K} \colon \sigma^i C_2 \to Z$
exhibit the relation
$x|_{C_2} \sim_i g$ (note the reversed order).
Now build the map
$H \colon \Lambda_o^{[e'_i]}[\sigma^i T] \to Z$ by
\[
	H|_{\sigma^i T - [e_i]} = x,
\qquad
	H|_{\sigma^i C_1} = K,
\qquad
	H|_{\sigma^i C_2} = \bar{K}.
\]
Again letting 
$\widetilde{H} \colon \sigma^i T \to Z$ be a lift,
the restriction 
$\widetilde{H}|_{\sigma^i T - [e'_i]}$
provides the desired $y \colon T \to Z$ in (ii),
finishing the proof.
\end{proof}

\begin{corollary}\label{HOOPUNIV COR}
Let $Z \in \mathsf{dSet}_G$ be a genuine $G$-$\infty$-operad. Then:
	\begin{itemize}
	\item[(a)] $\mathop{\mathrm{ho}}(Z) \in \mathsf{dSet}_G$ is a genuine equivariant operad
                (Definition \ref{OPG_DEF}).
                % \cite[Def. 3.35]{BP20},
                % i.e. it satisfies the strict right lifting condition against the Segal core inclusions
                % $\upsilon_{G,\**}\left(Sc[T] \to \Omega[T]\right)$ for $T \in \Omega_G$;
	\item[(b)] the quotient map
	$Z \to \mathop{\mathrm{ho}}(Z)$ is the universal map from $Z$ to a genuine equivariant operad.
	\end{itemize}
In particular, (a) and (b) yield a natural identification
$
\mathsf{ho}(Z)
\simeq
\tau_G(Z)
$.
\end{corollary}

\begin{proof}
	Note first that, 
	since by Remark \ref{HOMOTBOUND REM}
	the relation $\sim$ preserves edge colorings,
	Definition \ref{XTENDSIM DEF} implies that
	any map 
	$\upsilon_{G,\**}Sc[T] \to \mathop{\mathrm{ho}}(Z)$
	admits a factorization 
	$\upsilon_{G,\**}Sc[T] \to Z \xrightarrow{q} \mathop{\mathrm{ho}}(Z)$.
	 
	The right lifting property for $\tau_G(Z)$
	against the maps 
	$\upsilon_{G,\**}(Sc[T] \to \Omega[T])$
	is then automatic from the lifting property for $Z$.

	For strictness,	
	note that Definition \ref{XTENDSIM DEF}
	can be reinterpreted as saying that
	a pair of dendrices $\Omega[T] \rightrightarrows Z$
	give rise to the same point of 
	$\mathop{\mathrm{ho}}(Z)$, i.e. 
	the composites 
	$\Omega[T] \rightrightarrows Z \xrightarrow{q}
	\mathop{\mathrm{ho}}(Z)$ coincide, 
	iff the composites 
	$Sc[T] \to \Omega[T] \rightrightarrows Z \xrightarrow{q}
	\mathop{\mathrm{ho}}(Z)$ coincide, showing strictness, and establishing (a).
		
	For (b), since $\mathop{\mathrm{ho}}(Z)$ is a quotient of
	$Z$, it suffices to show that any map
	of the form $F \colon Z \to Y$ with $Y$ a genuine equivariant operad must also enforce the $\sim$ relation.
	For a $G$-corolla $C$ and
	$f,g\colon C \rightrightarrows Z$ such that 
	$H \colon \sigma^i C \to Z$ exhibits
	$f \sim_i g$, 
	the strict lifting condition for $Y$
	shows that the maps
	$F\circ H \colon \sigma^i C \to Y$,
	$F(f) \circ \sigma^i \colon \sigma^i C \to Y$
	must coincide, and hence $F(f)=F(g)$.
	The claim that $F$ respects equivalences
	of general pairs of dendrices $T \rightrightarrows Z$
	is now clear from Definition \ref{XTENDSIM DEF}.
\end{proof}

The following is the equivariant analogue of \cite[Prop. 4.8]{CM13b},
as discussed in Remark \ref{TWOHOMOP REM}.

\begin{proposition}\label{HOOPID_PROP}
Let $\mathcal{O} \in \mathsf{sOp}^G$
be a fibrant operad. 
Then there is a natural isomorphism of genuine equivariant operads
\begin{equation}\label{HOOPID EQ}
\tau_G(hcN(\mathcal{O})) \xrightarrow{\simeq}
\pi_0 \left( \upsilon_{G,\**} N\mathcal{O} \right).
\end{equation}
\end{proposition}

\begin{proof}
	To ease notation, we abbreviate
	$\upsilon_{G,\**}$ as $\upsilon_{\**}$
	throughout the proof.

      By \cite[Prop. 5.9]{BP20}, $\pi_0(\upsilon_{\**}N\O)$ is a genuine equivariant operad,
      and the existence of the map in \eqref{HOOPID EQ}
      will be an application of
      Corollary \ref{HOOPUNIV COR}(b).

Firstly,
note that we have the following identifications,
naturally on $T \in \Omega_G$.
\[
\upsilon_{\**}hcN(\O)(T)
\simeq
\mathsf{sOp}^G(W_!\Omega[T],\O)
\simeq
\mathsf{sdSet}^G(N W_!\Omega[T],N \O)
\simeq 
\mathsf{sdSet}_G(\upsilon_{\**}N W_!\Omega[T],\upsilon_{\**}N \O)
\]
where the second and third identifications use the fact that 
$N\colon \mathsf{Op} \to \mathsf{dSet}$ and $\upsilon_{\**} \colon \dSet^G \to \dSet_G$
are fully faithful inclusions. 
One now has a map
\begin{align*}
  \mathsf{sdSet}_G(\upsilon_{\**}N W_!\Omega[T],\upsilon_{\**}N \O)
  \to \phantom{|} &
    \mathsf{sdSet}_G(\upsilon_{\**}N W_!\Omega[T],\pi_0\upsilon_{\**}N \O)
  \\ & \simeq
       \mathsf{dSet}_G(\pi_0 \upsilon_{\**}  N W_!\Omega[T],\pi_0\upsilon_{\**}N \O)
  \\ & \simeq
       \mathsf{dSet}_G(\upsilon_{\**}\Omega[T],\pi_0\upsilon_{\**}N \O)
  \\ & =
       (\pi_0\upsilon_{\**}N \O)(T)
\end{align*}
so altogether we obtain a map
$\upsilon_{\**}hcN(\O) \to \pi_0 \upsilon_{\**} N \O$
and hence, by Corollary \ref{HOOPUNIV COR},
the desired map 
\[
	\tau_G(hcN(\O)) \to \pi_0 \upsilon_{\**} N \O.
\]
Moreover, both of these are quotients of $\upsilon_{\**}hcN(\O)$,
so to prove that this map is an isomorphism one needs only show that any two
% parallel
operations $f,g \colon C \to hcN \O$ of $\upsilon_{\**}hcN(\O)$
that are identified in 
$\pi_0 \upsilon_{\**} N \O$
were already identified in 
$\mathop{\mathrm{ho}}(hcN(\O))$.
But this now follows from the pushout below,
cf. Lemma \ref{WLEFTQPUSH LEM},
\[
\begin{tikzcd}
	\Omega(C) \otimes_{\mathfrak{C}_{\bullet}}
	\partial \Delta[1]
	\ar{r} \ar{d}
&
	W_! \left(\partial \Omega[\sigma^0 C]\right) 
	\ar{d}
\\
	\Omega(C) \otimes_{\mathfrak{C}_{\bullet}}
	\Delta[1]
	\ar{r}
&
	W(\sigma^0 C)
\end{tikzcd}
\]
since mapping out of the bottom left (resp. right) corner
encodes homotopy in
$\pi_0 \upsilon_{\**} N \O$
(resp. $\mathop{\mathrm{ho}}(hcN(\O))$).
\end{proof}

% Bibliography
\providecommand{\bysame}{\leavevmode\hbox to3em{\hrulefill}\thinspace}
\providecommand{\MR}{\relax\ifhmode\unskip\space\fi MR }
\providecommand{\MRhref}[2]{%
  \href{http://www.ams.org/mathscinet-getitem?mr=#1}{#2}
}
\providecommand{\doi}[1]{%
  doi:\href{https://dx.doi.org/#1}{#1}}
\providecommand{\arxiv}[1]{%
  arXiv:\href{https://arxiv.org/abs/#1}{#1}}
\providecommand{\href}[2]{#2}

\makeatletter\@input{labels-TAS.tex}\makeatother

\end{document}